\documentclass[11pt, A4paper, leqno]{amsart}
\pdfoutput=1

\usepackage{graphicx}
\usepackage{amssymb,amsmath}
\usepackage{epstopdf}
\usepackage{sfmath}
\usepackage{url}

\usepackage[small, nohug, heads-vee]{diagrams} 
\diagramstyle[labelstyle=\scriptstyle]

\usepackage{xr}

\usepackage{setspace}

\newtheorem{theorem}{Theorem}[section]
\newtheorem{lemma}[theorem]{Lemma}
\newtheorem{corollary}[theorem]{Corollary}
\newtheorem{proposition}[theorem]{Proposition}

\theoremstyle{remark}
\newtheorem{remark}[theorem]{Remark}

\theoremstyle{definition}
\newtheorem{definition}[theorem]{Definition}

\newcommand\bR{{\mathbb{R}}}

\newcommand\bZ{{\mathbb Z}}

\newcommand\bV{{\mathbb V}}

\newcommand\bA{{\mathbb{A}}}

\newcommand\Hom{{\rm Hom}}
\newcommand\dev{{\bf dev}}
\newcommand\SI{{\mathbb{S}}}

\newcommand\Bd{{\rm bd}}
\newcommand\clo{{\rm Cl}}
\newcommand\bdd{{\mathbf{d}}}

\newcommand\ra{\rightarrow}

\newcommand\emp{\emptyset}
\newcommand\eps{\epsilon}

\newcommand\Aff{{\mathbf{Aff}}}
\newcommand\ovl{\overline}
\newcommand\Aut{{\mathbf{Aut}}}
\newcommand\Idd{{\rm I}}

\newcommand\bv{{\mathbf{v}}}
\newcommand\bu{{\mathbf{u}}}

\newcommand\Ag{{\mathrm{Ag}}}

\newcommand\tri{\triangle}

\newcommand\rpn{\mathbb{RP}^n}

\newcommand\SL{{\mathsf{SL}}}

\newcommand\SO{{\mathsf{SO}}}

\newcommand\PGL{{\mathsf{PGL}}}
\newcommand\SLnp{{\mathsf{SL}}_\pm(n+1, \bR)}
\newcommand\SLn{{\mathsf{SL}}_\pm(n, \bR)}

\newcommand\GL{{\mathsf{GL}}}

\newcommand\GLnp{{\mathsf{GL}}(n+1, \bR)}
\newcommand\PGLnp{{\mathsf{PGL}}(n+1, \bR)}

\newcommand\orb{\mathcal{O}} 
\newcommand\torb{\tilde{\mathcal{O}}}
\newcommand\bGamma{{\boldsymbol \Gamma}}
\newcommand\leng{{\mathrm{length}}}

\newcommand\cwl{{\mathrm{cwl}}}

\newcommand\rlp{{\rm{(}}}
\newcommand\rrp{{\rm{)}}}

\newcommand\SLpm{{\mathrm{SL}}_{\pm}(n+1, \bR)}

\begin{document}

\title[Ends of real projective orbifolds]{A classification of ends of properly convex real projective orbifolds II: 
Properly convex radial ends and totally geodesic ends.}


\author{Suhyoung Choi} 
\address{ Department of Mathematics \\ KAIST \\
Daejeon 305-701, South Korea 
}
\email{schoi@math.kaist.ac.kr}


\date{\today}



\begin{abstract} 
Real projective structures on $n$-orbifolds are useful in understanding the space of 
representations of discrete groups into $\SL(n+1, \bR)$ or $\PGL(n+1, \bR)$. A recent work shows that many hyperbolic manifolds 
deform to manifolds with such structures not projectively equivalent to the original ones. 
The purpose of this paper is to understand 
the structures of properly convex ends of real projective $n$-dimensional orbifolds. 
In particular, these have the radial or totally geodesic ends. 
For this, we will study the natural conditions on eigenvalues of holonomy 
representations of ends when these ends are manageably understandable. 
In this paper, we only study the properly convex ends. 
The main techniques are the Vinberg duality and 
a generalization of 
the work of Goldman, Labourie, and Margulis on flat Lorentzian $3$-manifolds. 
Finally, we show that a noncompact strongly tame properly convex real projective orbifold with
generalized lens-type or horospherical 
ends satisfying some topological conditions always has a strongly irreducible holonomy group. 


\end{abstract}

%


\thanks{This work was supported by the National Research Foundation
of Korea (NRF) grant funded by the Korea government (MEST) (No.2010-0027001).} 






\maketitle

\tableofcontents






\section{Introduction} 



\subsection{Preliminary definitions.} 
\subsubsection{Topology of orbifolds and their ends.}
An {\em orbifold} $\orb$ is a topological space with charts modeling open sets by quotients of 
Euclidean open sets or half-open sets  by finite group actions and compatible patching maps with one another. 
The boundary $\partial \orb$ of an orbifold is defined as the set of points with only half-open sets as models. 
Let $\orb$ denote an $n$-dimensional orbifold with finitely many ends 
where end-neighborhoods are diffeomorphic to closed  $(n-1)$-dimensional orbifolds times an open interval.  We will require that $\orb$ is {\em strongly tame}; that is, $\orb$ has a compact suborbifold $K$ 
so that $\orb - K$ is a disjoint union of end-neighborhoods diffeomorphic to 
closed $(n-1)$-dimensional orbifolds multiplied by open intervals.
Hence $\partial \orb$ is a compact suborbifold.  



By strong tameness, $\orb$ has only finitely many ends $E_1, \ldots, E_m$,
and each end has an end-neighborhood diffeomorphic to $\Sigma_{E_i} \times (0, 1)$.
Let $\Sigma_{E_i}$ here denote the compact orbifold diffeomorphism type of the end $E_i$.
Such end-neighborhoods of these types are said to be of the {\em product types}. 
A system of end-neighborhoods for an end $E$ gives us a sequence of open sets in 
the universal $\torb$ cover of $\orb$. This system gives us a {\em pseudo-end neighborhood system}
and a {\em pseudo-end}. The subgroup $\bGamma_{\tilde E}$ acting on such a system for 
a pseudo-end $\tilde E$ is called a {\em pseudo-end fundamental group}.

\subsubsection{Real projective structures on orbifolds.} 
We will consider an orbifold $\orb$ with a real projective structure: 
This can be expressed as 
\begin{itemize}
\item having a pair $(\dev, h)$ where 
$\dev:\torb \ra \rpn$ is an immersion equivariant with respect to 
\item the homomorphism $h: \pi_1(\orb) \ra \PGLnp$ where 
$\torb$ is the universal cover and $\pi_1(\orb)$ is the group of deck transformations acting on $\torb$. 
\end{itemize}
$(\dev, h)$ is only determined up to an action of $\PGLnp$ 
given by 
\[ g \circ (\dev, h(\cdot)) = (g \circ \dev, g h(\cdot) g^{-1}) \hbox{ for } g \in \PGLnp. \]
We will use only one pair where $\dev$ is an embedding for this paper and hence 
identify $\torb$ with its image. 
A {\em holonomy} is an image of an element under $h$. 
The {\em holonomy group} is the image group $h(\pi_1(\orb))$.

We also have lifts $ \torb \ra \SI^n$ and $ \pi_1(\orb) \ra \SLnp$
again denoted by $\dev$ and $h$
and are also called developing maps and holonomy homomorphisms. 
The discussions below apply to $\rpn$ and $\SI^n$ equally. 
This pair also completely determines the real projective structure on $\orb$. 
Fixing $\dev$, we can identify $\torb$ with $\dev(\torb)$ in $\SI^n$ when $\dev$ is an embedding. 
This identifies $\pi_1(\orb)$ with a group of projective automorphisms $\Gamma$ in $\Aut(\SI^n)$.
The image of $h'$ is still called a {\em holonomy group}.

An orbifold $\orb$ is {\em convex} (resp. {\em properly convex} and {\em complete affine}) 
if $\torb$ is a convex domain (resp. a properly convex domain and an affine subspace). 


A {\em totally geodesic hypersurface} $A$ in $\torb$ or $\orb$ is 
a subset where each point $p$ in $A$ has a neighborhood $U$ projectively 
diffeomorphic to an open set in a closed half-space $\bR^{n}_{+}$
 where $A$ corresponds to $\Bd \bR^{n}_{+}$. 


\begin{remark}
A summary of the deformation spaces of real projective structures on closed orbifolds and surfaces is given 
in \cite{Cbook} and \cite{Choi2004}. See also Marquis \cite{Marquis} for the end theory of $2$-orbifolds. 
The deformation space of real projective structures on an orbifold 
loosely speaking is the space of isotopy equivalent real projective structures on
a given orbifold. (See \cite{conv} also.) 
\end{remark}

\subsection{A classification of ends.} 
There is a general survey \cite{EDC1} for these topics. 
We will now try to describe our classification methods. 
Two oriented geodesic starting from a point $x$ of $\bR P^{n}$ (resp. $\SI^{n}$) is {\em equivalent} if they agree on 
small open neighborhood of $x$.  
A {\em direction} of a geodesic starting a point $x$ of $\bR P^{n}$ (resp. $\SI^{n}$) is an equivalence class of geodesic segments 
starting from $x$. 

\begin{description} 
\item[Radial ends]
The end $E$ has a neighborhood $U$, and a
component $\tilde U$ of the inverse image $p_\orb^{-1}(U)$
has a $\bGamma_{\tilde E}$-invariant 
foliation by properly embedded projective geodesics ending at a common point $\bv_{\tilde U} \in \rpn$
where $\tilde E$ is a pseudo-end corresponding to $E$ and $\tilde U$.
We call such a point a {\em pseudo-end vertex}. 
\begin{itemize} 
\item The {\em space of directions} of oriented projective geodesics from $\bv_{\tilde E}$ gives us
an \break $(n-1)$-dimensional real projective space $\SI^{n-1}_{\bv_{\tilde E}}$, called a {\em linking sphere}. 
\item Let $\tilde \Sigma_{\tilde E}$ denote the space of equivalence classes of lines from $\bv_{\tilde E}$ in $\tilde \orb$.
$\tilde \Sigma_{\tilde E}$ projects to a convex open domain in an affine space in  $\SI^{n-1}_{\bv_E}$
by the convexity of $\torb$. 
\item The subgroup $\bGamma_{\tilde E}$, a so-called pseudo-end fundamental group, of $\bGamma$ 
fixes $\bv_{\tilde E}$ and  acts on 
as a projective automorphism group on $\SI^n _{\bv_E}$. 
Thus, $\tilde \Sigma_{\tilde E}/\bGamma_{\tilde E}$ admits a real projective 
structure of dimension $n-1$. 
\item Let $\Sigma_{\tilde E} $ denote the closed real projective $(n-1)$-orbifold $\tilde \Sigma_{E}/\bGamma_{E}$. 
Since we can find a transversal orbifold $\Sigma_{\tilde E}$ to the radial foliation in 
a pseudo-end-neighborhood for each pseudo-end $\tilde E$ of $\mathcal{O}$,
it lifts to a transversal surface $\tilde \Sigma_{\tilde E}$ in $\tilde U$. 
\item We say that a radial pseudo-end $\tilde E$ is  {\em convex} (resp. {\em properly convex}, and {\em complete affine}) 
if $\tilde \Sigma_{\tilde E}$ is convex  (resp. properly convex, and complete affine). 
\end{itemize}
Note $\tilde E$ is always convex. 
The real projective structure on $\Sigma_{\tilde E'}$ is independent of $\tilde E'$ as long 
as $\tilde E'$ corresponds to a same end $E'$ of $\orb$. 
We will just denote it by $\Sigma_{E'}$ sometimes.


\item[Totally geodesic ends] 
An {\em end} is totally geodesic if an end-neighborhood $U$ has as the closure an orbifold $\clo(U)$
in an ambient orbifold where 
\begin{itemize}
\item $\clo(U) = U \cup \Sigma_{E}$ for a totally geodesic suborbifold $\Sigma_E$ and 
\item where $\clo(U)$ is diffeomorphic to $\Sigma_E \times I$ for an interval $I$. 
\end{itemize}
$\Sigma_{E}$ is said to be the {\em ideal boundary component} of $E$, also called the end orbifold of $E$. 
Two compactifications are equivalent if some respective neighborhoods of
the ideal boundary components in ambient orbifolds 
are projectively diffeomorphic. 
If $\Sigma_E$ is properly convex, then the end is said to be {\em properly convex}. 
(One can see in \cite{cdcr1} two nonequivalent ways to compactify for a real projective elementary annulus.)
\end{description}
Note that the diffeomorphism types of end orbifolds are determined for radial or totally geodesic ends. 
(For other types of ends not covered, there might be some ambiguities.) 
From now on, we will  say that a radial end is an {\em R-end} and a totally geodesic end is a {\em T-end}. 

In this paper, we will only consider the properly convex radial ends and totally geodesic ends. 




\subsubsection{Lens domains, lens-cones, and so on.} 

Define $\Bd A$ for a subset $A$ of $\rpn$ or in $\SI^n$ to be the {\em topological boundary} in $\rpn$  or in $\SI^n$ respectively. 
If $A$ is a domain of subspace of $\rpn$ or $\SI^n$, we denote by $\Bd A$ the topological boundary 
in the subspace. 
The closure $\clo(A)$ of a subset $A$ of $\rpn$ or $\SI^n$ is the topological closure in $\rpn$ or in $\SI^n$. 
Define $\partial A$ for a manifold or orbifold $A$ to be the {\em manifold or orbifold boundary}. 
Also, $A^o$ will denote the manifold or orbifold interior of $A$. 

\begin{definition}\label{defn-join}
Given a convex set $D$ in $\rpn$, we obtain a connected cone $C_D$ in $\bR^{n+1}-\{O\}$ mapping to $D$,
determined up to the antipodal map. For a convex domain $D \subset \SI^n$, we have a unique domain $C_D \subset \bR^{n+1}-\{O\}$. 

A {\em join} of two properly convex subsets $A$ and $B$ in a convex domain $D$ of $\rpn$ or $\SI^n$ is defined 
\[A \ast B := \{[ t x + (1-t) y]| x, y \in C_D,  [x] \in A, [y] \in B, t \in [0, 1] \} \]
where $C_D$ is a cone corresponding to $D$ in $\bR^{n+1}$. The definition is independent of the choice of $C_D$
but depends on $D$. 
\end{definition} 

\begin{definition}
Let $C_1, \dots, C_m$ be cone respectively in a set of independent vector subspaces $V_1, \dots, V_m$ of $\bR^{n+1}$. 
In general, the {\em sum} of convex sets $C_1, \dots, C_m$ in $\bR^{n+1}$ 
in independent subspaces $V_i$ is defined as
\[ C_1+ \dots + C_m := \{v | v = c_1+ \cdots + c_m, c_i \in C_i \}.\]
A {\em strict join} of convex sets $\Omega_i$ in $\SI^n$ (resp. in $\bR P^n$) is given as 
\[\Omega_1 \ast \cdots \ast \Omega_m := \Pi(C_1 + \cdots C_m) \hbox{ (resp. } \Pi'(C_1 + \cdots C_m)  ) \]
where each $C_i-\{O\}$ is a convex cone with image $\Omega_i$ for each $i$. 
\end{definition}

In the following, all the sets are required to be inside an affine subspace $A^n$ and its closure either is in $\bR P^n$ or $\SI^n$. 
\begin{itemize}
\item $K$ is {\em lens-shaped} if it is a convex domain and $\partial K$ is a disjoint union of two smoothly strictly convex embedded open 
$(n-1)$-cells $\partial_+ K$ and $\partial K_-$. 
\item A {\em cone} is a domain $D$ in $A^n$ whose closure has a point in the boundary, called an {\em end vertex} $v$
so that every other point $x \in D$ has a properly convex segment $l$, $l^{o}\subset D$, with endpoints $x$ and $v$. 
\item A {\em cone} $\{p\} \ast L$ over a lens-shaped domain $L$ in $A^n$, $p\not\in \clo(L)$ is a convex domain
so that 
\begin{itemize}
\item $\{p\}\ast L = \{p\} \ast \partial_+ L$ for one boundary component $\partial_+ L$ of $L$. 
\item $\partial_{-} L$ meets each maximal segment in $\{p\} \ast L$ from $p$ at a unique point. 
\end{itemize}
A {\em lens} is the lens-shaped domain $L$ ( not determined uniquely by the lens-cone itself).
One of  two boundary components of $L$ is called {\em top} or {\em bottom} hypersurfaces 
depending on whether it is further away from $p$ or not. The top component is denoted by $\partial_+ L$.
The bottom one is denoted by $\partial_-L$.
\item We can allow $L$ to have non-smooth boundary or not strictly convex boundary 
that lies in the boundary of $p \ast L$. 
\begin{itemize}
\item A cone over $L$ where $\partial (\{p\} \ast L -\{p\}) = \partial_+ L, p \not\in \clo(L)$ is said to be a {\em generalized lens-cone} 
and $L$ is said to be a {\em generalized lens}. We define $\partial_{+} L$ and $\partial_{-} L$ similarly 
as above. 
\end{itemize}
\end{itemize}

\begin{itemize} 
\item A {\em totally-geodesic domain} is a convex domain in a hyperspace. 
A {\em cone-over} a totally-geodesic domain $D$ is a union of all segments with  one endpoint a point $x$ not in the hyperspace
and the other in $D$. We denote it by $\{x\} \ast D$. 
\end{itemize}



\begin{description}
\item[Lens-shaped R-end]  An R-end $\tilde E$ is {\em lens-shaped} (resp. {\em totally geodesic cone-shaped},
{\em generalized lens-shaped})  
 if it has a pseudo-end-neighborhood that is a lens-cone (resp. a cone over a totally-geodesic domain, 
 a generalized lens-cone pseudo-end-neighborhood) 
 Here, we require that $\bGamma_{\tilde E}$ acts on the lens of the lens-cone.  
\end{description} 

Let the radial pseudo-end $\tilde E$ have a pseudo-end-neighborhood that is the interior of
$\{p\} \ast L -\{p\}$ where 
$p \ast L$ is a generalized lens-cone over a generalized lens $L$
where $\partial (p\ast L -\{p\}) = \partial_+L$, and let $\bGamma_{\tilde E}$ acts on $L$. 
A {\em concave pseudo-end-neighborhood} of $\tilde E$ is the open pseudo-end-neighborhood in $\torb$
of form $\{p\}\ast L -\{p\} - L$. 
 
\begin{description}  
\item[Lens-shaped T-end] A pseudo-T-end $\tilde E$ of $\torb$ is of {\em lens-type} if 
it has a $\bGamma_{\tilde E}$-invariant 
closed lens-neighborhood $L$ in an ambient orbifold of $\torb$. 
Here a closed p-T-end neighborhood  in $\torb$ 
is compactified by a totally geodesic hypersurface in a hyperplane $P$. 
We require that $L/\bGamma_{\tilde E}$ is a compact orbifold. 
$\partial L \cap \torb$ is smooth and strictly convex and $\Bd \partial L \subset P$ for
the hyperspace $P$ containing the ideal boundary of $\tilde E$.  
A T-end of $\orb$ is of {\em lens-type} if 
the corresponding pseudo-T-end is of lens-type. 
\end{description}

From now on, we will replace the term ``pseudo-end'' with ``p-end'' everywhere. 
A {\em lens p-end neighborhood} of a p-T-end $\tilde E$ of lens-type is a component $C_{1}$ of $L - P$ in $\torb$. 

 \begin{remark} 
 The main reason we are studying the lens-type R-ends are to use them in studying the deformations preserving the convexity properties. 
 These objects are useful in trying to understand this phenomenon.
 \end{remark} 
 
 \begin{remark}
 There is an independent approach to the end theory by Cooper, Long, Leitner, and Tillman announced in the summer of 2014. Our theory 
 overlaps with theirs in many cases. (See \cite{Leitner1} and \cite{Leitner2}.) 
 However, their ends have nilpotent fundamental groups.
 They approach gives us some what simpler criterions to tell the existence of these types of ends. 
 \end{remark}
 
 Also, sometimes, a lens-type p-end neighborhood may not exist for a p-R-end. 
 However, generalized lens-type p-end neighborhood  may exists for the p-R-end.


\subsection{Main results.} 
\subsubsection{The definitions} 
The following applies to both R-ends and T-ends. 
Let $\tilde E$ be a p-end and $\bGamma_{\tilde E}$ the associated p-end fundamental group.
We say that 
$\tilde E$ is {\em virtually non-factorable} if any finite index subgroup has a finite center or 
$\bGamma_{\tilde E}$ is virtually center-free;
otherwise, $\tilde E$ is virtually factorable. 

Let $\tilde \Sigma_{\tilde E}$ denote the universal cover of the end orbifold $\Sigma_{\tilde E}$ associated with $\tilde E$. 
By Theorem 1.1 of Benoist \cite{Ben2}, if $\bGamma_{\tilde E}$ is virtually factorable, then 
$\bGamma_{\tilde E}$ satisfies the following condition
where $K_{i}$ is not necessarily strictly convex and $G_{i}$ below may not be discrete. 


\begin{definition} \label{defn-admissible} 
$\bGamma_{\tilde E}$ is an {\em admissible group} if 
the following hold: 
\begin{itemize}
\item  $\clo(\tilde \Sigma_{\tilde E}) = K_{1}\ast \cdots \ast K_{k}$ 
where each $K_{i}$ is strictly convex or is a singleton. 
\item Let $G_{i}$ be the restriction of the $K_{i}$-stabilizing subgroup of $\bGamma_{\tilde E}$ to $K_{i}$. 
Then $G_{i}$ acts on $K_{i}^{o}$ cocompactly and discretely. 
(Here $K_{i}$ can be a singleton and $\Gamma_{i}$ a trivial group. )
\item A finite index subgroup $G'$ of $\bGamma_{\tilde E}$ is isomorphic 
to $\bZ^{k-1}\times G_{1}\times \cdots \times G_{k}$. 
\item We assume that $G_{i}$ is a subgroup of $\bGamma_{\tilde E}$ acting 
on $K_{j}$ trivially for each $j$, $j \ne i$.
\item  The center $\bZ^{k-1}$ of $G'$ is a subgroup acting trivially on each $K_{i}$. 
\end{itemize} 
In this case, we say that $\bGamma_{\tilde E}$ is {\em admissible} with respect to 
$\clo(\tilde \Sigma_{\tilde E}) = K := K_1 * \cdots * K_{l_0}$ in $\SI^{n-1}_{\bv_{\tilde E}}$
for the subgroup $\bZ^{l_{0}-1}\times \bGamma_{1}\times \cdots \times \bGamma_{l_{0}}$.

\end{definition} 
We will use $\bZ^{k-1}$, $G_{i}$ to simply represent the corresponding group on $\bGamma_{\tilde E}$.
Since $K_{i}$ is strictly convex, $G_{i}$ is a hyperbolic group or is a trivial group for each $i$. 
$\bZ^{k-1}$ is called a {\em virtual center} of $G$. 
(We will of course wish to remove this admissibility condition in the future.)



Let $\Gamma$ be generated by finitely many elements $g_1, \ldots, g_m$. 
Let $w(g)$ denote the minimum word length of $g \in G$ written as words of $g_{1}, \dots, g_{m}$. 
The {\em conjugate word length} $\cwl(g)$ of $g \in \pi_1(\tilde E)$ is
\[\min \{ w(cgc^{-1})| c \in \pi_{1}(\tilde E) \}. \]

Let $d_K$ denote the Hilbert metric of the interior $K^o$ of a properly convex domain $K$ in $\bR P^n$ or $\SI^n$. 
Suppose that a projective automorphism group $\Gamma$ acts on $K$ properly. 
Let $\leng_K(g)$ denote the infimum of $\{ d_K(x, g(x))| x \in K^o\}$, compatible with $\cwl(g)$. 

We show that the norm of the eigenvalues $\lambda_i(g)$ equals $1$ for every $g \in \bGamma_{\tilde E}$ 
if and only if $\tilde E$ is horospherical or 
an NPCC-end with fiber dimension $n-1$
by Theorems 4.10 
and \ref{I-thm-comphoro} of \cite{EDC1}. 

A subset $A$ of $\bR P^n$ or $\SI^n$ {\em spans} a subspace $S$ if $S$ is the smallest subspace containing $A$. 


\begin{definition}\label{defn-umec}
Let $\bv_{\tilde E}$ be a p-end vertex of a p-R-end $\tilde E$. 
The p-end fundamental group $\bGamma_{\tilde E}$ satisfies the {\em uniform middle-eigenvalue condition} 
if the following hold: 
\begin{itemize}
\item each $g\in \bGamma_{\tilde E}$ satisfies for a uniform  $C> 1$ independent of $g$
\begin{equation}\label{eqn-umec}
C^{-1} \leng_K(g) \leq \log\left(\frac{\bar\lambda(g)}{\lambda_{\bv_{\tilde E}}(g)}\right) \leq C \leng_K(g) , 
\end{equation}
for $\bar \lambda(g)$ equal to 
the largest norm of the eigenvalues of $g$
and the eigenvalue $\lambda_{\bv_{\tilde E}}(g)$ of $g$ at $\bv_{\tilde E}$.

\end{itemize} 


\end{definition}
The definition of course applies to the case when $\bGamma_{\tilde E}$ has the finite index subgroup with the above properties. 


We give a dual definition: 
\begin{definition} 
Suppose that $\tilde E$ is a properly convex p-T-end. 
Let $g^*:\bR^{n+1 \ast} \ra \bR^{n+1 \ast}$ be the dual transformation of $g: \bR^{n+1} \ra \bR^{n+1}$. 
The p-end fundamental group $\bGamma_{\tilde E}$ satisfies the {\em uniform middle-eigenvalue condition}
if it satisfies 
\begin{itemize}
\item  if each $g\in \bGamma_{\tilde E}$ satisfies for a uniform  $C> 1$ independent of $g$
\begin{equation}\label{eqn-umecD}
C^{-1} \leng_K(g) \leq \log\left(\frac{\bar\lambda(g)}{\lambda_{K^{*}}(g^{*})}\right) \leq C \leng_K(g) , 
\end{equation}
 for the largest norm $\bar \lambda(g)$ 
of the eigenvalues of $g$ for  
the eigenvalue $\lambda_{K^{*}}(g^{*})$ of $g^*$ in the vector in the direction of $K^*$, the point dual 
to the hyperplane containing $K$. 
\end{itemize} 

\end{definition} 

Here $\bGamma_{\tilde E}$ will act on a properly convex domain $K^o$ of lower dimension
and we will apply the definition here. 
This condition is similar to ones studied by Guichard and Wienhard \cite{GW}, and the results also 
seem similar. We do not use their theories.
Our main tools to understand these questions are in Appendix \ref{app-dual}. 

We will see that the condition is an open condition; and hence a ``structurally stable one."
(See Corollary \ref{cor-mideigen}.)  

\subsubsection{Main results}

As holonomy groups, the conditions for being a generalized lens p-R-end and 
being a lens p-R-end are equivalent. 
For the following, we are not concerned with a lens-cone being in $\torb$. 

\begin{theorem}[Lens holonomy]\label{thm-equiv}
Let $\tilde E$ be a p-R-end of a  strongly tame properly convex real projective orbifold. 
Let $h(\pi_{1}(\tilde E))$ be the admissible holonomy group of a p-R-end. 
Then 
 $h(\pi_{1}(\tilde E))$ satisfies the uniform middle eigenvalue condition
 if and only if it acts on a lens-cone. 
\end{theorem} 

For the following, we are concerned with a lens-cone being in $\torb$. 

\begin{theorem}[Actual lens-cone]\label{thm-secondmain} 
Let $\mathcal{O}$ be a strongly tame properly convex real projective orbifold. 
Assume that the holonomy group of $\mathcal{O}$ is strongly irreducible.
\begin{itemize} 
\item Let $\tilde E$ be a properly convex p-R-end with an admissible end fundamental group. 
\begin{itemize} 
\item The p-end holonomy group satisfies the uniform middle-eigenvalue condition
if and only if $\tilde E$ is a generalized lens-type p-R-end.
\end{itemize} 
\item If $\orb$ satisfies the triangle condition {\rm (}see Definition \ref{defn-tri}{\rm )} 
or $\tilde E$ is virtually factorable or is a totally geodesic R-end, 
then we can replace the word ``generalized lens-type''
to ``lens-type'' in each of the above statements. 
\end{itemize} 
\end{theorem}
This is repeated as Theorem \ref{thm-equ}. 
We will prove the analogous result for totally geodesic ends in Theorem \ref{thm-equ2}. 

Another main result is on the duality of lens-type ends: 
For a vector space $V$, we define $\mathcal{P}(V)$ as $(V-\{O\})/v \sim s v$ for $s \ne 0$. 
Let $\bR P^{n \ast}={\mathcal P}(\bR^{n+1 \ast})$ be the dual real projective space of $\bR P^n$. 
In Section \ref{sec-endth}, we define the projective dual domain $\Omega^*$ in $\bR P^{n \ast}$ 
to a properly convex domain $\Omega$ in $\bR P^n$ where 
the dual group $\Gamma^*$ to $\Gamma$ acts on. 
Vinberg showed that there is a duality diffeomorphism between $\Omega/\Gamma$ and $\Omega^{\ast}/\Gamma^{\ast}$. 
The ends of $\orb$ and $\orb^*$ are in a one-to-one correspondence. 
Horospherical ends are dual to themselves, i.e., ``self-dual types'', 
and properly convex R-ends and T-ends are dual to one another. (See Proposition \ref{prop-dualend}.)
We will see that properly convex R-ends of generalized lens-type 
are always dual to T-ends of lens-type by Corollary \ref{cor-duallens2}.

\subsubsection{Examples} 
We caution the readers that 
results theorems work well for orbifolds with actual singularities in the end neighborhoods. 
For manifolds, we may not have these types of ends as investigated by Ballas \cite{Ballas2012}, \cite{Ballas2014}, 
Cooper, and Leitner (see \cite{Leitner1}, \cite{Leitner2}).

In Chapter 8 of \cite{conv}, there are two examples given by S. Tillman and myself with above types of ends. 
Later, Gye-Seon Lee and I computed more examples 
starting from hyperbolic Coxeter orbifolds (These are not published results.)
Assume that these structures are properly convex. 
In these cases, they have only lens-type R-end by Proposition \ref{I-prop-lensauto} in \cite{EDC1}. 

Recently in 2014, Gye-Seon Lee has found exactly computed one-parameter families of 
real projective structures deformed from a complete hyperbolic structure on 
the figure eight knot complement and from one on the figure-eight sister knot complement. 
These have R-ends only. 
Assuming that these structures are properly convex, 
the ends will correspond to lens-type R-ends or cusp R-ends by Corollary \ref{cor-Coxeter}
since the computations shows that the end satisfies the unit eigenvalue condition of the corollary. 

Also, Ballas \cite{Ballas2014} and \cite{Ballas2012}
found another types of ends using cohomological methods. We believe that they are classified 
in the next paper in this series \cite{EDC3}. 
Ballas, Danciger, and Lee also announced in Cooperfest in Berkeley in May, 2015, 
that the deformation to radial lens-type R-ends are very generic phenomena
when they scanned the Hodgson-Week's censors of hyperbolic manifolds. 

The proper convexity of these types deformed real projective orbifolds of examples will be proved in \cite{conv}. 






\begin{corollary} \label{cor-Coxeter} 
Let $\mathcal{O}$ be a strongly tame properly convex real projective orbifold with radial or totally geodesic ends. 
Assume that the holonomy group of $\mathcal{O}$ is strongly irreducible.
Let $\tilde E$ be a p-R-end  with an admissible end fundamental group. 
\begin{itemize} 
\item Let $\tilde E$ be a p-R-end 
has the p-end holonomy group with eigenvalue $1$ 
at the p-end vertex.  
Suppose that 
$\tilde E$ is not NPCC. Then $\tilde E$ is a generalized lens-type p-R-end or a horospherical {\rm (}cusp{\rm )} R-end. 
\item Let $\tilde E$ be a p-T-end and have the $1$-form defining the p-T-end $\tilde E$ has eigenvalue $1$.  
Then $\tilde E$ is a lens-type p-T-end. 
\end{itemize} 
\end{corollary} 
Examples are orbifolds with R-end orbifolds that have Coxeter groups as the fundamental groups
since generators must fix each end vertex with eigenvalue $1$. 




Our work is a ``classification'' since 
we will show how to construct lens-type R-ends (Theorem \ref{thm-equ}), 
lens-type T-ends (Theorem \ref{thm-equ2}). 
(See also Example \ref{III-exmp-joined} in \cite{EDC3}.) 
(Of course, provided that we know how to 
compute certain cohomology groups.)

\subsection{Applications}
Now, we explain the applications of the main results: 
We will also show that lens-shaped ends are stable (see Theorem \ref{thm-qFuch}) 
and that we can always approximate the whole universal cover with lens-shaped end neighborhoods. 
(See Lemma \ref{lem-expand}.)

For a strongly tame orbifold $\mathcal{O}$, 
\begin{itemize}
\item[(IE)] $\orb$ or $\pi_1(\orb)$ satisfies the {\em infinite-index end fundamental group condition} 
if $\pi_1(\tilde E)$ is of infinite index in $\pi_1(\mathcal O)$ 
for the fundamental group $\pi_1(\tilde E)$ of each p-end $\tilde E$. 
\item[(NA)] 
If $\pi_{1}(\orb)$ does not contain a free abelian group of rank two, and 
if $\bGamma_{E_{1}} \cap \bGamma_{E_{2}}$ is finite for any pair of distinct end fundamental groups
$\bGamma_{E_{1}}$ and $\bGamma_{E_{2}}$, 
we say that $\orb$ or $\pi_1(\orb)$ satisfies {\em no essential annuli condition} or (NA).
\end{itemize}

Our final main result of this paper is the following: 
\begin{theorem}\label{thm-sSPC} 
Let $\orb$ be a noncompact strongly tame properly convex real projective orbifold with 
horospherical,  generalized lens-type R-ends or lens-type T-ends 
with admissible end fundamental groups
and satisfy {\em (IE)} and {\em (NA)}. 
Then the holonomy group is strongly  irreducible and is not 
contained in a parabolic subgroup of $\PGL(n+1, \bR)$  {\rm (}resp.  $\SLpm${\rm ).}
\end{theorem} 
For closed properly convex real projective orbifold, this was shown by Benoist \cite{Ben1}.
This result should generalize with different types of ends.

\subsection{Outline.} 





In Section \ref{sec:prelim},  we review some basic terms.

In Section \ref{sec-endth}, we start to study the R-end theory. First, we discuss the holonomy representation spaces.
Tubular actions and the dual theory of affine actions are discussed. We show that distanced actions
and asymptotically nice actions are dual. We prove that the uniform middle eigenvalue condition 
implies the existence of the distanced action. 



In Section \ref{sec-chlens}, 
we show that the uniform middle-eigenvalue condition of a properly convex end is equivalent to the 
lens-shaped property of the end under some assumptions. In particular, this is true for 
virtually factorable properly convex ends. This is a major section with numerous central lemmas. 

First, we estimate the largest eigenvalue $\lambda_{1}(g)$ in terms of word length.  Next, we study orbits 
under the action with the uniform middle eigenvalue conditions. We show how to make a strictly convex boundary 
of a lens. 
We prove Theorems \ref{thm-equiv} and \ref{thm-secondmain}. 

In Section \ref{sec-lens}, 
we discuss the properties of lens-shaped ends. We show that if the holonomy is strongly irreducible, 
the lens shaped ends have concave neighborhoods. If the generalized lens-shaped end is virtually factorable, then 
it can be made into a totally-geodesic R-end of lens-type, which is a surprising result in the author's opinion. 

In Section \ref{sec-dualT},  we discuss the theory of lens-type T-ends. The theory basically follows from that of 
lens-type R-ends. 
We obtain the duality between the T-ends of lens-type and R-ends of generalized lens-type. 
We also prove Corollary \ref{cor-Coxeter}.

From now on the article list applications of the main theory. 

In Section \ref{sec-results}, we prove many results we need in another paper \cite{conv}, 
not central to this paper. 
Also, we show that the lens-shaped property is a stable property under 
the change of holonomy representations. 
We will define limits sets of ends and discuss the properties in Proposition \ref{prop-I}. 
We obtain the exhaustion of $\torb$ by a sequence of p-end-neighborhoods 
of $\torb$. 
We have two other results here. 

We go to Section \ref{sec-strirr}.
Let $\orb$ be a strongly tame 
properly convex real projective orbifold with generalized lens-type R-ends or lens-type T-end
and satisfy {\rm (IE)} and {\rm (NA)}. We prove the strong irreducibility of $\orb$;
that is, Theorem \ref{thm-sSPC}.

In Appendix \ref{app-dual}, we show that the affine action of a 
strongly irreducible group $\Gamma$ acting cocompactly 
on a convex domain $\Omega$ in the boundary of the affine space is 
asymptotically nice if $\Gamma$ satisfies the uniform middle-eigenvalue condition. 
We will dualize this result. 
This was needed in Section \ref{sec-endth}.

In Appendix \ref{app-quasi-lens}, we will generalize the uniform middle eigenvalue condition 
slightly and show that the corresponding end has to be of quasi-lens type one. 
We classify these in Propositions \ref{prop-quasilens1},  \ref{prop-quasilens2}),

In Appendix \ref{app-Koszul}, we prove a minor extension of Koszul's openness for bounded manifolds, well-known 
to many people.

\begin{remark}
Note that the results are stated in the space $\SI^n$ or $\bR P^n$. Often the result for $\SI^n$ implies 
the result for $\bR P^n$. In this case, we only prove for $\SI^n$. In other cases, we can easily modify 
the $\SI^n$-version proof to one for the $\bR P^n$-version proof. We will say this in the proofs. 
\end{remark}

 
 We also remark that this paper is a part of a longer earlier paper \cite{endclass} to be published in three papers.  

We thank David Fried for helping me understand the issues with the distanced nature of the tubular actions and duality.
We thank Yves Benoist with some initial discussions on this topic, which were very helpful
for Section \ref{sub-holfib} and thank Bill Goldman and Francois Labourie 
for discussions resulting in Appendix \ref{sub-asymnice}.
We thank Samuel Ballas, Daryl Cooper and Stephan Tillmann 
for explaining their work and help and we also thank Micka\"el Crampon and Ludovic Marquis also. 
Their works obviously were influential 
here. The study was begun with a conversation with Tillmann at ``Manifolds at Melbourne 2006" 
and I began to work on this seriously from my sabbatical at Univ. Melbourne from 2008. 
We also thank Craig Hodgson and Gye-Seon Lee for working with me with many examples and their 
insights. The idea of R-ends comes from the cooperation with them.



\section{Preliminaries} \label{sec:prelim}

This section is a reminder of notation. These were all explained in \cite{EDC1}. 
Each end-neighborhood $U$ diffeomorphic to $\Sigma_{\tilde E} \times (0, 1)$ of an end $E$ lifts to a connected open set 
$\tilde U$ in $\torb$ 
where a subgroup of deck transformations $\bGamma_{\tilde U}$ acts on $\tilde U$ where 
$p_{\torb}^{-1}(U) = \bigcup_{g\in \pi_1(\orb)} g(\tilde U)$. Here, each component of 
$\tilde U$ is said to a {\em proper pseudo-end-neighborhood}.
\begin{itemize} 
\item A {\em pseudo-end sequence} is a sequence of proper pseudo-end-neighborhoods 
$U_1 \supset U_2 \supset \cdots $ so that for each compact subset $K$ of $\orb$
there exists an integer $N$ so that 
$p_{\orb}^{-1}(K) \cap U_i = \emp$ for $i > N$.  
\item Two pseudo-end sequences are {\em compatible} if an element of one sequence is contained 
eventually in the element of the other sequence. 
\item A compatibility class of a pseudo-end sequence is called a {\em pseudo-end} of $\torb$.
Each of these corresponds to an end of $\orb$ under the universal covering map $p_{\orb}$.
\item For a pseudo-end $\tilde E$ of $\torb$, we denote by $\bGamma_{\tilde E}$ the subgroup $\bGamma_{\tilde U}$ where 
$U$ and $\tilde U$ is as above. We call $\bGamma_{\tilde E}$ is called a {\em pseudo-end fundamental group}.
\item A {\em pseudo-end-neighborhood} $U$ of a pseudo-end $\tilde E$ 
is a $\bGamma_{\tilde E}$-invariant open set containing 
a proper pseudo-end-neighborhood of $\tilde E$. 
\end{itemize}
(See Section \ref{I-sub-ends} of \cite{EDC1} for more detail.)

The general linear group $\GLnp$ acts on $\bR^{n+1}$ and $\PGLnp$ acts faithfully on $\rpn$. 

Denote by $\bR_+ =\{ r \in \bR| r > 0\}$.
The {\em real projective sphere} $\SI^n$ is defined as the quotient of $\bR^{n+1} -\{O\}$ under the quotient relation 
$\vec{v} \sim \vec{w}$ iff $\vec{v} = s\vec{w}$ for $s \in \bR_+$. 
The projective automorphism group $\Aut(\SI^n)$ is isomorphic to the subgroup $\SLnp$ of $\GLnp$ of 
determinant $\pm 1$, double-covers $\PGLnp$. 
A {\em projective map} of a real projective orbifold to another is a map that is projective by charts to $\rpn$. 

Let $\mathcal{P}: \bR^{n+1}-\{O\} \ra \bR P^n$ be a projection 
and let $\mathcal{S}:  \bR^{n+1}-\{O\} \ra \SI^n$ denote one for $\SI^n$. 
the origin removed under the projection $\mathcal{P}$ (resp. $\mathcal{S}$).
Also, given any subspace $V$ of $\bR^{n+1}$ we denote $\mathcal{P}(V)$ the image of $V - \{O\}$ under 
$\mathcal{P}$
(resp. $\mathcal{S}(V)$ the image of $V - \{O\}$ under $\mathcal{S}$). 

A line in $\rpn$ or $\SI^n$ is an embedded arc in a $1$-dimensional subspace. 
A {\em projective geodesic} is an arc immersing into a line in $\rpn$
or to a one-dimensional subspace of $\SI^n$. 
A {\em convex subset} of $\rpn$ is a convex subset of an affine patch.  
A {\em properly convex subset} of  $\rpn$ is a precompact convex subset of an affine subspace. 
$\bR^n$ identifies with an open hemisphere in $\SI^n$ defined by a linear function on $\bR^{n+1}$. 
$\SI^{n}$ and $\bR P^{n}$ have spherical metrics both to be denoted by $\bdd$ where all 
geodesics are projective geodesics and vice versa up to reparameterizations. 

An {\em $i$-dimensional complete affine subspace} is 
a subset of a projective orbifold projectively diffeomorphic to 
an $i$-dimensional affine subspace in some affine subspace $A^n$ of $\rpn$ or $\SI^n$. 

Let $\Omega$ be a convex domain in an affine space $A$ in $\bR P^n$ or $\SI^n$. 
Let $[o, s, q, p]$ denote the cross ratio of four points as defined by 
\[ \frac{\bar o - \bar q}{\bar s - \bar q} \frac{\bar s - \bar p}{\bar o - \bar p} \] 
where \[o=[\bar o, 1], p=[\bar p, 1], q=[\bar q, 1], s=[\bar s, 1]\] for homogeneous coordinates 
of a line or a great circle containing $o, s, p, q$. 
Define the Hilbert metric
\[d_\Omega(p, q)= \log|[o,s,q,p]|\] where $o$ and $s$ are 
endpoints of the maximal segment in $\Omega$ containing $p, q$
where $o, q$ separate $p, s$. 
The metric is one given by a Finsler metric provided $\Omega$ is properly convex. (See \cite{Kobpaper}.)
Given a properly convex real projective structure on ${\mathcal{O}}$, 
the cover $\tilde{\mathcal{O}}$ carries a Hilbert metric which we denote by $d_{\torb}$. 
This induces a metric on ${\mathcal{O}}$. 
(Note that even if $\torb$ is not properly convex, $d_{\torb}$ is still a pseudo-metric.) 

\begin{lemma} \label{lem-nhbd} 
Let $U$ be a convex subset of a properly convex domain $V$. 
Let \[U' := \{x \in V | d_{V}(x, U) \leq \eps\}\] for $\eps > 0$. 
Suppose that $\Bd U \cap V$ is strictly convex or $U$ is totally geodesic. 
Then $U'$ is properly convex and $\Bd U' \cap V$ is strictly convex. 
\end{lemma}
\begin{proof} 
By Lemma 1.8 of \cite{CLT2}. 
Given $u, v \in U'$, we find 
\[w, t \in \Omega \hbox{ so that } d_V(u, w) < \eps, d_V(v, t) < \eps.\]
Then each point of $\ovl{uv}$ is within $\eps$ of $\ovl{wt} \subset U$ in the $d_V$-metric.
\end{proof} 

Let $d_K$ denote the Hilbert metric of the interior $K^o$ of a properly convex domain $K$ in $\bR P^n$ or $\SI^n$. 
Suppose that a projective automorphism group $\Gamma$ acts on $K$ properly and discretely.
Define 
$\leng_K(g):= \inf\{ d_K(x, g(x))| x \in K^o\}$, compatible with $\cwl(g)$. 

Given a properly convex domain $D$ in $\bR P^{n}$, the dual domain is given 
by $D^{*}$ as the set of hyperspaces not meeting $D^{*}$ corresponding to a properly convex domain
in $\bR P^{n\ast}$. 

Note the reversal of inclusions of properly convex domains $A, B$ and the duals $A^{\ast}, B^{\ast}$:
\begin{equation}\label{eqn-reversal} 
A\subset B \hbox{ if and only if } B^{\ast} \subset A^{\ast} 
\end{equation}

\section{The end theory}
\label{sec-endth}
In this section, we discuss the properties of lens-shaped radial and totally geodesic ends and their duality also.

\subsection{The holonomy homomorphisms of the end fundamental groups: the tubes.} \label{sub-holfib}

We will discuss for $\SI^n$ only here but the obvious $\bR P^n$-version exists for the theory. 
Let $\tilde E$ be a p-R-end of $\torb$. 
Let $\SLnp_{\bv_{\tilde E}}$ be the subgroup of $\SLnp$ fixing a point $\bv_{\tilde E} \in \SI^n$.
This group can be understood as follows by letting $\bv_{\tilde E} = [0, \ldots, 0, 1]$ 
as a group of matrices: For $g \in \SLnp_{\bv_{\tilde E}}$, we have 
\[ \left( \begin{array}{cc} 
        \frac{1}{\lambda_{\bv_{\tilde E}}(g)^{1/n}} \hat h(g) & \vec{0} \\ 
        \vec{v}_g                & \lambda_{\bv_{\tilde E}}(g)
        \end{array} \right) \] 
where $\hat h(g) \in \SLn, \vec{v} \in \bR^{n \ast}, \lambda_{\bv_{\tilde E}}(g) \in \bR_+ $, 
is the so-called linear part of $h$.
Here, \[\lambda_{\bv_{\tilde E}}: g \mapsto \lambda_{\bv_{\tilde E}}(g) \hbox{ for } g \in \SLnp_{\bv_{\tilde E}}\] is a homomorphism 
so it is trivial in the commutator group $[\bGamma_{\tilde E}, \bGamma_{\tilde E}]$. 
There is a group homomorphism 
\begin{align} 
{\mathcal L}': \SLnp_{\bv_{\tilde E}} & \ra \SLn \times \bR_+ \nonumber \\
g &\mapsto (\hat h(g), \lambda_{\bv_{\tilde E}}(g)) 
\end{align} 
with the kernel equal to $\bR^{n \ast}$, a dual space to $\bR^n$. 
Thus, we obtain a diffeomorphism \[\SLnp_{\bv_{\tilde E}} \ra \SLn \times \bR^{n \ast} \times \bR_+.\]
We note the multiplication rules
\[ (A, \vec{v}, \lambda) (B, \vec{w}, \mu) = (AB, \frac{1}{ \mu^{1/n} } \vec{v}B + \lambda \vec{w}, \lambda \mu). \] 
(We denote by ${\mathcal L}_{1}$ the further projection to $\SLn$.)

Let $\Sigma_{\tilde E}$ be the end $(n-1)$-orbifold. 
Given a representation 
\[\hat h: \pi_1(\Sigma_{\tilde E}) \ra \SLn \hbox{ and a homomorphism } \lambda: \pi_1(\Sigma_{\tilde E}) \ra \bR_+,\] 
we denote by $\bR^{n}_{\hat h, \lambda}$
the $\bR$-module with the $\pi_1(\Sigma_{\tilde E})$-action given 
by \[g\cdot \vec v = \frac{1}{\lambda(g)^{1/n}}\hat h(g)(\vec v).\] 
And we denote by $\bR^{n \ast}_{\hat h, \lambda}$ the dual vector space
with the right dual action given by 
\[g\cdot \vec v = \frac{1}{{\lambda(g)^{1/n}}}\hat h(g)^{\ast}(\vec v).\] 
Let $H^1(\pi_1(\tilde E), \bR^{n \ast}_{\hat h, \lambda})$ denote the cohomology 
space of $1$-cocycles $\vec v(g) \in  \bR^{n \ast}_{\hat h, \lambda}.$ 


As $\Hom(\pi_1(\Sigma_{\tilde E}), \bR_+)$ equals $H^1(\pi_1(\Sigma_{\tilde E}), \bR)$, we obtain: 

\begin{theorem} \label{thm-defspace}
Let $\orb$ be a  strongly tame  convex real projective orbifold, and 
let $\torb$ be its universal cover. 
Let $\Sigma_{\tilde E}$ be the end orbifold associated with a p-R-end $\tilde E$ of $\torb$. 
Then the space of representations 
\[\Hom(\pi_1(\Sigma_{\tilde E}), \SLnp_{\bv_{\tilde E}})/\SLnp_{\bv_{\tilde E}}\] 
is  the fiber space over 
\[\Hom(\pi_1(\Sigma_{\tilde E}), \SLn)/\SLn \times H^1(\pi_1(\Sigma_{\tilde E}), \bR)\]
with the fiber isomorphic to $H^1(\pi_1(\Sigma_{\tilde E}), \bR^{n \ast}_{\hat h, \lambda}) $ 
for each $([\hat h], \lambda)$. 
\end{theorem}


We remark that we don't really understand the fiber dimensions and their behavior as we change 
the base points. A similar idea is given by Mess  \cite{Mess}. 
In fact, the dualizing these matrices gives us
a representation to $\Aff(A^n)$. In particular if we restrict ourselves 
to linear parts to be in $\SO(n, 1)$, then we are exactly in the cases studied by Mess. 
(See the concept of the duality in Section \ref{sub-affdualtub} and Appendix \ref{app-dual}.)
Thus, one interesting question of Benoist is how to compute the dimension of 
$H^1(\pi_1(\Sigma_{\tilde E}), \bR^{n \ast}_{\hat h, \lambda}) $ under some general conditions on $\hat h$.

\subsubsection{Tubular actions.}

Let us give a pair of antipodal points $\bv$ and $\bv_-$. 
If a group $\Gamma$ of projective automorphisms fixes a pair of fixed points $\bv$ and $\bv_-$, 
then $\Gamma$ is said to be {\em tubular}.
There is a projection $\Pi_{\bv}: \SI^n -\{\bv, \bv_-\} \ra \SI^{n-1}_{\bv}$ given 
by sending every great segment with endpoints $\bv$ and $\bv_-$
to the sphere of directions at $\bv$. 
(We denote by $\bR P^{n-1}_{\bv}$ the quotient of $\SI^{n-1}_\bv$ under the antipodal map 
given by the change of directions. 
We use the same notation $\Pi_{\bv}: \bR P^n -\{\bv\} \ra \bR P^{n-1}_{\bv}$ 
for the induced projection.)

A {\em tube} in $\SI^n$ (resp. in $\bR P^n$) is the closure of the inverse image 
$\Pi^{-1}_{\bv}(\Omega)$ of a convex domain $\Omega$
in $\SI^{n-1}_{\bv}$ (resp. in $\bR P^{n-1}_{\bv}$).
We denote the closure in $\SI^{n}$ by ${\mathcal T}_{\bv}$, which we call a {\em tube domain}. 
Given a p-R-end $\tilde E$ of $\torb$, let $\bv := \bv_{\tilde E}$. 
The {\em end domain} is $R_{\bv}(\torb)$. 
If a p-R-end $\tilde E$ has the end domain $\tilde \Sigma_{\tilde E} = R_{\bv}(\torb)$, 
$h(\pi_1(\tilde E))$ acts on ${\mathcal T}_{\bv}$. 

We will now discuss for the $\SI^n$-version but the $\bR P^n$ version is obviously clearly obtained from this 
by a minor modification. 

Letting $\bv$ have the coordinates $[0, \dots, 0, 1]$, we obtain 
the matrix of $g$ of $\pi_1(\tilde E)$ of form 
\begin{equation}\label{eqn-bendingm3} 
\left(
\begin{array}{cc}
\frac{1}{\lambda_{\bv}(g)^{\frac{1}{n}}} \hat h(g)          &       0                \\
\vec{b}_g           &      \lambda_{\bv}(g)                  
\end{array}
\right)
\end{equation}
where $\vec{b}_g$ is an $n\times 1$-vector and $\hat h(g)$ is an $n\times n$-matrix of determinant $\pm 1$
and $\lambda_{\bv}(g) $ is a positive constant. 

Note that the representation $\hat h: \pi_1(\tilde E) \ra \SLn$ is given by 
$g \mapsto \hat h(g)$. Here we have $\lambda_{\bv}(g) > 0$.  
If $\tilde \Sigma_{\tilde E}$  is properly convex, then the convex tubular domain and the action are {\em properly tubular}

\subsubsection{Affine actions  dual to tubular actions.}\label{sub-affdualtub}






Let ${\SI^{n-1}}$ in $\SI^{n} = {\mathcal S}(\bR^{n+1})$ be a great sphere of dimension $n-1$.
A component of a component of the complement of ${\SI^{n-1}}$
can be identified with an affine space $A^{n}$. 
The subgroup of projective automorphisms preserving ${\SI^{n-1}}$ and the components equals
the affine group $\Aff(A^n)$.

By duality, a great $(n-1)$-sphere ${\SI^{n-1}}$ corresponds to a point $\bv_{\SI^{n-1}}$. 
Thus, for a group $\Gamma$ in $\Aff(A^n)$, 
the dual groups $\Gamma^*$ acts on $\SI^{n\ast}:={\mathcal S}(\bR^{n+1, *})$ fixing $\bv_{\SI^{n-1}}$.
(See Proposition \ref{I-prop-duality} also.)
%




A hyperspace of $\SI^{m}$ for $0 \leq m \leq n$, {\em supports} a convex domain $\Omega$ if it passes $\Bd \Omega$ but disjoint 
from $\Omega^{o}$.  An oriented hypersurface $\SI^{m}$ for $0 \leq m \leq n$, {\em supports} a convex domain $\Omega$ 
if the hypersurface supports $\Omega$ and the open hemisphere bounded by it in the orientation direction contains $\Omega^{o}$. 

Suppose that $\Gamma$ acts on a properly convex open domain $U$ where $\Omega := \Bd U \cap {\SI^{n-1}_\infty}$
is a properly convex domain. 
We call $\Gamma$ a {\em properly convex affine} action.
Let us recall some facts. 
\begin{itemize}
\item A great $(n-2)$-sphere $P \subset \SI^{n}$ is dual to a great circle $P^{\ast}$ in $\SI^{n\ast}$
given by hyperspheres containing $P$. 
\item The great sphere $\SI^{n-1}_{\infty} \subset \SI^{n}$ with an orientation is dual to a point $\bv \in \SI^{n\ast}$ 
and it with an opposite orientation is dual to $\bv_{-}\in \SI^{n\ast}$. 
\item An oriented hyperspace $P \subset \SI^{n-1}_{\infty}$ of dimension $n-2$ is dual to an oriented great circle 
passing $\bv$ and $\bv_{-}$, giving us an element $P^{\dagger}$ of the linking sphere 
$\SI^{n-1\ast}_{\bv}$ of rays from $\bv$ in $\SI^{n}_{\ast}$.  
\item The space $S$ of oriented hyperspaces in $\SI^{n-1}_{\infty}$ equals $\SI^{n-1 \ast}_{\infty}$. 
Thus, there is a projective isomorphism 
\[\mathcal{I}_{2}: S= \SI^{n-1\ast}_{\infty} \ni P \leftrightarrow P^{\dagger} \in \SI^{n-1\ast}_{\bv}\]
\end{itemize} 

For the following, let's use the terminology that an oriented hyperspace $V$ in $\SI^{i}$ {\em g-supports} 
an open submanifold $A$ if  it bounds an open $i$-hemisphere $H$ in the right orientation containing $A$.

\begin{proposition}\label{prop-dualtube}
Suppose that $\Gamma \subset \SLnp$ acts on a properly convex open domain $\Omega \subset {\SI^{n-1}_\infty}$
cocompactly.
Then the dual group $\Gamma^*$ acts on a properly tubular
domain $B$ with vertices $\bv:= \bv_{\SI^{n-1}_\infty}$ and $\bv_- := \bv_{{\SI^{n-1}_\infty}, -}$ dual to 
$\SI^{n-1}_{\infty}$.
The domain $\Omega^o$ and domain $R_{\bv}(B)$ in the linking sphere $\SI^{n-1}_{\bv}$ from $\bv$ in direction of $B^o$
are projectively diffeomorphic to a pair of dual domains. 
\end{proposition} 
 \begin{proof} 
Given $\Omega^o \subset \SI^{n-1}_{\infty}$, we obtain the properly convex open dual domain 
$\Omega^{o\ast}$ in $\SI^{n-1 \ast}_\infty$. 
A supporting $n-2$-hemisphere of $\Omega$  in $\SI^{n-1}_\infty$  corresponds to 
a point of $\Bd \Omega^{o\ast}$ and vice versa. (See Section \ref{I-sec-duality} of \cite{EDC1}.) 
A great $n-1$-sphere in $\SI^{n}$ g-supporting $\Omega^o$ contains a great $n-2$-sphere $P$ in $\SI^{n-1}_{\infty}$
g-supporting $\Omega^o$. 
The dual $P^{*}$ of $P$ is the set of hyperspaces containing $P$, 
a great circle in $\SI^{n\ast}$. 
The set of oriented great $n-1$-spheres containing $P$  g-supporting $\Omega^{o}$
forms a pencil, in this case a great open segment $I_{P^{\ast}}$ in $\SI^{n \ast}$ with endpoints $\bv$ and $\bv_-$.
Let $P^{\ddagger} \in \SI^{n-1\ast}_{\bv}$ denote the dual of $P$ in terms of $\SI^{n-1}_{\infty}$.
Then $P^{\dagger}:= \mathcal{I}_{2}(P^{\ddagger})$ is the direction of $P^{\ast}$ at $\bv$ as we can see from the projective isomorphism 
$\mathcal{I}_{2}$. Now $P$  g-supports $\Omega^{o}$ if and only if $P^{\ddagger} \in \Omega^{\ast o}$.
Hence, there is a homeomorphism
\begin{align} \label{eqn-dualsupp}
I_{P}&:= \{ Q | Q \hbox{ is an oriented great $n-1$-sphere g-supporting } \Omega^{o}, Q \cap \SI^{n-1}_{\infty }= P\}  \leftrightarrow \nonumber \\
S_{P^{\ast}} &= \{ p| p \hbox{  is a point of a great open segment in } P^{\ast} 
\hbox{ with endpoints } \bv, \bv_{-}  \\ 
& \hbox{ where  the direction $P^{\dagger}= \mathcal{I}_{2}(P^{\ddagger}), P^{\ddagger} \in \Omega^{\ast o}$}\}.
\end{align}

The set $B$ of oriented hyperplanes g-supporting $\Omega^{o}$ meets an oriented $(n-2)$-hyperspace in $\SI^{n-1}_{\infty}$
g-supporting $\Omega^{o}$. 
Thus, we obtain
 \[ B^{\ast} = \bigcup_{P^{\dagger} \in \Omega^{o \ast}} S_{P^{\ast}} \subset \SI^{n\ast}.\]
Let ${\mathcal T}(\Omega^{o\ast})$ denote the union of open great segments of  with endpoints $\bv$ and $\bv_{-}$ 
in direction of $\Omega^{o \ast}$. Thus, $B^{\ast} = {\mathcal T}(\Omega^{o\ast})$.
Thus, there is a homeomorphism
\begin{align} \label{eqn-supp}
I&:= \{ Q | Q \hbox{ is an oriented great $n-1$-sphere supporting } \Omega^{o}\} \leftrightarrow  \nonumber \\
S &= \{ p| p \in S_{P^{\ast}}, P^{\ddagger} \in \Bd \Omega^{o\ast}\} = \Bd B^{\ast} -\{\bv, \bv_{-}\}.
\end{align} 
Also, $R_{\bv}(B^{\ast}) = \Omega^{o\ast}$ by $B^{\ast}:= {\mathcal T}(\Omega^{o\ast})$. 
Thus, $\Gamma$ acts on $\Omega^{o}$ if and only if $\Gamma$ acts on $I$ 
if and only if $\Gamma^{\ast}$ acts on $S$ if and only if 
$\Gamma^{\ast}$ acts on $B^{\ast}$ and on $\Omega^{o\ast}$.
Since these are properly convex open domains, and the actions are cocompact, 
they are uniquely determined up to projective diffeomorphisms. 
\end{proof} 



\subsection{Distanced tubular actions and asymptotically nice affine actions.} 

Given a convex open subset $U$ of $A^n$, an {\em asymptotic hyperspace} $H$ of $U$ at 
a point $x \in \Bd A^n \cap \clo(\Bd U)$ is a hyperspace so that a component of $A^n -H$ contains $U$. 
(There is an approach to this by D. Fried for representations with linear parts in $SO(2, 1)$  alternative 
to the approach of this section.)

\begin{definition}\label{defn-tubular}
\begin{description}
\item[Radial action] A properly tubular action is said to be {\em distanced} if the tubular domain contains 
a properly convex compact $\Gamma$-invariant subset disjoint from the vertices. 
\item[Affine action] A properly convex affine action of $\Gamma$ is said to be {\em asymptotically nice} if 
$\Gamma$ acts on a properly convex open domain $U'$ in $A^n$ with boundary in 
$\Omega \subset {\SI^{n-1}_\infty}$, 
and $\Gamma$ acts on a compact subset $J$ of 
\[\{ H| H \hbox{ is a supporting hyperspace at } x \in \Bd \Omega, H \not\subset {\SI^{n-1}_\infty}\}\] 
 where we require that every supporting $(n-2)$-dimensional space of $\Omega$ in ${\SI^{n-1}_\infty}$ is 
 contained in at least one of the element of $J$. 
\end{description}
\end{definition}

Let $\bdd_{H}$ denote the Hausdorff metric of $\SI^{n}$ with the spherical metric $\bdd$. (See \cite{EDC1} for some details.)

The following is a simple consequence of the homeomorphism 
given by equation \ref{eqn-supp}.
\begin{proposition}\label{prop-dualDA} 
Let $\Gamma$ and $\Gamma^*$ be dual groups where $\Gamma$ has an affine action on $A^n$ and $\Gamma^*$ is tubular with 
the vertex $\bv = \bv_{\SI^{n-1}_\infty}$ dual to the boundary $\SI^{n-1}_\infty$ of $A^n$.
Let $\Gamma= (\Gamma^*)^*$ acts on a convex open domain $\Omega$ with compact $\Omega/\Gamma$.
Then $\Gamma$ acts asymptotically nicely if and only if 
$\Gamma^*$ acts on a properly tubular domain $B$ and is distanced. 
\end{proposition}

\begin{theorem}\label{thm-distanced}
Let $\Gamma$ be a nontrivial properly convex tubular action at vertex $\bv = \bv_{\SI^{n-1}_\infty}$ on 
$\SI^n$ {\rm (}resp. in $\bR P^n${\rm )} 
and acts on a properly convex tube $B$
and satisfies the uniform middle-eigenvalue conditions. 
We assume that $\Gamma$ acts cocompactly and admissibly on a convex open domain $\Omega \subset \SI^{n-1}_{\bv}$ 
where $B = \mathcal{T}(\Omega)$. 
Then $\Gamma$ is distanced inside the tube $B$  where $\Gamma$ acts on. 
Furthermore, $K$ meets each open boundary great segment in $\partial B$ 
at a unique point. 
Finally, $K$ is contained in a hypersphere disjoint from $\bv, \bv_{-}$ when $\Gamma$ is virtually factorable. 
\end{theorem} 
\begin{proof}
Let $\bv$ be the vertex of  $B$. 
First assume that $\Gamma$ is virtually non-factorable. 
$\Gamma$ induces a strongly irreducible action on the link sphere $\SI^{n-1}_{\bv}$. 
Let $\Omega$ denote the convex domain in $\SI^{n-1}_{\bv}$ corresponding to $B^{o}$. 
By Theorem \ref{thm-asymnice}, $\Gamma^*$ is asymptotically nice. 
Proposition \ref{prop-dualDA} implies the result.

Suppose that $\Gamma$ acts virtually reducibly on $\SI^{n-1}_{\bv}$ on a properly convex domain $\Omega$. 
Then $\Gamma$ is isomorphic to $\bZ^{l_0-1} \times \bGamma_1 \times \dots \times \bGamma_{l_0}$ where 
$\bGamma_i$ is nontrivial hyperbolic for $i=1, \dots, s$ and trivial for $s+1 \leq i \leq l_0$ where $s \leq  l_0$. 
By \cite{Ben2}, 
$\Gamma$ acts on \[K:= K_{1}\ast \cdots \ast K_{l_{0}}= \clo(\Omega) \subset \SI^{n-1}_{\bv} \]
where $K_{i}$ denotes the properly convex compact set in $\SI^{n-1}_{\bv}$ where $\bGamma_{i}$ acts on
for each $i$. Here, $K_{i}$ is $0$-dimensional for $i=1, \dots, s$. 
Let $B_{i}$ be the convex tube with vertices $\bv$ and $\bv_{-}$ corresponding to $K_{i}$.
Each $\bGamma_i$ for $i = 1, \dots, s$ acts on a nontrivial tube $B_i$ with vertices $\bv$ and $\bv_-$ in a subspace. 


For each $i$, $s+1 \leq i \leq r$,  $B_i$ is a great segment with endpoints $\bv$ and $\bv_-$. 
A point $p_i$ corresponds to $B_i$ in $\SI^{n-1}_{\bv}$. 



Recall that a nontrivial element $g$ of the center acts trivially on the subspace $K_{i}$ of $\SI^{n-1}_{\bv}$; that is, 
$g$ has only one associated eigenvalue in points of $K_{i}$ by Proposition \ref{I-prop-Ben2} of \cite{EDC1}. 
There exists a nontrivial element $g$ of the center with the largest norm eigenvalue in $K_{i}$ 
since the action of $\bGamma_{\tilde E}$ on $\tilde \Sigma_{\tilde E}$ is compact.  

By the middle eigenvalue condition, 
for each $i$, we can find $g$ in the center so that $g$ has a hyperspace $K'_{i} \subset B_{i}$ with largest norm 
eigenvalues. Since $\bGamma_{i}$ acts on $K'_{i}$ and commutes with $g$, 
$\bGamma_{i}$ also acts on $K'_{i}$. 


The convex hull of
\[K'_1 \cup \cdots \cup K'_{l_0}\] 
in $\clo(B)$ is a distanced $\Gamma$-invariant compact convex set. 

\end{proof}






\section{The characterization of lens-shaped representations} \label{sec-chlens}

The main purpose of this section is to characterize the lens-shaped representations
in terms of eigenvalues. This is a major result of this paper and is needed 
for understanding the duality of the ends.

First, we prove the eigenvalue estimation in terms of lengths for virtually non-factorable and hyperbolic ends. 
We show that the uniform middle-eigenvalue conditions imply the existence of limits. 
This proves Theorem \ref{thm-equiv}.
Finally, we prove the equivalence of the lens condition 
and the uniform middle-eigenvalue condition in Theorem \ref{thm-equ}
for both R-ends and T-ends under very general conditions. That is, we prove 
Theorem \ref{thm-secondmain}.


Techniques here are somewhat related to the work of Guichard, Weinhard \cite{GW}
and Benoist \cite{Ben5}.  Also, when the linear part is in $SO(2, 1)$, D. Fried has 
proven similar results without going to the dual space using cocycle conditions. 

\subsection{The eigenvalue estimations}

Let $\orb$ be a properly convex real projective orbifold 
and $\torb$ be the universal cover in $\SI^n$. 
Let $\tilde E$ be a properly convex p-R-end of $\torb$ and $\bv_{\tilde E}$ be the p-end vertex. 
Let \[h: \pi_1(\tilde E) \ra \SLnp_{\bv_{\tilde E}}\] be a homomorphism and suppose that $\pi_1(\tilde E)$ is hyperbolic. 

Assume that for each nonidentity element of $\pi_1(\tilde E)$, 
the eigenvalue of $g$ at the vertex ${\bv_{\tilde E}}$ of $\tilde E$ has a norm strictly between the maximal
and the minimal norms of eigenvalues of $g$. In this case, we say that $h$ satisfies 
the {\em middle-eigenvalue condition}.

In this article, we assume that $h$ satisfies the middle eigenvalue condition.  
We denote by the norms of eigenvalues of $g$ by
\[\lambda_1(g), \ldots , \lambda_n(g), \lambda_{\bv_{\tilde E}}(g), \hbox{where } \lambda_1(g) \cdots \lambda_n(g) \lambda_{\bv_{\tilde E}}(g)= \pm 1. \]

Recall the linear part homomorphism 
${\mathcal L}_{1}$ from the beginning of Section \ref{sec-endth}. 
We denote by $\hat h: \pi_1(\tilde E) \ra \SLn$ the homomorphism 
${\mathcal L}_1 \circ h$. Since $\hat h$ is a holonomy of a closed convex real projective $(n-1)$-orbifold,
and $\Sigma_{\tilde E}$ is assumed to be properly convex, 
$\hat h(\pi_1(\tilde E))$ divides a properly convex domain $\tilde \Sigma_{\tilde E}$ in $\SI^{n-1}_{\bv_{\tilde E}}$.

We denote by $\tilde \lambda_1(g), ..., \tilde \lambda_n(g)$ the norms of eigenvalues of 
$\hat h(g)$ so that 
\[\tilde \lambda_1(g) \geq  \ldots \geq \tilde \lambda_n(g), \tilde \lambda_1(g)   \ldots \tilde \lambda_n(g) = \pm 1\] hold.
These are called the {\em relative norms of eigenvalues} of $g$.
We have $\lambda_i(g) = \tilde \lambda_i(g)/ \lambda_{\bv_{\tilde E}}(g)^{1/n}$ for $i=1, .., n$.  

Note here that eigenvalues corresponding to 
\[\lambda_1(g), \tilde \lambda_1(g), \lambda_n(g), \tilde \lambda_n(g), \lambda_{\bv_{\tilde E}}(g)\]
are all positive by Benoist \cite{Benasym}. 
We define 
\[\leng(g):= \log\left(\frac{\tilde\lambda_1(g)}{\tilde \lambda_n(g)}\right) = \log\left(\frac{\lambda_1(g)}{\lambda_n(g)}\right).\]
This equals the infimum of the Hilbert metric lengths of the associated closed curves in $\tilde \Sigma_{\tilde E}/\hat h(\pi_1(\tilde E))$
as first shown by Kuiper. (See \cite{Benasym} for example.)

We recall the results in \cite{Benasym} and \cite{Ben5}.
\begin{definition} \label{defn-pro}
Each element $g \in \SLnp$ 
\begin{itemize}
\item that has the largest and smallest norms of the eigenvalues 
which are distinct and
\item the largest or the smallest norm correspond to the eigenvectors with positive eigenvalues 
(and do not correspond to the eigenvectors of negative ones)
respectively
\end{itemize} 
is said to be {\em bi-semiproximal}.
Each element $g \in \SLnp$ 
\begin{itemize}
\item that has the largest and smallest norms of the eigenvalues 
which are distinct and of multiplicity one and
\item each of the largest or the smallest norm corresponds to an eigenvector of positive eigenvalue 
unique up to scalars respectively (and does not correspond to an eigenvector of negative eigenvalue)
\end{itemize} 
is said to be {\em biproximal}. 
\end{definition}  
Note also 
when 
$\Gamma$ acts on a properly convex domain divisibly, 
 an element is {\em semiproximal} if and only if it is bi-semiproximal (see \cite{Ben2}).
Since $\tilde \Sigma_{\tilde E}$ is properly convex, all infinite order elements of $\hat h(\pi_1(\tilde E))$ are bi-semiproximal
and a finite index subgroup has only bi-semiproximal elements and the identity.



When $\pi_1(\tilde E)$ is hyperbolic, 
all infinite order elements of $\hat h(\pi_1(\tilde E))$ are biproximal
and a finite index subgroup has only biproximal elements and the identity.
When $\bGamma_{\tilde E}$ is a hyperbolic group, 
an element is {\em proximal} if and only if it is biproximal.  


Assume that $\bGamma_{\tilde E}$ is hyperbolic. 
Suppose that $g \in \bGamma_{\tilde E}$ is proximal. 
We define 
\begin{equation}\label{eqn-alphabetag}
\alpha_g := \frac{\log \tilde \lambda_1(g)- \log \tilde \lambda_n(g)}{\log \tilde \lambda_1(g) - \log \tilde \lambda_{n-1}(g)}, 
\beta_g :=   \frac{\log \tilde \lambda_1(g)- \log \tilde \lambda_n(g)}{\log \tilde \lambda_1(g) - \log \tilde \lambda_{2}(g)},
\end{equation} 
and denote by $\bGamma_{\tilde E}^p$ the set of proximal elements. We define
\[\beta_{\bGamma_{\tilde E}} := \sup_{g \in \bGamma_{\tilde E}^p} \beta_g, 
\alpha_{\bGamma_{\tilde E}} := \inf_{g\in \bGamma_{\tilde E}^p} \alpha_g. \]
Proposition 20 of Guichard \cite{Guichard} shows that  
we have 
\begin{equation}\label{eqn-betabound}
1 < \alpha_{\tilde \Sigma_{\tilde E}} \leq \alpha_\Gamma \leq 2 \leq \beta_\Gamma \leq \beta_{\tilde \Sigma_{\tilde E}} < \infty 
\end{equation}
for constants $\alpha_{\tilde \Sigma_{\tilde E}}$ and $\beta_{\tilde \Sigma_{\tilde E}}$ depending only on $\tilde \Sigma_{\tilde E}$
since $\tilde \Sigma_{\tilde E}$ is properly and strictly convex.

Here, it follows that $\alpha_{\bGamma_{\tilde E}}, \beta_{\bGamma_{\tilde E}}$
depends on $\hat h$, and they form positive-valued functions on the union of components of  
\[\Hom(\pi_1(\tilde E), \SLnp)/\SLnp\] 
consisting of convex divisible representations
with the algebraic convergence topology as given by Benoist \cite{Ben3}. 




\begin{theorem}\label{thm-eignlem} 
Let $\orb$ be a strongly tame convex real projective orbifold. 
Let $\tilde E$ be a properly convex p-R-end of the universal cover $\torb$, $\torb \subset \SI^n$, $n \geq 2$.
Let $\bGamma_{\tilde E}$ be a hyperbolic group. 
Then 
\[ \frac{1}{n}\left(1+ \frac{n-2}{\beta_{\bGamma_{\tilde E}}}\right) \leng(g)
 \leq \log \tilde \lambda_1(g)   \leq  \frac{1}{n}\left(1+ \frac{n-2}{\alpha_{\bGamma_{\tilde E}}}\right) \leng(g)\]
for every proximal element $g \in \hat h(\pi_1(\tilde E))$.
\end{theorem}
\begin{proof} 
Since there is a biproximal subgroup of finite index, we concentrate on biproximal elements only.
We obtain from above that 
\[ \frac{\log \frac{\tilde \lambda_1(g)}{\tilde \lambda_n(g)}}{\log \frac{\tilde \lambda_1(g)}{\tilde \lambda_2(g)}} 
\leq \beta_{\tilde \Sigma_{\tilde E}}.\] 
We deduce that 
\begin{equation}\label{eqn-eigratio} 
\frac{\tilde \lambda_1(g)}{\tilde \lambda_2(g)} \geq \left( \frac{\lambda_1(g)}{\lambda_n(g)} \right)^{1/\beta_{\tilde \Sigma_{\tilde E}}}
=  \left( \frac{\tilde \lambda_1(g)}{\tilde \lambda_n(g)} \right)^{1/\beta_{\Omega}} = \exp\left(\frac{\leng(g)}{\beta_{\tilde \Sigma_{\tilde E}}}\right).
\end{equation}
Since we have $\tilde \lambda_i \leq \tilde \lambda_2 $ for $i\geq 2$, we obtain
\begin{equation}\label{eqn-betab} 
\frac{\tilde \lambda_1(g)}{\tilde \lambda_i(g)} \geq \left( \frac{\lambda_1}{\lambda_n} \right)^{1/\beta_{\tilde \Sigma_{\tilde E}}}
\end{equation}
and since $\tilde \lambda_1 \cdots \tilde \lambda_n = 1$, 
we have 
\[ \tilde \lambda_1(g)^n = \frac{\tilde \lambda_1(g)}{\tilde \lambda_2(g)} \cdots  \frac{\tilde \lambda_1(g)}{\tilde \lambda_{n-1}(g)}
 \frac{\tilde \lambda_1(g)}{\tilde \lambda_n(g)} \geq \left(  \frac{\tilde \lambda_1(g)}{\tilde \lambda_n(g)} \right)^{\frac{n-2}{\beta} + 1}.\] 
 We obtain 
 \begin{equation}\label{eqn-betabd}
  \log \tilde \lambda_1(g) \geq \frac{1}{n}\left(1+ \frac{n-2}{\beta_{\bGamma_{\tilde E}}}\right) \leng(g).
  \end{equation}
By similar reasoning, we also obtain 
\begin{equation}\label{eqn-alphabd}
\log \tilde \lambda_1(g) \leq \frac{1}{n}\left(1+ \frac{n-2}{\alpha_{\bGamma_{\tilde E}}}\right) \leng(g).
\end{equation} 

\end{proof}

\begin{remark} \label{rem-eigenlem}
Under the assumption of Theorem \ref{thm-eignlem}, if we do not assume that $\pi_1(\tilde E)$ is hyperbolic, then 
we obtain 
\[ \frac{1}{n} \leng(g) \leq \log \tilde \lambda_1(g)   \leq  \frac{n-1}{n} \leng(g)\]
for every semiproximal element $g \in \hat h(\pi_1(\tilde E))$.
\end{remark} 
\begin{proof} 
Let $\tilde \lambda_i(g)$ denote the norms of $\hat h(g)$ for $i=1, 2, \dots, n$. 
\[\log \tilde \lambda_1(g) \geq  \ldots \geq \log \tilde \lambda_n(g), 
\log \tilde \lambda_1(g)  + \cdots + \log \tilde \lambda_n(g) = 0\] 
hold.
We deduce 
\begin{alignat}{3} 
\log \tilde \lambda_n(g) &=& -\log \lambda_1 - \cdots - \log \tilde \lambda_{n-1}(g) \nonumber \\
& \geq & -(n-1) \log \tilde \lambda_1 \nonumber\\ 
\log \tilde \lambda_1(g) & \geq & -\frac{1}{n-1} \log \tilde \lambda_n(g) \nonumber\\ 
\left(1+ \frac{1}{n-1}\right) \log \tilde \lambda_1(g) & \geq & \frac{1}{n-1} \log \frac{\tilde \lambda_1(g)}{\tilde \lambda_n(g)}\nonumber \\ 
\log \tilde \lambda_1(g) & \geq & \frac{1}{n} \leng(g).
\end{alignat}
We also deduce 
\begin{alignat}{3} 
-\log \tilde \lambda_1(g) & = & \log \tilde \lambda_2(g) + \cdots + \log \tilde \lambda_{n}(g) \nonumber \\
 & \geq & (n-1) \log \tilde \lambda_{n}(g) \nonumber \\ 
-(n-1) \log \tilde \lambda_{n}(g) & \geq & \log \tilde \lambda_1(g) \nonumber \\ 
(n-1) \log \frac{\tilde \lambda_1(g)}{\tilde \lambda_{n}(g)} & \geq & n \log \tilde \lambda_1(g) \nonumber \\ 
\frac{n-1}{n} \leng(g) & \geq & \log \tilde \lambda_1(g).
\end{alignat} 
\end{proof}

\begin{remark}
We cannot show that the middle-eigenvalue condition implies 
the uniform middle-eigenvalue condition. This could be false.
For example, we  could obtain a sequence of elements $g_i \in \Gamma$ so that 
$\lambda_1(g_i)/ \lambda_{\bv_{\tilde E}}(g_i) \ra 1$ while $\Gamma$ satisfies the middle-eigenvalue 
condition. Certainly, we could have an element $g$ where 
$\lambda_1(g) = \lambda_{\bv_{\tilde E}}(g)$. 
However, even if there is no such element, we might still have 
a counter-example. 
For example, suppose that we might have 
\[\frac{\log \left(\frac{\lambda_1(g_i)}{\lambda_{\bv_{\tilde E}}(g_i)}\right)}{\leng(g)} \ra 0.\] 
(If the orbifold were to be homotopy-equivalent to the end orbifold, this could happen
by changing $\lambda_{v}$ considered as a homomorphism 
$\pi_{1}(\Sigma_{\tilde E}) \ra \bR^{+}$.   
Such assignments are not really understood globally
but see Benoist \cite{Benasym}. Also, an analogous phenomenon seems to happen 
with the Margulis space-time and diffused Margulis invariants as investigated by 
Charette, Drumm, Goldman, Labourie, and Margulis 
recently.  See \cite{GLM})
\end{remark}


\subsubsection{The uniform middle-eigenvalue conditions and the orbits.} \label{subsub:umecorbit}

Let $\tilde E$ be a properly convex p-R-end of the universal cover $\torb$ of 
 a properly convex real projective strongly-tame orbifold $\orb$. 
Assume that $\bGamma_{\tilde E}$ satisfies the uniform middle-eigenvalue condition. 
There exists a $\bGamma_{\tilde E}$-invariant convex set $K$ distanced from $\{\bv_{\tilde E}, \bv_{\tilde E-}\}$
by Theorem \ref{thm-distanced}. 
For the corresponding tube ${\mathcal T}_{\bv_{\tilde E}}$, $K \cap \Bd {\mathcal T}_{\bv_{\tilde E}}$ is a compact 
subset distanced from $\{\bv_{\tilde E}, \bv_{\tilde E-}\}$.
Let $C_1$ be the convex hull of $K$ in the tube ${\mathcal T}_{\bf_{\tilde E}}$ obtained by 
Theorem \ref{thm-distanced}.
Then $C_1$ is a $\bGamma_{\tilde E}$-invariant 
distanced subset of ${\mathcal T}_{\bv_{\tilde E}}$. 

Also, $K \cap \Bd {\mathcal T}_{\bv_{\tilde E}}$ contains all attracting and repelling 
fixed points of $\gamma \in \bGamma_{\tilde E}$ by invariance and the middle-eigenvalue condition. 

Recall that a {\em geometric limit} of a sequence of subsets of $\SI^{n}$ is defined by
the Hausdorff distance $\bdd_{\SI^{n}}^{H}$ using the standard Riemannian metric $\bdd_{\SI^{n}}$ of $\SI^{n}$.
(See Definition \ref{I-defn-Haus} of \cite{EDC1} for detail.)

\begin{lemma}\label{lem-attracting} 
Let $\orb$ be a strongly tame properly convex real projective orbifold. 
Let $\tilde E$ be a properly convex p-R-end. 
Assume that $\bGamma_{\tilde E}$ is admissible
 and satisfies the uniform middle eigenvalue conditions. 
\begin{itemize}
\item Suppose that $\gamma_i$ is a sequence of elements of $\bGamma_{\tilde E}$ acting on ${\mathcal T}_{\bv_{\tilde E}}$. 
\item The sequence of attracting fixed points $a_i$ and the sequence of  repelling fixed points $b_i$ are so that 
$a_i \ra a_\infty$ and $b_i \ra b_\infty$ where $a_\infty, b_\infty$ are not in $\{ \bv_{\tilde E}, \bv_{\tilde E-}\}$
for $a_\infty \ne b_\infty$. 
\item Suppose that the sequence $\{\lambda_i\}$ of eigenvalues where 
$\lambda_i$ corresponds to $a_i$ converges to $+\infty$. 
\end{itemize} 
Let 
\[M := {\mathcal T}_{\bv_{\tilde E}} - \clo(\bigcup_{i=1}^\infty \ovl{b_i\bv_{\tilde E}} \cup \ovl{b_i\bv_{\tilde E-}}). \]
Then the point $a_\infty$ is the geometric limit of $\{\gamma_i(K)\}$ for any compact subset $K \subset M$. 
\end{lemma} 
\begin{proof} 
Let $k_{i}$ be the inverse of the factor 
\[\min \left\{\frac{\tilde \lambda_1(\gamma_i)}{\tilde \lambda_2(\gamma_i)}, 
\frac{\tilde \lambda_1(\gamma_i)}{\lambda_{\bv_{\tilde E}}(\gamma_i)^{\frac{n+1}{n}}}
= \frac{\lambda_{1}(\gamma_{i})}{\lambda_{\bv_{\tilde E}}(\gamma_i)}
\right\}.\]
Then $k_i \ra 0$ by the uniform middle eigenvalue condition and equation \eqref{eqn-eigratio}. 

There exists a totally geodesic sphere $\SI^{n-1}_i$ at $b_i$ supporting ${\mathcal T}_{\bv_{\tilde E}}$. 
$a_i$ is uniformly bounded away from $\SI^{n-1}_i$ for $i$ sufficiently large.
$\SI^{n-1}_i$ bounds an open hemisphere $H_i$ containing $a_i$ where $a_i$ 
is the attracting fixed point so that for a Euclidean metric $d_{E, i}$, 
$\gamma_i| H_i: H_i \ra H_i$ we have 
\begin{equation}\label{eqn-kcont}
d_{E, i}(\gamma_{i}(x), \gamma_{i}(y)) \leq k_{i} d_{E, i}(x, y), x, y \in H_{i}.
\end{equation}
Note that $\{\clo(H_i)\}$ converges geometrically to $\clo(H)$ for an open hemisphere containing $a$ in
the interior. 

Actually, we can choose a Euclidean metric $d_{E, i}$ on $H_i^o$ 
so that $\{d_{E, i}| J \times J \}$ is uniformly convergent for any compact subset $J$ of 
$H_\infty$.
Hence there exists a uniform positive constant $C'$ so that 
\begin{equation}\label{eqn-Cp} 
\bdd(a_i, K)  < C' d_{E_i}(a_i, K). 
\end{equation} 
provided $a_{i}, K \subset J$ and sufficiently large $i$. 

Since $\bGamma_{\tilde E}$ is hyperbolic, the domain $\Omega$ corresponding to ${\mathcal T}_{\bv_{\tilde E}}$
in $\SI^{n-1}_{\bv_{\tilde E}}$ is strictly convex. 
For any compact subset $K$ of $M$, the equation $K \subset M$ is equivalent to 
\[K \cap \clo(\bigcup_{i=1}^\infty \ovl{b_i\bv_{\tilde E}} \cup \ovl{b_i\bv_{\tilde E-}}) = \emp.\]
Since the boundary sphere $\Bd H_{\infty}$ meets $\clo({\mathcal T}_{\bv_{\tilde E}})$ in this set only
by the strict convexity of $\Omega$, we obtain $K \cap \Bd H_{\infty} = \emp$. And
 $K \subset H_\infty$ since $\clo({\mathcal T}_{\bv_{\tilde E}}) \subset \clo(H_{\infty})$. 

We have $\bdd(K, \Bd H_{\infty}) > \eps_{0}$ for $\eps_{0}> 0$. 
Thus, the distance $\bdd(K, \Bd H_i)$ is uniformly bounded by a constant $\delta$. 
$\bdd(K, \Bd H_i) > \delta$ implies that 
$d_{E_i}(a_i, K) \leq C/\delta$ for a positive constant $C> 0$
Acting by $g_i$, we obtain 
$d_{E_i}(g_i(K), a_i) \leq k_i C/\delta$ by equation \eqref{eqn-kcont}, which implies  
$\bdd(g_i(K_i), a_i) \leq C' k_i C/\delta$ by equation \eqref{eqn-Cp}.
Since $\{k_i\} \ra 0$ and $\{a_i\} \ra a$ imply that $\{g_i(K)\}$ geometrically converges to $a$. 
 \end{proof}


For the following, $\bGamma_{\tilde E}$ can be virtually factorable.
\begin{proposition}\label{prop-orbit}
Let $\orb$ be a strongly tame properly convex real projective orbifold. 
Let $\tilde E$ be a properly convex p-R-end. 
Assume that $\bGamma_{\tilde E}$ satisfies the uniform middle eigenvalue condition. 
Let $\bv_{\tilde E}$ be the R-end vertex
and  $z \in {\mathcal T}^o_{\bv_{\tilde E}}$. 
Then a $\bGamma_{\tilde E}$-invariant distanced compact set $K$ in
$\clo({\mathcal T}_{\bv_{\tilde E}}) - \{\bv_{\tilde E}, \bv_{\tilde E -}\}$
satisfies the following properties\,{\rm :} 
\begin{itemize} 
\item[{\rm (i)}] $K^{b} := K \cap \partial {\mathcal T}_{\bv_{\tilde E}}$ equals the limit set of the orbit of $z$. 
$K^{b}$ is uniquely determined. In fact $K^{b}$ is the closure of the set 
of attracting fixed points of $\bGamma_{\tilde E}$ in $\partial {\mathcal T}_{\bv_{\tilde E}}$. 
\item[{\rm (ii)}] For each segment $s$ in $\partial {\mathcal T}_{\bv_{\tilde E}}$ with 
an endpoint $\bv_{\tilde E}$, the great segment containing $s$ meets $K^{b}$ 
at a point other than $\bv_{\tilde E}, \bv_{\tilde E-}$. 
That is, there is a one-to-one correspondence between $\Bd \Sigma_{\tilde E}$
and $K^{b}$. 
\item[{\rm (iii)}] $K^{b}$ is homeomorphic to $\SI^{n-2}$. 
\end{itemize} 
\end{proposition} 
\begin{proof} 
Let $K$ be any given $\bGamma_{\tilde E}$-invariant distanced compact set in
$\clo({\mathcal T}_{\bv_{\tilde E}}) - \{\bv_{\tilde E}, \bv_{\tilde E -}\}$ 
by Theorem \ref{thm-distanced}.

Consider first when $\bGamma_{\tilde E}$ is not virtually factorable and hyperbolic.
Let $z \in {\mathcal T}^{o}_{\bv_{\tilde E}}  - \{\bv_{\tilde E}, \bv_{\tilde E -}\}$.
Let $[z]$ denote the corresponding element in $\Sigma_{\tilde E}$. 
Let $\{\gamma_i\}$ be any sequence in $\bGamma_{\tilde E}$ 
so that the corresponding sequence $\{\gamma_i([z])\}$
in $\Sigma_{\tilde E} \subset \SI^{n-1}_{\bv_{\tilde E}}$ converges to a point $z'$ in 
$\Bd \Sigma_{\tilde E} \subset \SI^{n-1}_{\bv_{\tilde E}}$. 

Clearly, a fixed point of $g \in \bGamma_{\tilde E} -\{\Idd\}$ 
in $\Bd {\mathcal T}_{\bv_{\tilde E}}  - \{\bv_{\tilde E}, \bv_{\tilde E -}\}$ is in $K^{b}$
 since $g$ has at most one fixed point on each open segment in the boundary. 
We can assume that for the attracting fixed points $a_i$ and $r_i$ of $\gamma_i$, 
we have 
\[\{a_i\} \ra a, \{r_i\} \ra r \hbox{ for } a_i, r_i \in  K \]
where $a, r \in K^{b}$ by the closedness of $K^{b}$. 
Assume $a \ne r$ first.
By Lemma \ref{lem-attracting}, we have $\{\gamma_i(z)\} \ra a$ and hence the limit 
$z_\infty = a$. 

However, it could be that $a = r$. In this case, we choose $\gamma_0 \in \bGamma_{\tilde E}$
so that $\gamma_0(a) \ne r$. Then $\gamma_0\gamma_i$ has the attracting fixed point $a'_i$ 
so that we obtain $\{a'_i\} \ra \gamma_0(a)$ 
and repelling fixed points $r'_i$ so that $\{r'_i \}\ra r$ holds
by Lemma \ref{lem-gatt}.

Then as above $\{\gamma_0 \gamma_i(z) \} \ra \gamma_0(a)$
and we need to multiply by $\gamma_0^{-1}$ now
to show $\{\gamma_i(z) \} \ra a$. 
Thus, the limit set is contained in $K^{b}$. 

Conversely, an attracting fixed point of $g \in \bGamma_{\tilde E}$ must be in $K^{b}$ since $K$ is $\bGamma_{\tilde E}$-invariant. 
The set of attracting fixed point of $g$ in $\clo(\tilde \Sigma_{\tilde E}) \subset \SI^{n-1}$ is dense 
by \cite{Ben1}. Thus, by density, the closure $K'$ of the set of attracting fixed point of $\bGamma_{\tilde E}$ is 
a compact subset of $K^{b}$. 

Since $\bGamma_{\tilde E}$ is hyperbolic,  any point $y$ of $\Bd \tilde \Sigma_{\tilde E} \subset \SI^{n-1}_{\bv_{\tilde E}}$ 
is a limit point of some sequence $\{g_{i}(x)\}$ for $x \in \tilde \Sigma_{\tilde E}$.  
Thus, at least one point in the segment $l_{y}$ containing $y$ with endpoints $\bv_{\tilde E}$ and $\bv_{\tilde E -}$ 
is a limit point of some subsequence of $\{g_{i}(z)\}$ by Lemma \ref{lem-attracting}. 
Thus, $l_{y} \cap K' \ne \emp$. 

Also, $l_{y} \cap K^{b}$ is a unique since otherwise we can apply $\{g_{i}^{-1}\}$ and obtain that $K^{b}$ is not 
uniformly bounded away from $\bv_{\tilde E}$ and $\bv_{\tilde E-}$ 
using the argument of the proof of Lemma \ref{lem-attracting} 
in reverse. Thus, $K'= K^{b}$, and 
 (i), (ii), and (iii) hold for $K^{b}$. 




Suppose that $\bGamma_E$ is virtually factorable. 
Then a totally geodesic hyperspace $H$ is disjoint from $\{\bv_{\tilde E}, \bv_{\tilde E-}\}$ 
and meets $\torb$ by the proof of Theorem \ref{thm-distanced}. 
Then consider any sequence $g_i$ so that 
$g_i(x) \ra x_0$ for a point $x \in  {\mathcal T}^o_{\bv_{\tilde E}}$ and $x_{0} \in {\mathcal T}_{\bv_{\tilde E}}$.
Let $x'$ denote the corresponding point of $\tilde \Sigma_{\tilde E}$ for $x$. 
Then $g_i(x')$ converges to a point $y \in \SI^{n-1}_{\bv_{\tilde E}}$. 
Let $\vec x\in \bR^{n+1}$ be the vector in the direction of $x$. 
We write \[\vec x = \vec x_E + \vec x_H\] where $\vec x_H$ is in the direction of $H$ and $\vec x_E$ is in the direction of $\bv_{\tilde E}$. 
By the uniform middle eigenvalue condition, we obtain
$g_i(x) \ra x_{0}$ for $x_{0} \in H$. Hence, $x_{0} \in H \cap K$. 
Thus, every limit point of an orbit of $x$ is in $K^{b}$. 

Each point of $\Bd \tilde \Sigma_{\tilde E} \subset \SI^{n-1}_{\bv_{\tilde E}}$ is 
a limit point of an orbit of $\bGamma_{\tilde E}$ 
since $\bZ^{l_{0}-1}$ is cocompact lattice in $\bR^{l_{0}-1}$ and 
$\Gamma_{i}$ acts cocompactly on $K_{i}$. 
Conversely, we can easily show that $H\cap K^{b}$ is in the limit set  
and $H\cap K^{b}= K^{b}$. 
\end{proof}


\begin{lemma} \label{lem-gatt}
Let $\{g_i\}$ be a sequence of projective automorphisms 
acting on a strictly convex domain $\Omega$ in $\SI^n$ {\rm (}resp. $\bR P^n${\rm ).} 
Suppose that the sequence of attracting fixed points 
$\{a_i \in \Bd \Omega\} \ra a$ and the sequence of 
repelling fixed points $\{r_i \in \Bd \Omega\} \ra r$. 
Assume that the corresponding sequence of eigenvalues of 
$a_i$ limits to $+\infty$ and that of $r_i$ limits to $0$. 
Let $g$ be any projective automorphism of $\Omega$. 
Then $\{gg_i\}$ has the sequence of attracting fixed points 
$\{a'_i\}$ converging to $g(a)$ and the sequence of repelling 
fixed points converging to $r$. 
\end{lemma}
\begin{proof}
Recall that $g$ is a quasi-isometry. 
Given $\eps >0$ and 
a compact ball $B$ disjoint from a ball around $r$, 
we obtain that $g g_{i}(B)$ is in a ball of radius $\eps$ of $g(a)$
for sufficiently large $i$. 
For a choice of $B$ and  a sufficiently large $i$, 
we obtain $g g_{i}(B) \subset B^{o}$. 
Since $g g_{i}(B) \subset B^{o}$, we obtain 
\[(g g_{i})^{n}(B) \subset (g g_{i})^{m}(B)^{o} \hbox{ for } n > m\] by induction, 
There exists an attracting fixed point $a'_{i}$ of $g g_{i}$ in $g g_{i}(B)$. 
Since the diameter of $g g_{i}(B)$ is converging to $0$, we obtain that 
$\{a'_{i}\} \ra g(a)$. 

Also, given $\eps > 0$ and a compact ball $B$ disjoint from a ball around $g(a)$, 
$g_{i}^{-1}g^{-1}(B)$ is in the ball of radius $\eps$ of $r$. 
Similarly to above, we obtain the needed conclusion. 

%
\end{proof}


\subsubsection{Convex cocompact actions of the p-end fundamental groups.} \label{subsec-redlens} 

In this section, we will prove Proposition \ref{prop-convhull2} obtaining a lens. 


For the following we require only the convexity of the orbifold. The following can be proved 
for the linear holonomy in $\SO(2, 1)$ using a different method as shown by D. Fried. 
For the following proposition, we can just assume convexity. 
\begin{proposition}\label{prop-convhull2} 
Let $\orb$ be a convex real projective orbifold. 
Assume that the universal cover $\torb$ is a subset of $\SI^n$.
\begin{itemize} 
\item Let $\bGamma_{\tilde E}$ be the admissible holonomy group of a properly convex p-R-end $\tilde E$.
\item Let ${\mathcal T}_{\bv_{\tilde E}}$ be an open tube corresponding to $R(\bv_{\tilde E})$.
\item Suppose that $\bGamma_{\tilde E}$ satisfies the uniform middle eigenvalue condition, 
and acts on a distanced compact convex set 
$K$ in $\clo({\mathcal T}_{\bv_{\tilde E}})$  
with $K \cap {\mathcal T}_{\bv_{\tilde E}} \subset \torb$. 
\end{itemize} 
Then any open  p-end-neighborhood containing 
$K \cap {\mathcal T}_{\bv_{\tilde E}}$ contains a lens-cone p-end-neighborhood of 
the p-R-end $\tilde E$. 
\end{proposition}
\begin{proof} 
By assumption, $\torb - K$ has two components since 
\begin{itemize}
\item either $K$ is in a totally geodesic hyperspace
meeting the rays from $\bv_{\tilde E}$ transversally, or 
\item $K^{o}\cap {\mathcal T}_{\bv_{\tilde E}} \ne \emp$ and  $K \cap {\mathcal T}_{\bv_{\tilde E}}$ 
has two boundary components closer and farther away from $\bv_{\tilde E}$. 
\end{itemize}
Let $K^b$ denote $\Bd {\mathcal T}_{\bv_{\tilde E}} \cap K$. 
Let us choose finitely many points $z_1, \dots, z_m \in U - K$ in the two components of $\torb - K$.

Proposition \ref{prop-orbit} shows that the orbits of $z_i$ for each $i$ accumulate to points of $K^b$ only. 
Hence, a totally geodesic hypersphere separates $\bv_{\tilde E}$ with these orbit points
and another one separates $\bv_{\tilde E-}$ and the orbit points. 
Define the convex hull $C_2:= C(\bGamma_{\tilde E}(\{z_1, \dots, z_m\}\cup K)$. 
Thus, $C_2$ is a compact convex set disjoint from $\bv_{\tilde E}$ and $\bv_{\tilde E-}$ 
and $C_2 \cap \Bd {\mathcal T}_{\bv_{\tilde E}} = K^{b}$. 

We need the following lemma.

\begin{lemma}\label{lem-push} 
Continuing to assume as above, 
let $U$ be a p-end-neighborhood of $\bv_{\tilde E}$ containing $K \cap {\mathcal T}_{\bv_{\tilde E}}$. 
Then we can choose 
$z_1, \dots, z_m$ in $U$ so that 
for $C_2:= C(\bGamma_{\tilde E}(\{z_1, \dots, z_m\}\cup K))$, 
$\Bd C_2 \cap \torb$ is disjoint from $K$
and $C_2 \subset U$. 
\end{lemma}
\begin{proof} 
First, suppose $K^{o}\ne \emp$. Then 
$(\Bd K \cap {\mathcal T}_{\bv_{\tilde E}})/\bGamma_{\tilde E}$ is diffeomorphic to 
a disjoint union of two copies of $\Sigma_{\tilde E}$. 
We can cover a compact fundamental domain of 
$\Bd K \cap {\mathcal T}_{\bv_{\tilde E}}$ by the interior of $n$-balls in $\torb$ that are convex hulls of finite sets
of points in $U$. Since $(K\cap \torb)/\bGamma_{\tilde E}$ is compact, 
there exists a positive lower bound of $\{d_{\torb}(x, \Bd U)| x \in K\}$.
Let $F$ denote the union of these finite sets. 
We can choose $\eps > 0$ so that 
the $\eps$-$d_{\torb}$-neighborhood $U'$ of $K$ in $\torb$ is a subset of $U$. 
Moreover $U'$ is convex by Lemma \ref{lem-nhbd} following \cite{CLT2}. 

The convex hull $C_{2}$ is a union of simplices with vertices in $\bGamma_{\tilde E}(F)$. 
If we choose $F$ to be in $U'$, then by convexity $C_{2}$ is in $U'$ as well. 

The disjointedness of $\Bd C_{2}$ from $K \cap  {\mathcal T}_{\bv_{\tilde E}}$ follows since 
the $\bGamma_{\tilde E}$-orbits of above balls cover $\Bd K \cap {\mathcal T}_{\bv_{\tilde E}}$. 

If $K^{o} = \emp$, then $K$ is in a hyperspace. The reasoning is similar to the above. 
\end{proof}


We continue:
\begin{lemma} \label{lem-infiniteline} 
Let $C$ be a $\bGamma_{\tilde E}$-invariant 
distanced compact convex set with boundary in ${\mathcal{T}}_{\tilde E}$
where $(C \cap {\mathcal{T}^o_{\tilde E}})/\bGamma_{\tilde E}$ is compact. 
There are two components $A$ and $B$ of
$\Bd C \cap {\mathcal T}^o_{\tilde E}$ meeting every great segment in ${\mathcal T}^o_{\tilde E}$. 
Suppose that $A$ {\rm (}resp. $B$\,{\rm )} are disjoint from $K$.
Then $A$ {\rm (}resp. $B$\,{\rm )} contains no line ending in $\Bd \torb$. 
\end{lemma} 
 \begin{proof} 
 It is enough to prove for $A$. 
Suppose that there exists 
a line $l$ in  $A$ ending at a point of $\Bd {\mathcal T}_{\bv_{\tilde E}}$.
 Assume $l \subset A$. 
The line $l$ project to a line $l'$ in ${\tilde E}$. 

Let $C_1 = C \cap {\mathcal T}_{\bv_{\tilde E}}$. 
Since $A/\bGamma_{\tilde E}$ and $B/\bGamma_{\tilde E}$ are both compact, 
and there exists a fibration $C_1/\bGamma_{\tilde E} \ra A/\bGamma_{\tilde E}$ 
induced from $C_1 \ra A$ using the foliation by great segments with endpoints $\bv_{\tilde E}, \bv_{\tilde E-}$. 

Since $A/\bGamma_{\tilde E}$ is compact,  
we choose a compact fundamental domain $F$ in $A$ and 
choose a sequence $\{x_i \in l\}$ 
whose image sequence in $l'$ converges to the endpoint of $l'$ in $\Bd \tilde \Sigma_{\tilde E}$. 
We choose $\gamma_i \in \bGamma_{\bv_{\tilde E}}$ so that $\gamma_i(x_i) \in F$ 
where $\{\gamma_i(\clo(l'))\}$ geometrically converges to a segment $l'_\infty$ with both endpoints in $\Bd \tilde \Sigma_{\tilde E}$. 
Hence, $\{\gamma_i(\clo(l))\}$ geometrically converges to a segment $l_\infty$ in $A$.
We can assume that for the endpoint $z$ of $l$ in $A$, $\gamma_i(z) $ converges to the endpoint $p_1$. 
Proposition \ref{prop-orbit} implies that the endpoint $p_1$ of $l_\infty$ is in $K^b:= K \cap \partial {\mathcal T}_{\tilde E}$. 
Let $t$ be the endpoint of $l$ not equal to $z$. Then $t \in A$. 
Since $\gamma_{i}$ is not a bounded sequence, $\gamma_i(t)$ converges to a point of $K^b$. 
Thus, both endpoints of $l_\infty$ are in $K^b$
and hence $l_\infty^{o} \subset K$ by the convexity of $K$.
However, $l \subset A$ implies that $l_\infty^{o} \subset A$. As $A$ is disjoint from $K$, 
this is a contradiction.  The similar conclusion holds for $B$. 
\end{proof} 

Since $A$ and analogously $B$ do not contain any geodesic ending at $\Bd \torb$, 
$\Bd C'_1 - \Bd {\mathcal T}_{\bv_{\tilde E}}$ is a union of compact $n-1$-dimensional simplices
meeting one another in strictly convex dihedral angles. 
By choosing $\{z_1, \dots, z_m\}$ sufficiently close to $\Bd C_1$, we may assume 
that $\Bd C'_1 - \Bd {\mathcal T}_{\bv_{\tilde E}}$ is in $\torb$. 
Now by smoothing $\Bd C'_1 - \Bd {\mathcal T}_{\bv_{\tilde E}}$,
we obtain two boundary components of a lens. 
This completes the proof of Proposition \ref{prop-convhull2}. 




\end{proof} 

\begin{proof}[{\sl Proof of Theorem \ref{thm-equiv}}.] 
First, we show that the uniform middle eigenvalue condition implies the existence of lens: 
Let ${\mathcal T}_{\bv_{\tilde E}}$ denote the tube domain with vertices $\bv_{\tilde E}$ and $\bv_{\tilde E -}$. 
Let $K^b$ denote the intersection of $\Bd  {\mathcal T}_{\bv_{\tilde E}}$ with the distanced compact 
$\bGamma_{\tilde E}$-invariant convex set $K$ by Theorem \ref{thm-distanced}. 



Let $C_{1}$ be the convex hull of $K$ and the finite number of points in the inner component of  
${\mathcal T}_{\bv_{\tilde E}} - K$ so that $\Bd C_{1} \cap {\mathcal T}_{\bv_{\tilde E}} $ is disjoint from $K$. 
By Lemma \ref{lem-infiniteline}, the component $\Bd C_1 \cap {\mathcal T}_{\bv_{\tilde E}} $ 
contains no line $l$ with endpoints $x, y$ in $K$, and 
hence can be isotopied to be strictly convex and smooth as above. 
Thus, a component of 
${\mathcal T}_{\bv_{\tilde E}}- \Bd C_{1}$ is a concave end neighborhood of $\tilde E$.

Now, we show the converse. 
Let $L$ be a lens of the lens-cone where $\bGamma_{\tilde E}$ acts on. 
There is a lower boundary component $B$ of $D \cap {\mathcal T}_{\bv_{\tilde E}}^o$ closer to $\bv_{\tilde E}$ 
that is strictly convex and transversal to every radial great segment from $\bv_{\tilde E}$ in $\tilde \Sigma_{\tilde E}$. 
$B$ bounds a properly convex domain $C$ in ${\mathcal T}_{\bv_{\tilde E}}^{o}$. Each radial rays from $\bv_{\tilde E}$ meets 
$B = \partial C$ transversally. $C/\bGamma_{\tilde E}$ is a properly convex real projective orbifold with boundary. 

Let $g \in \bGamma_{\tilde E}$ be an infinite order element. Then $g$ is bi-semi-proximal. 
Suppose that $\lambda_{\bv_{\tilde E}}(g) > \lambda_{1}(g)$ for any 
Then $g^{n}(x)$, $x \in C$ must accumulate to $\bv_{\tilde E}$ or $\bv_{\tilde E-}$, which contradicts 
the disjoint of $\clo(C)$ to $\bv_{\tilde E}$ and $\bv_{\tilde E-}$. 
If $\lambda_{\bv_{\tilde E}}(g) = \lambda_{1}(g)$, then let $l_{g}$ be the line in $\tilde \Sigma_{\tilde E}$
where $g$ acts on. Let $P_{g}$ be the $2$-dimensional subspace where $g$ acts on. 
Then $g$ acts on $\partial C \cap P_{g}$. Since it is  a strictly convex arc, $g$ cannot act on it 
with the eigenvalue condition.  
$\bGamma_{\tilde E}$ satisfies the middle-eigenvalue condition that $\lambda_1(g)/\lambda_{\bv_{\tilde E}}(g) > 1$ for every
infinite order $g$.


There is a map 
\[ \bGamma_{\tilde E} \ra H_1(\bGamma_{\tilde E}, \bR)\]
obtained by taking a homology class. 
The above map $g \ra \log \lambda_{\bv_{\tilde E}}(g)$ 
induces homomorphism
\[\Lambda^h: H_1(\bGamma_{\tilde E}, \bR) \ra \bR \]
that depends on the holonomy homomorphism $h$. 



If $\bGamma_{\tilde E}$ satisfies the middle-eigenvalue condition, then so does its factors. 
Suppose that $\bGamma_{\tilde E}$ does not satisfy the uniform middle-eigenvalue condition. 
Then there exists a sequence of elements $g_i$ so that 
\[\frac{\log\left(\frac{\lambda_1(g_i)}{\lambda_{\bv_{\tilde E}}(g_i)}\right)}{ \leng(g_i)} \ra 0 \hbox{ as } i \ra \infty.\] 

Note that we can change $h$ by only changing the homomorphism $\Lambda^h$
and still obtain a representation. 
Let $[g_{\infty}]$ denote a limit point of $\{[g_i]/\leng(g_i)\}$ in the space of currents on $\Sigma_{\tilde E}$. 
By a small change of $h$ so that $\Lambda^h(k)$ 
becomes strictly bigger at $[g_{\infty}]$.
From this, we obtain that 
\[\log\left(\frac{\lambda_1(g_{i})}{\lambda^{h}_{\bv_{\tilde E}}(g)}\right)< 0  \hbox{ for some } g_{i} \in \Gamma.\]
We know that a small perturbation of a lower 
boundary component of a generalized lens-shaped end remains 
strictly convex and in particular distanced
since we are changing the connection by a small amount 
which does not change the strict convexity by Proposition \ref{prop-koszul}. 
We obtain that $\lambda_1(g) < \lambda^h_{\bv_{\tilde E}}(g)$ for some $g$
for the largest eigenvalue $\lambda_1(g)$ of $h(g)$
and that $\lambda^h_{\bv_{\tilde E}}(g)$ at $\bv_{\tilde E}$.
This implies as above $\clo(C)$ contains $\bv_{\tilde E}$ or $\bv_{\tilde E-}$. 

By Proposition \ref{prop-lensP}, this is a contradiction.

\end{proof}



\subsection{The uniform middle-eigenvalue conditions and the lens-shaped ends.} \label{sub-umecl} 



A {\em radially foliated end-neighborhood system} of $\mathcal{O}$
is a collection of end-neighborhoods of $\mathcal{O}$ that is radially foliated 
and outside a compact suborbifold of $\mathcal{O}$ 
whose interior is isotopic to $\mathcal{O}$.  


\begin{definition} \label{defn-tri}
We say that $\mathcal{O}$ satisfies the {\em triangle condition} if 
for any fixed radially foliated end-neighborhood system of $\mathcal{O}$, 
every triangle $T \subset \clo(\torb)$, 
if $\partial T \subset  \Bd \torb, T^{o} \subset \torb$, then 
$T^{o}$ is a subset of a radially foliated p-end-neighborhood $U$ in $\tilde {\mathcal{O}}$. 
\end{definition}

In \cite{conv}, we will show that this condition is satisfied if $\pi_1(\mathcal{O})$ is 
relatively hyperbolic with respect to the end fundamental groups. 
 We will prove this in \cite{conv} since it is a global result and 
not a result on ends only. 


A {\em minimal} $\bGamma_{\tilde E}$-invariant distanced compact set is 
the smallest compact $\bGamma_{\tilde E}$-invariant distanced set in ${\mathcal T}_{\tilde E}$. 

\begin{theorem}\label{thm-equ} 
Let $\orb$ be a strongly tame  properly convex real projective orbifold. 
Assume the following conditions. 
\begin{itemize}
\item The universal cover $\torb$ is a subset of $\SI^n$ {\rm (}resp. in $\bR P^n${\rm ).}
\item The holonomy group $\bGamma$ is strongly irreducible. 
\end{itemize} 
Let $\bGamma_{\tilde E}$ be the admissible holonomy group of a properly convex R-end $\tilde E$.
Then the following are equivalent:
\begin{itemize}
\item[{\rm (i)}] $\bGamma_{\tilde E}$ is a generalized lens-type R-end.
\item[{\rm (ii)}] $\bGamma_{\tilde E}$ 
satisfies the uniform middle-eigenvalue condition.
\end{itemize}
Furthermore, if ${\mathcal{O}}$ furthermore 
satisfies the triangle condition or, alternatively, assume that $\tilde E$ is virtually factorable,
then the following are equivalent. 
\begin{itemize} 
\item $\bGamma_{\tilde E}$ is of lens-type if and only if 
 $\bGamma_{\tilde E}$ satisfies the uniform middle-eigenvalue condition.
\end{itemize}
\end{theorem}
\begin{proof} 
(ii) $\Rightarrow$ (i): This follows from Theorem \ref{thm-equiv} since 
we can intersect the lens with $\torb$ to obtain a generalized lens and generalized lens-cone from it. 
(i) $\Rightarrow$ (ii):   This follows from the proof of Theorem \ref{thm-equiv} since 
the proof only uses the strictly convex lower boundary component of the generalized lens. 

The final part follows by Lemma \ref{lem-genlens}.
\end{proof}

\begin{lemma} \label{lem-genlens} 
Suppose that $\mathcal{O}$ is a strongly tame properly convex real projective orbifold 
and satisfies the triangle condition or, alternatively, assume that a p-R-end $\tilde E$ is virtually factorable.  
Suppose that the holonomy group $\bGamma$ is strongly irreducible. 
Then the p-R-end $\tilde E$ is of generalized lens-type if and only if it is of lens-type. 
\end{lemma}
\begin{proof} 
If $\tilde E$ is virtually factorable, this follows by Theorem \ref{thm-redtot} (iv). 

Suppose that $\tilde E$ is not virtually factorable. 
Now assume the triangle condition. 

Thus, given a generalized lens $L$, let $L^{b}$ denote $\clo(L) \cap \clo({\mathcal T}_{\bv_{\tilde E}})$. 
We obtain the convex hull $K$ of $L^{b}$. $K$ is a subset of $\clo(L)$. 
The lower boundary component of $L$ is a smooth convex surface. 

Let $K_{1}$ be the outer component of $\Bd K \cap {\mathcal T}_{\bv_{\tilde E}}$. 
Suppose that $K_{1}$ meets $\Bd \torb$. 
$K_{1}$ is a union of the interior of simplices. 
By Lemma \ref{lem-simplexbd}, a simplex is either in $\Bd \torb$ or disjoint from it. 
Hence, there is a simplex $\sigma$ in $K_{1}\cap \Bd \torb$. 
Taking the convex hull of $\bv_{\tilde E}$ and an edge in $\sigma$, 
we obtain a triangle $T$ with $\partial T \subset \Bd \torb$ and $T^{o}\subset \torb$. 
This contradicts the triangle condition by Lemma \ref{lem-coneseq}. 
Thus, $K_{1}\subset \torb$. 
By Proposition \ref{prop-convhull2}, we obtain a lens-cone in $\torb$. 
\end{proof} 



\begin{lemma} \label{lem-coneseq} 
Suppose that $\mathcal{O}$ is a strongly tame properly convex real projective orbifold 
and satisfies the triangle condition. 
Then every triangle $T$ with $\partial T \subset \Bd \torb$ 
has no vertex equal to a p-R-end vertex.
\end{lemma} 
\begin{proof} 
Let $\bv_{\tilde E}$ be a p-end vertex. Choose a fixed radially foliated p-end-neighborhood system. 
Suppose that a triangle $T$ with $\partial T \subset \Bd \torb$ contains a vertex equal to a p-end vertex. 
Let $U$ be an inverse image of a radially foliated end-neighborhood in the end-neighborhood system,
and be a p-end neighborhood of a p-end $\tilde E$ with a p-end vertex $\bv_{\tilde E}$. 

Choose a maximal line $l$ in $T$ with endpoints $\bv_{\tilde E}$ and $w$
 in the interior of an edge of $T$ not containing $\bv_{\tilde E}$. 
Then this line has to pass a point of the boundary of $U$ and in $T^o$ by definition of 
the radial foliations of the p-end-neighborhoods. 
This implies that $T^o$ is not a subset of a p-end-neighborhood
and contradicts the assumption. 
\end{proof}





%


%

\section{The properties of lens-shaped ends.} \label{sec-lens}

One of the main results of this section is that a generalized lens-type end has 
a ``concave end-neighborhood'' that actually covers a p-end-neighborhood. 

First, we introduce a lemma on recurrences of geodesics relating it to the lens condition. 
Next, we discuss the properties of the lens-cone p-end neighborhoods when the p-end is nonfactorable. 
Then we discuss those for p-ends that are factorable. 
We end with some important lemmas.


A {\em trivial one-dimensional cone} is an open half-space in $\bR^1$ given by $x > 0$ or $x < 0$. 

Recall that $\pi_1(\tilde E)$ is an admissible group; $\pi_1(\tilde E)$ has  a finite index subgroup isomorphic to 
$\bZ^{k-1} \times \bGamma_1 \times \cdots \times \bGamma_k$ for some $k \geq 0$
where each $\bGamma_i$ is hyperbolic or trivial. 

Let us consider $\Sigma_{\tilde E}$ the real projective $(n-1)$-orbifold associated with $\tilde E$ and 
consider $\tilde \Sigma_{\tilde E}$ as a domain 
in $\SI^{n-1}_{\bv_{\tilde E}}$ and $h(\pi_1(\tilde E))$ induces $\hat h: \pi_1(\tilde E) \ra \SL_\pm(n, \bR)$ acting on $\tilde \Sigma_{\tilde E}$. 
We denote by $\Bd \tilde \Sigma_{\tilde E}$ the boundary of $\tilde \Sigma_{\tilde E}$ in $\SI^{n-1}$. 

\begin{definition} \label{defn-sl}
A (resp. generalized) lens-shaped p-R-end with the p-end vertex $\bv_{\tilde E}$ is {\em strictly} ({\em resp. generalized}) {\rm lens-shaped} 
if we can choose a (resp. generalized) lens domain $D$ 
\begin{itemize} 
\item with the top hypersurfaces $A$ and the bottom one $B$ 
so that 
\item each great open segment in $\SI^n$ from $\bv_{\tilde E}$ in the direction of 
$\Bd \tilde \Sigma_{\tilde E}$ meets 
$\clo(D) - A - B$ at a unique point. 
\end{itemize} 
\end{definition}
In this case, as a consequence $\clo(A) - A = \clo(B) - B$ and $\clo(A) \cup \clo(B) = \partial D$. 



\subsection{A lemma: recurrence and a lens} 
Given three sequences of projectively independent points $\{p^{(j)}_i\}$ with $j=1, 2, 3$ 
so that $\{p^{(j)}_i\} \ra p^{(j)}$ where $p^{(1)}, p^{(2)}, p^{(3)}$ are independent points in $\SI^n$. 
Then a simple matrix computation shows that 
a uniformly bounded sequence $\{r_i\}$ of elements of $\Aut(\SI^n)$ or $\PGL(n+1, \bR)$
acts so that $r_i(p^{(j)}_i) = p^{(j)}$ for every $i$ and $j=1, 2, 3$. 


A {\em convex arc} is an arc in a two-dimensional totally geodesic subspace where an arc projectively equivalent to 
a graph of a convex function $I \ra \bR$ for a connected interval in $\bR$. 

Find the tube $B_{\tilde E}$ with vertices $\bv_{\tilde E}$ and $\bv_{\tilde E-}$ corresponding to 
$\tilde \Sigma_{\tilde E}$. 

We first need the following technical lemmas on recurrent geodesics.
The main point of the lemma is that strict convexity of the boundary curves will force some facts about 
the endpoints being identical. 

\begin{figure}[h]
\centering
\includegraphics[height=5cm]{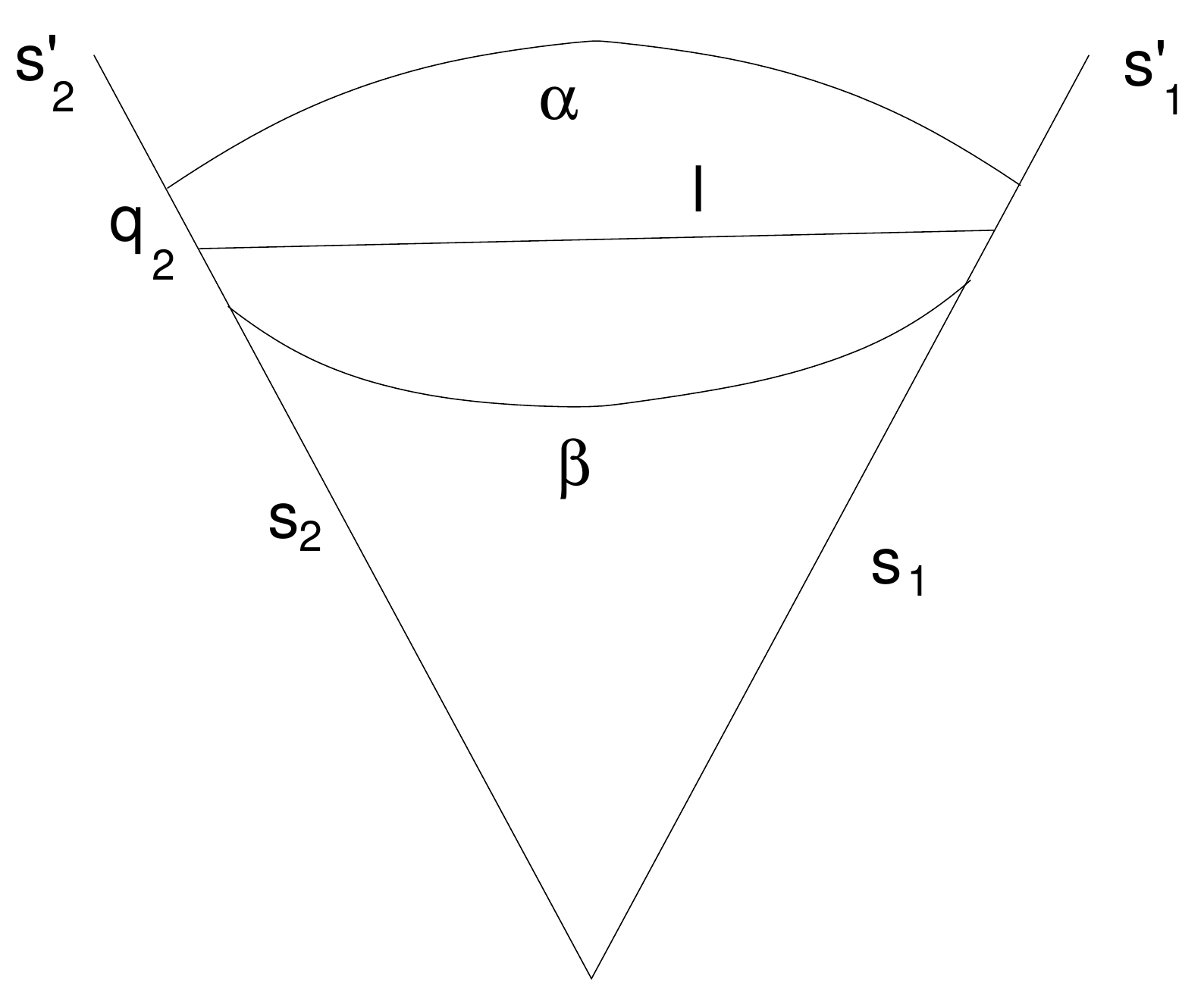}

\caption{The figure for Lemma \ref{lem-recurrent}. }

\label{fig:lenslem}
\end{figure}

\begin{lemma}[Recurrence and lens]\label{lem-recurrent} 
Let $\mathcal O$ be a strongly tame convex real projective $n$-orbifold. 
Suppose that $g_i \in \SLnp$ be a sequence of end fundamental group of a p-R-end $\tilde E$
and $l$ is a maximal segment in a generalized lens
with endpoints in $\Bd \tilde{\mathcal{O}}$.
{\rm (}See Figure \ref{fig:lenslem}.{\rm )}
Let $l'$ be the projected image of $l$ to the linking sphere $\SI^{n-1}_{\bv_{\tilde E}}$ of ${\bv_{\tilde E}}$. 
Let $g'_i$ denote the induced projective automorphisms on $\SI^{n-1}_{\bv_{\tilde E}}$. 
Suppose that $\{g'_i(l') \subset \tilde \Sigma_{\tilde E}\}$ converges geometrically to $l'$. 
Let $P$ be the $2$-dimensional subspace containing ${\bv_{\tilde E}}$ and $l$. 
Furthermore, we suppose that 
\begin{itemize}
\item  In $P$, $l$ is in the disk $D$ bounded by two segments $s_1$ and $s_2$ from ${\bv_{\tilde E}}$ and a compact convex curve $\alpha$  with 
endpoints $q_1$ and $q_2$  that are endpoints of $s_1$ and $s_2$ respectively. 
\item $\beta$ is another compact convex curve with $\beta^o \subset D^o$ and endpoints in $s_1-\{\bv_{\tilde E}\}$ and $s_2-\{\bv_{\tilde E}\}$
so that $\alpha$ and $\beta$ and parts of $s_1$ and $s_2$ bound a convex disk in $D$. 
\item There is a sequence of points $\tilde q_i \in \alpha$ converging to $q_1$ and $g_i(\tilde q_i) \in F$ 
for a fixed fundamental domain $F$ of $\tilde{\mathcal{O}}$.
\item The sequences $g_i(D)$, $g_i(\alpha)$, $g_i(\beta)$, $g_i(s_1)$, and $g_i(s_2)$ respectively geometrically 
converge to a disk $D$, arcs $\alpha$, $\beta$, segments $s_1,$ and $s_2$ respectively. 
\end{itemize} 
Then 
\begin{itemize} 
\item[(i)] If the endpoints of $\alpha$ and $\beta$ do not coincide at $s_1$, 
then $\alpha$ and $\beta$ must be geodesics from $q_2$. 
\item[(ii)] Suppose that the pairs of endpoints of $\alpha$ and $\beta$ coincide and they are distinct curves.
Then no segment in $\clo(\tilde{\mathcal{O}})$ contains $s_1$  properly.
\end{itemize} 
\end{lemma} 
\begin{proof}
By the geometric convergence conditions, 
we obtain a bounded sequence of elements $r_i \in \SLnp$ so that $r_i(g_i(s_1))=s_1$ and $r_i(g_i(s_2))=s_2$
and $\{r_i\} \ra \Idd$. Then $r_i \circ g_i| D$ is represented as an element of $\SL_\pm(3, \bR)$ in the subspace $P$ of dimension $2$. 
containing $D$. Using ${\bv_{\tilde E}}$ and $q_1$ and $q_2$ as standard basis points, $r_i\circ g_i$ is represented as a diagonal matrix.
Moreover $\{r_i \circ g_i(\alpha)\}$ is still converging to $\alpha$ as $\{r_i\} \ra \Idd$. 
(Thus, $r_i \circ g_i$ is diagonalizable with fixed points $q_1, q_2, \bv_{\tilde E}$.) 
Let $\lambda_i, \mu_i, \tau_i$ denote the diagonal matrix elements of $r_{i}\circ g_{i}$ where
\begin{itemize}
\item $\lambda_i$ is associated with $q_1$, 
\item $\mu_i$ is associated with ${\bv_{\tilde E}}$,  and 
\item $\tau_i$ is associated with $q_2$.
\end{itemize} 
Since $\{\tilde q_i\}$ is converging to $q_1$ and $r_{i}\circ g_i(\tilde q_i)$ is in a fixed compact set $\bigcup_{i} r_{i}(F)$, 
we obtain
\[\{\lambda_i/\tau_i\} \ra 0 \hbox{ as } i \ra \infty.\] 

(i) We have that $\{r_i\circ g_i(\beta)\}$ also converges to $\beta$. 
Suppose that the endpoint $\partial_{1} \beta$ of $\beta$ at $s_1$ is different from that of $\alpha$.
Since $r_i \circ g_i(\partial_1 \beta) \ra \partial_1\beta \ne q_{1}$, it follows that $\lambda_i/\mu_i \ra 1$.
In this case, from the diagonal matrix form of $r_{i } \circ g_{i}$, we obtain that $\beta$ has to be a geodesic from $q_2$ 
since $\{r_i \circ g_i(\beta)\} \ra \beta$. 
And so is $\alpha$. The similar argument holds for the case involving $s_2$.

(ii) If there is $c> 1$ such that $1/c < |\{\lambda_{i}/\mu_i\} |< c$, 
then $\beta$ and $\alpha$ have to be geodesics with distinct endpoints
from the matrix form of $r_{i}\circ g_{i}$ as in (i). This is a contradiction. 

Suppose that $\{\lambda_{i}/\mu_i\} \ra \infty$. 
Then any segment ending in $s_{1}^{o}$ and $s_{2}^{o}$ geometrically converges to 
the segment $\ovl{q_{1}q_{2}}$. Since $\beta$ is in a quadrilateral bounded by $s_{1}, s_{2}, \ovl{q_{1}q_{2}}$ 
and such a segment, $\{r_i \circ g_i(\beta)\}$ geometrically converges $\ovl{q_{1}q_{2}}$. This is a contradiction. 

Therefore, it must be that $\{\lambda_{i}/\mu_i\} \ra 0$. 
If a segment $s'_1$ in $\clo(\tilde \Omega)$ extends $s_1$,
then $\{r_i \circ g_i(s'_1)\}$ converges to a great segment
and so does $\{g_i(s'_1)\}$ as $i \ra \infty$ or $i \ra -\infty$. This contradicts the proper convexity of $\mathcal{O}$. 
\end{proof}

\subsection{The properties for a lens-cone in nonfactorable case} 
\begin{theorem}\label{thm-lensclass}
Let $\mathcal{O}$ be a strongly tame convex $n$-orbifold. 
Let $\tilde E$ be a p-R-end of $\tilde{\mathcal{O}}$ with a generalized lens p-end-neighborhood. 
Let $\bv_{\tilde E}$ be the p-end vertex. 
Assume that $\pi_1(\tilde E)$ is hyperbolic, i.e., virtually non-factorable. 
\begin{itemize} 
\item[{\rm (i)}] 
\begin{itemize}
\item $\Bd D -\partial D$ is independent of the choice of $D$.
\item $D$ is strictly {\rm (}resp. generalized{\rm )} lens-shaped. 
\item Each element $g \in \bGamma_{\tilde E}$ has an attracting 
fixed point in $\Bd D$ intersected with a great segment from $\bv_{\tilde E}$ in $\Bd \tilde \Sigma_{\tilde E}$. 
\item The set of attracting fixed points is dense in $\Bd D - A - B$ for the top and the bottom hypersurfaces $A$ and $B$. 
\end{itemize}
\item[{\rm (ii)}] 
\begin{itemize} 
\item Let $l$ be a segment $l \subset \Bd \tilde{\mathcal{O}}$ 
with $l^o \cap \clo(U) \neq \emp$ for  any concave p-end-neighborhood $U$ of $\bv_{\tilde E}$. 
Then $l$ is in the closure in $\clo(V)$ of every concave or proper p-end-neighborhood $V$ of $\bv_{\tilde E}$. 
\item The set $S(\bv_{\tilde E})$ of maximal segments from $\bv_{\tilde E}$ in $\clo(V)$ is independent of 
a concave or proper p-end neighborhood $V$,
\item \[\bigcup S(\bv_{\tilde E}) = \clo(V) \cap \Bd \tilde{\mathcal{O}}.\]
\end{itemize} 
\item[{\rm (iii)}] $S(g(\bv_{\tilde E}))=g(S(\bv_{\tilde E}))$ for $g \in \pi_1(\tilde E)$. 
\item[{\rm (iv)}] 
A concave p-end-neighborhood is a proper p-end-neighborhood. 
\item[{\rm (v)}] Assume that $w$ is 
the p-end vertex of a p-R-end with hyperbolic end-fundamental group.  
Then 
\[S^o(\bv_{\tilde E}) \cap S(w) = \emp \hbox{ or } S(\bv_{\tilde E})=S(w) \hbox{ {\rm (}with } \bv_{\tilde E} = w \hbox{{\rm )}}\] 
for  p-end vertices $\bv_{\tilde E}$ and $w$ where we defined $S^o(\bv_{\tilde E})$ to denote 
the relative interior of $\bigcup S(\bv_{\tilde E})$ in $\Bd \tilde{\mathcal{O}}$. 
\end{itemize}
\end{theorem}   
\begin{proof}  The proof is done for $\SI^n$ but the result implies the $\bR P^n$-version. 
Here the closure is independent of the ambient spaces. 

(i) By Fact 2.12 \cite{Ben3}, we obtain that $\pi_1(\tilde E)$ is virtually center free and 
acts irreducibly on a strictly convex domain in $\SI^{n-1}_{\bv_{\tilde E}}$
 by Theorem 1.1 of \cite{Ben2}. 

Let $C_{\tilde E}$ be a concave end. Since $\bGamma_{\tilde E}$ acts on $C_{\tilde E}$, 
$C_{\tilde E}$ is a component of the complement of a generalized lens domain $D$ in 
a generalized R-end by definition. 

We have a generalized lens domain $D$ with boundary components $A$ and $B$ transversal to the lines in $R_{\bv_{\tilde E}}(\torb)$. 
We can assume that $B$ is strictly concave and smooth as we have a concave end-neighborhood. 
$\bGamma_{\tilde E}$ acts on both $A$ and $B$. We define
 \[\partial_{1} A:=\clo(A) - A \hbox{ and } \partial_{1} B:= \clo(B) - B.\] 

By Theorem 1.2 of \cite{Ben1}, the geodesic flow on the real projective 
$(n-1)$-orbifold $\tilde \Sigma_{\tilde E}/\bGamma_{\tilde E}$ is topologically mixing, 
i.e., recurrent since $\bGamma_{\tilde E}$ is hyperbolic. 
Thus, each geodesic $l$ in $\tilde \Sigma_{\tilde E} \subset \SI^{n-1}_{\bv_{\tilde E}}$, 
we can find a sequence $\{g_i \in \bGamma_{\tilde E}\}$ 
that satisfies the conditions of Lemma \ref{lem-recurrent}. 
The two arcs in $\Bd D$ corresponding to $l$ share endpoints. 
Since this is true for all geodesics, we obtain $\partial_{1} A=\partial_{1} B$ and $A\cup B$ is dense in $\Bd D$. 
The strictness of $D$ also follows. 


Hence, $\partial D = \clo(A) \cup \clo(B)$. 
Thus, $\Bd D - \partial D$ is the closure of the set of the attracting and repelling fixed points of $h(\pi_1(\tilde E))$
since the set of fixed points is dense in $\partial_{1} A= \partial_{1} B$  by Theorem 1.1 of \cite{Ben1}. 
Therefore this set is independent of the choice of $D$. 


(ii) Consider any segment $l$ in $\Bd \tilde{\mathcal O}$ with $l^o$ 
meeting $\clo(U_1)$ for a concave p-end-neighborhood $U_1$ of $\bv_{\tilde E}$. 
Let $T$ be the open tube corresponding to $\tilde \Sigma_{\tilde E}$. 
Let $T_{1}$ be a component of $\Bd T - \partial_{1} B $ containing $\bv_{\tilde E}$.
Then $T_{1} \subset \clo(U_{1}) \cap \Bd \torb$ by the definition of concave p-end neighborhoods. 
In the closure of $U_{1}$, an endpoint of $l$ is in $T_{1}$. 
Then $l^{o} \subset \Bd T$ since  $l^{o}$ is tangent to $\partial T_{1} - \{\bv_{\tilde E}, \bv_{\tilde E-}\}$. 
For any convex segment $s$ from $\bv_{\tilde E}$ to any point of $l$
must be in $\Bd T$. By convexity of $\clo(\torb)$, we have $s \subset \clo(\torb)$. 
Thus, $s$ is in $\Bd \torb$ since $\Bd T \cap \clo(\torb) \subset \Bd \torb$. 
Therefore, the segment $l$ is contained in the union of segments in $\Bd \torb$ from $\bv_{\tilde E}$. 

We suppose that $l$ is a segment from 
$\bv_{\tilde E}$ containing a segment $l_0$ in $\clo(U_1)\cap \Bd \tilde{\mathcal O}$ from $\bv_{\tilde E}$, 
and we will show that $l$ is in $\clo(U_1) \cap \Bd \torb$. This will be sufficient to prove (ii). 
A point of $\Bd \tilde \Sigma_{\tilde E}$ is a p-end vertex of a recurrent geodesic 
by Lemma \ref{lem-redergodic}. 
$l^{o}$ contains a point $p$ of $\Bd D - A - B$ that
is in the direction of a p-end vertex of a recurrent geodesic $m$ in $\tilde \Sigma_{\tilde E}$. 
Lemma \ref{lem-recurrent} again applies. Thus, $l^o$ does not meet  $\Bd D - A - B$. 
%
 Thus, \[l \subset \clo(U_1)\cap \Bd \tilde{\mathcal O}.\]

Let $U'$ be any proper p-end-neighborhood associated with $\bv_{\tilde E}$. 
Let $s$ be a segment in $U'$ from $\bv_{\tilde E}$. 
Then since each $g \in \bGamma_{\tilde E}$ has an attracting fixed point and the repelling fixed point 
on $\Bd\clo(D) - A - B$, $\{g^i(s)\}$ converges  to an element of $S(\bv _{\tilde E})$. 
The set of the attracting and the repelling fixed points of elements of $\bGamma_{\tilde E}$ 
is dense in the directions of $\Bd \tilde \Sigma_{\tilde E}$. 
Thus, every segment of $S(\bv_{\tilde E})$ is in the closure $\clo(U')$. 
We have 
\[\bigcup S(\bv_{\tilde E}) \subset \clo(U') \cap \Bd \tilde{\mathcal O}.\]

We can form $S'(\bv_{\tilde E})$ as
the set of maximal segments from $\bv_{\tilde E}$ in $\clo(U') \cap \Bd \tilde{\mathcal O}$.
Then no segment $l$ in $S'(\bv_{\tilde E})$ has interior points in $\Bd D - A - B$ as above. 
 Thus, \[S(\bv_{\tilde E}) = S'(\bv_{\tilde E}).\] 
 Also, since every points of $\clo(U') \cap \Bd \torb$ has a segment in the direction of $\Bd \tilde \Sigma_{\tilde E}$, 
 we obtain 
 \[\bigcup S(\bv_{\tilde E}) = \clo(U') \cap \Bd \torb.\] 

(iii) By the proof above, 
we now characterize $S(\bv_{\tilde E})$ as the set of 
maximal segments in $\Bd \tilde{\mathcal O}$ from $\bv_{\tilde E}$ 
ending at points of  $\Bd D - A - B$.
 Since $g(D)$ is the generalized lens for the the generalized lens neighborhood $g(U)$ of $g(\bv_{\tilde E})$, 
we obtain $g(S(\bv_{\tilde E}))= S(g(\bv_{\tilde E}))$ for any p-end vertex $\bv_{\tilde E}$.

(iv) Given a concave-end-neighborhood $C_{\tilde E}$ of a p-end vertex $\bv_{\tilde E}$,
we show that \[g(C_{\tilde E})=C_{\tilde E} \hbox{ or } g(C_{\tilde E})\cap C_{\tilde E} = \emp \hbox{  for } g \in \bGamma:\]

Suppose that 
\[g (C_{\tilde E}) \cap C_{\tilde E} \ne \emp, g(C_{\tilde E}) \not \subset C_{\tilde E}, 
\hbox{ and } C_{\tilde E} \not\subset g(C_{\tilde E}).\]



Since $C_{\tilde E}$ is concave, 
each point $x$ of $\Bd C_{\tilde E} \cap \torb$ is contained in a supporting totally geodesic hypersurface 
$D$ so that 
\begin{itemize}
\item a component $C_{\tilde E, x}$ of $C_{\tilde E} - D$ is in $C_{\tilde E}$ 
where
\item $\clo(C_{\tilde E, x}) \ni \bv_{C_{\tilde E}}$ for the p-end vertex $\bv_{C_{\tilde E}}$ of $C_{\tilde E}$. 
\end{itemize} 
Similar statements hold for $g(C_{\tilde E})$. 

Since $g(C_{\tilde E}) \cap C_{\tilde E} \ne \emp$, and one is not a subset of the other, 
it follows that 
\[\Bd g(C_{\tilde E}) \cap C_{\tilde E} \ne \emp \hbox{ or } g(C_{\tilde E}) \cap \Bd C_{\tilde E} \ne \emp.\]
Then by above 
a set of form of $C_{\tilde E,x}$ and $g(C_{\tilde E, y})$, $x, y \in \Bd C_{\tilde E}$ meet at some boundary point of $C_{E,1}$.
Now, $\clo(C_{E, x})$ is the closure of a component $C_{x}$ of $\clo(\torb) - H$ for a separating hyperspace, 
$C_{x}\cap \Bd \torb$ is a union of lines in $S(\bv_{\tilde E})$.
Similar statements hold for $\clo(g(C_{\tilde E, y}))$, we obtain 
\[l^{o} \cap m^{o} \hbox{ for some } l\in S(\bv_{\tilde E}), m \in S(g(\bv_{\tilde E})) = g(S(\bv_{\tilde E})). \]
Suppose that $\bv_{\tilde E}\ne g(\bv_{\tilde E})$. 
Then $l^{o}$ must be inside 
$(\bigcup S(g(\bv_{\tilde E})))^{o}$ by (ii). 
Since $\tilde \Sigma_{\tilde E}$ is strictly convex,
no subinterval of $l$ projects to a nontrivial segment in $\Bd \tilde \Sigma_{\tilde E}$. 
Thus, $l$ must agree with a segment in $S(g(\bv_{\tilde E}))$ in an interval. 
By maximality $l$ agrees with a segment in $S(g(\bv_{\tilde E}))$ and have vertices $\bv_{\tilde E}$ and $g(\bv_{\tilde E})$. 
For any nearby segment $l'$ in $S(\bv_{\tilde E})$ to $l$, the
fact that $l'$ has vertices $\bv_{\tilde E}$ and $g(\bv_{\tilde E})$.   must be true also by the same reason. 
This implies a contradiction to the fact that $S(g(\bv_{\tilde E}))$ is a singleton. 
We conclude
$\bv_{\tilde E} = g(\bv_{\tilde E})$. 



Hence, $g \in \bGamma_{\tilde E}$, and thus, $C_{\tilde E} = g(C_{\tilde E})$ 
as $C_{\tilde E}$ is a concave neighborhood. Therefore, this is a contradiction. 
We obtain three possibilities  \[g (C_{\tilde E}) \cap C_{\tilde E} = \emp, g(C_{\tilde E}) \subset C_{\tilde E}
\hbox{ or } C_{\tilde E} \subset g(C_{\tilde E}).\]

In the last two cases, $\bv_{\tilde E} = g(\bv_{\tilde E})$ by considerations of maximal segments in 
$S(\bv_{\tilde E})$ in $\bigcup g(S_{\bv_{\tilde E}})$ since $\tilde \Sigma_{g(\tilde E)}$ is strictly convex. 
It follows that $g(C_{\tilde E}) = C_{\tilde E}$ since 
$g$ fixes $\bv_{\tilde E}$, i.e., $g \in \bGamma_{\tilde E}$. 
This implies that $C_{\tilde E}$ is a proper p-end-neighborhood. 

(v) If $S(\bv_{\tilde E})^o\cap S(w) \ne \emp$, then the above argument in (iv) applies with in this situation 
to show that $\bv_{\tilde E} = w$. 


\end{proof}



 

\begin{lemma}\label{lem-simplexbd}
Let $\orb$ be a strongly tame properly convex real projective orbifold. 
Suppose that $\orb$ is properly convex. 
Let $\sigma$ be a convex domain in $\clo(\torb) \cap P$ for a subspace $P$. 
Then either $\sigma \subset  \Bd \torb$ or $\sigma^o$ is in $\torb$. 
\end{lemma} 
\begin{proof} 
Suppose that $\sigma^o$ meets 
$\Bd \tilde {\mathcal{O}}$ and is not contained in it entirely.  
Since the complement of $\sigma^{o}\cap \Bd \torb$ is a relatively open set in $\sigma^{o}$, 
we can find a segment $s \subset \sigma^o$ with a point $z$ so that a component $s_1$ of $s-\{z\}$ 
is in $ \Bd \tilde {\mathcal{O}}$ and the other component $s_2$ is disjoint from it. 
We may perturb $s$ in the subspace containing $s$ and 
$\bv_{\tilde E}$ so that the new segment $s' \subset \clo(\torb)$ 
meets $\Bd \tilde {\mathcal {O}}$ only in its interior point. 
This contradicts the fact that $\tilde {\mathcal{O}}$ is convex by Theorem A.2 of \cite{psconv}. 
\end{proof}

\subsection{The properties of lens-cones for factorable case}

A group $G$ {\em divides} an open domain $\Omega$ if $\Omega/G$ is compact. 

\begin{lemma} \label{lem-redergodic} 
Let $\tilde E$ be a p-end that can be virtually factorable or not virtually factorable. 
Every point of $\Bd \tilde \Sigma_{\tilde E}$ is an end point of an oriented geodesic $l$ that is recurrent 
in that direction when projected to $\Sigma_{\tilde E}$. 
\end{lemma}
\begin{proof} 
We will prove for $\SI^{n}$-version but this implies the version for $\bR P^{n}$. 
Also, we discuss for the case when $\tilde E$ is a p-R-end. But the other case is similar. 
If $\pi_1(\tilde E)$ is a hyperbolic group, then the conclusion follows from Theorem 1.2 of \cite{Ben1}. 

We assumed that $\pi_1(\tilde E)$ is admissible. 
Let $D $, $D \subset \SI^{n-1}_{\bv_{\tilde E}}$, be a properly convex compact set so that $D^o = \tilde \Sigma_{\tilde E}$.
Then  as in Section \ref{I-sub-ben} of \cite{EDC1}, we obtain  
$D$ is a strict join $D_1 \ast \cdots \ast D_k$ for some $k$, $k \geq 2$
where the virtual center isomorphic to $\bZ^{k-1}$ acts trivially and each $D_{i}$ is a compact properly convex domain.
For any subset $J \subset \{1, \dots, k\}$, we denote by 
\[D_{J}:= \ast_{i\in J} D_{i}, \bZ^{J}:= \oplus_{i\in J} \bZ, \hbox{ and } \bR^{J}:= \oplus_{i\in J} \bR.\] 

Let $x \in \Bd D$. 
Then $x = [\sum_{i =1}^{k} \lambda_{i} x_{i}]$ for $[x_{i}] \in D_{i}$ and $\lambda_{i} \geq 0$. 
Let $J_{x}$ denote the set where $\lambda_{i} > 0$. $J_{x}$ is a proper subset of $\{1, \dots, k\}$. 
Let $J'_{x} \subset J_{x}$ denote the set of indices where $[x_{i}]$ is in the boundary of $D_{i}$. 
We choose a geodesic $l_{i}$ ending in $x_{i}$ in the positive direction for each $i \in J'_{x}$ so that 
$l_{i}$ projects to a recurrent geodesic in $D_{i}^{o}/\bGamma_{i}$ since $\bGamma_{i}$ is hyperbolic. 
Let $J''_{x} =\{1, \dots, k\} - J'_{x}$. Then we choose a geodesic $l$ 
in $D_{J''_{x}}$ ending at $[\sum_{i \in J_{x}\cap J''_{x}} \lambda_{i} x_{i}]$ in the positive direction
and at an interior point of $D_{J''_{x} - J_{x}}$. 
$l$ projects to a recurrent geodesic in $D_{J''_{x}}^{o}/\bZ^{J''_{x}}$
since $\bZ^{J''_{x}}$ is a lattice acting cocompactly on $\bR^{J''_{x}}$. 
Then we let $l_{i}$ for each $i\in J''_{x}$ to be 
the ones obtained by projection of $l$ to each subspace corresponding to $D_{i}$. 
Let $x_{i}$ denote the end point of $l_{i}$ for every $i=1,\dots, k$ in the positive direction. 
We lift $l_{i}$ for each $i$ to an affine line $\tilde l_{i}$ in $\bR^{n+1}$ 
with unit speed parameters and the vector direction $x_{i}$. 
Then we let $\hat l$ denote the affine geodesic obtained by 
$\hat l(t) = \sum_{i=1}^{k} \lambda_{i} \tilde l_{i}(t)$. 
The projection of $\hat l$ to $D$ gives us the desired recurrent geodesic passing $D^{o}$
since the factor groups commute with one another. The recurrence follows from 
the recurrence of each $l_{i}$. 


\end{proof}

\begin{theorem}\label{thm-redtot}
Let $\mathcal{O}$ be a strongly tame properly convex real projective $n$-orbifold. 
Suppose that 
\begin{itemize}
\item $\clo(\torb)$ is not a strict join, or  
\item the holonomy group $\bGamma$ is strongly irreducible. 
\end{itemize} 
Let $\tilde E$ be a p-R-end of the universal cover $\torb$, $\torb \subset \SI^n$ 
{\rm (}resp. $\subset \bR P^n${\rm ),} with a {\rm(}generalized{\rm )} lens p-end-neighborhood. 
Let $\bv_{\tilde E}$ be the p-end vertex  and $\tilde \Sigma_{\tilde E}$ the p-end domain of $\tilde E$. 
Suppose that the p-end fundamental group $\bGamma_{\tilde E}$ is admissible and factorable.
Then the following statements hold\,{\rm :}
\begin{itemize} 
\item[{\rm (i)}] For $\SI^{n-1}_{\bv_{\tilde E}}$, we obtain 
\begin{itemize}
\item[{\rm (i-1)}]  Under a finite-index subgroup of $\hat h(\pi_1(\tilde E))$, 
 $\bR^{n}$ splits into $V_1 \oplus \cdots \oplus V_{l_0}$ and $\tilde \Sigma_{\tilde E}$ is the quotient of the sum 
$C'_1+ \cdots + C'_{l_0}$ for properly convex or trivial one-dimensional cones $C'_i \subset V_i$ for $i=1, \dots, l_0$
\item[{\rm (i-2)}] The Zariski closure of a finite index subgroup of $\hat h(\pi_1(\tilde E))$ is isomorphic 
to the product $G = G_1 \times \cdots \times G_{l_0} \times \bR^{l_0-1}$ 
where $G_i$ is a semisimple subgroup of $\Aut({\mathcal S}(V_i))$ with identity components isomorphic to 
$\SO(\dim V_{i} -1, 1)$ or $\SL(\dim V_{i}, \bR)$. 
\item[{\rm (i-3)}] Let $D_i$ denote the image of $C'_i$ in $\SI^{n-1}_{\bv_{\tilde E}}$.
Each hyperbolic virtual factor group of $\pi_1(\tilde E)$ divides
exactly one $D_i$ and acts on trivially on $D_j$ for $j \ne i$.
\item[{\rm (i-4)}] A finite index subgroup of 
$\pi_1(\tilde E)$ has a rank $l_0-1$ free abelian group center corresponding to $\bZ^{l_0-1}$ in $\bR^{l_0-1}$.
\end{itemize}
\item[{\rm (ii)}] $g$ in the center is diagonalizable with positive eigenvalues. 
For a nonidentity element $g$  in the center, the eigenvalue $\lambda_{\bv_{\tilde E}}$ 
of $g$ at ${\bv_{\tilde E}}$ is strictly between its largest norm and smallest norm eigenvalues. 
\item[{\rm (iii)}] The p-R-end is  totally geodesic. 
$D_i \subset \SI^{n-1}_{\bv_{\tilde E}}$ 
is projectively diffeomorphic by the projection $\Pi_{\bv_{\tilde E}}$ 
to totally geodesic convex domain $D'_i$ in $\SI^n$ {\rm (} resp. in $\bR P^n$\/{\rm )}
of dimension $\dim V_i -1$ disjoint from $\bv_{\tilde E}$, and the actions of $\bGamma_i$ 
are conjugate by $\Pi_{\bv_{\tilde E}}$.
\item[{\rm (iv)}] The p-R-end is strictly lens-shaped, and
each $C'_i$ corresponds to a cone $C^*_i = \bv_{\tilde E}\ast D'_{i}$. 
The p-R-end has a p-end-neighborhood equal to the interior of 
\[\bv_{\tilde E} \ast D \hbox{ for } D:=  \clo(D'_1) \ast \cdots \ast \clo(D'_{l_0})\]
where the interior of $D$ 
forms the boundary of the p-end neighborhood in $\torb$. 
\item[{\rm (v)}] The set $S(\bv_{\tilde E})$ of maximal segments in $\Bd \torb$ from $\bv_{\tilde E}$ in the closure of a p-end-neighborhood of 
$\bv_{\tilde E}$ is independent of the p-end-neighborhood. 
\[S(\bv_{\tilde E}) = \bigcup_{i=1}^{l_{0}} \bv_{\tilde E} *\clo(D'_1)*\cdots *\clo(D'_{i-1})* \clo(D'_{i+1}) * \cdots * \clo(D'_{l_0}).\]
\item[{\rm (vi)}] 
A concave p-end-neighborhood of $\tilde E$ is a proper p-end-neighborhood. 
Finally, the statements {\rm (iii)} and {\rm (v)} of Theorem \ref{thm-lensclass} also hold. 
\end{itemize}
\end{theorem} 
\begin{proof} 
Again the $\SI^n$-version is enough.
(i)  This follows by Definition \ref{defn-admissible} and  Proposition \ref{I-prop-Ben2} in \cite{EDC1} following Benoist. 

(ii) If $\lambda_{\bv_{\tilde E}}(g)$ is the largest norm of eigenvalue with multiplicity one, 
then $\{g^n(x)\}$ for a point $x$ of a generalized lens converges to $\bv_{\tilde E}$ as $n \ra \infty$. 
Since the closure of a generalized lens is disjoint from the point, this is a contradiction. 
Therefore, the largest norm $\lambda_1(g)$ of the eigenvalues of $g$ is greater than or equal to 
$\lambda_{\bv_{\tilde E}}(g)$. 

Let $U$ be a concave p-end-neighborhood of $\tilde E$ in $\tilde {\mathcal{O}}$.
Let $S_1,..., S_{l_0}$ be the projective subspaces in general position meeting only at the p-end vertex $\bv_{\tilde E}$
where on the corresponding subspaces in $\SI^{n-1}_{\bv_{\tilde E}}$ 
the factor groups $\bGamma_1, ...,\bGamma_{l_0}$ act irreducibly. 
Let $C_i$ denote the union of great segments from $\bv_{\tilde E}$ corresponding to the invariant cones in $S_i$ for each $i$.  
The abelian center isomorphic to $\bZ^{l_0-1}$ acts as the identity on the subspace corresponding to $C_{i}$
in the projective space $\SI^{n-1}_{\bv_{\tilde E}}$. 

Let $g\in \bZ^{l_0-1}$. 
By the above property of being the identity, 
$g| C_i$ is semisimple with two eigenvalues or nonsemisimple with just single eigenvalue by 
the last item of Proposition \ref{I-prop-Ben2} of \cite{EDC1}. 
In the second case $g|C_i$ could be represented by a matrix with eigenvalues all $1$  fixing $\bv_{\tilde E}$. 
Since a generalized lens $L$ meets it, $g|C_i$ has to be identity by the proper convexity of $\torb$: 
Otherwise, $g^n|C$ will send some $x\in L \cap C_i$ to $\bv_{\tilde E}$ and to $\bv_{\tilde E -}$ as $n \ra \pm \infty$
since a matrix form restricted to $1$-dimensional subspaces containing $\bv_{\tilde E}$ and $x$ is of form
\[
\left(
\begin{array}{cc}
1  &  \pm 1\\
 0 &   1
\end{array}
\right).
\]
This contradicts the proper convexity of $\torb$. 

Therefore, 
we have one of the two possibilities for $g$ in the center and $C_i$: 
\begin{itemize} 
\item[{\rm (a)}] $g|C_i$ fixes each point of a hyperspace $P_i \subset S_i$ not passing through $\bv_{\tilde E}$ 
and $g$ has a representation as a nontrivial scalar multiplication in the affine subspace $S_i - P_i$ of $S_i$. 
Since $g$ commutes with every element of $\bGamma_i$ acting on $C_i$, 
$\bGamma_i$ acts on $P_i$ as well.  
\item[{\rm (b)}] $g|C_i$ is an identity. 
\end{itemize}
We denote $I_1:=\{ i| \exists g \in \bZ^{l_0-1}, g|C_i \ne \Idd\} $ and 
 $I_2:= \{i| \forall g \in \bZ^{l_0-1}, g|C_i = \Idd\}.$

By the cocompactness of $\bGamma_{\tilde E}$, 
we can choose an element $g \in \bZ^{l_{0}-1}$ so that 
$g|C_{i}$ for each $i \in I_{2}$ has the submatrix with the largest norm eigenvalues in the unimodular matrix representation
of $g$. Thus, $I_{2}$ cannot have more than one elements. Hence, $I_{1}\ne \emp$. 

Suppose that $I_2 \ne \emp$. 
For each $C_i$, we can find $g_i \in \bZ^{l_0-1}$ with the largest norm eigenvalue associated with it. 
By multiplying with some other element of the virtual center, we can show that 
if $i \in I_1$, then $C_i \cap P_i$ has a sequence $\{g_{i, j}\}$ with $i$ fixed
so that the premises of Proposition \ref{prop-decjoin} are 
satisfied, and 
if $i \in I_2$, then $C_i$ has such a sequence $\{g_{i, j}\}$.

By Proposition \ref{prop-decjoin}, this implies that $\clo(\torb)$ is a join
\[\ast_{i \in I_{1}} K_i \ast \ast_{i \in I_{2}}K_{i}\] where $K_{i}$, $i \in I_{1}$, 
for a properly convex domain in $C_{i}\cap P_{i}$ and $K_{i}$, $i\in I_{2}$, 
is a properly convex domain in $C_{i}$ containing $\bv_{\tilde E}$. 
 
This contradicts the assumptions that $\clo(\torb)$ is not a join 
or that $\bGamma$ is not virtually reducible by Proposition \ref{prop-joinred}. 
Thus, $I_2 = \emp$. 

(iii)  By (ii), for all $C_i$, every $g\in\bZ^{l_0-1}-\{\Idd\}$ acts as nonidentity. 
Then the strict join of all $P_i$ gives us a hyperspace $P$ disjoint from $\bv_{\tilde E}$. 
We will show that it forms a totally geodesic p-R-end for $\tilde E$:

From above, we obtain that every nontrivial $g \in \bZ^{l_0-1}$ 
is clearly diagonalizable with positive eigenvalues associated with $P_i$ and $\bv_{\tilde E}$, 
and the eigenvalue at $\bv_{\tilde E}$ is smaller than the maximal ones at $P_i$. 


Let us choose $C_i$. 
We can find at least one $g'\in \bZ^{l_0-1}$ 
so that $g'$ has the largest norm eigenvalue $\lambda_1(g'_i)$ 
with respect to $C_i$ as an automorphism of $\SI^{n-1}_{\bv_{\tilde E}}$. 
We have $\lambda_1(g') > \lambda_{\bv_{\tilde E}}(g')$ by (ii).

Let $D'_i$ denote $C_i \cap P_i$. 
Each $D'_{i}$ has an attracting fixed point of some $g_i \in \bGamma_i$ restricted to $P_i$ if $\bGamma_i$ is hyperbolic:
Since $\bGamma_i$ is linear on $S_i-P_i$
and $C_{i} - P_{i}$ is a union of two strictly convex cones, 
the theories of Koszul implies the result. 

If $\bGamma_i$ is a trivial group, then we choose $g_i| C_{i}$ to be the identity. 
By multiplying by a sufficiently large power of $g'$ to a chosen $g_{i}$ if necessary, 
we can choose $g_i$ so that the largest norm eigenvalue $\lambda_i$ of $g_i| P_i$ is sufficiently large. 
Then by taking $k$ sufficiently large, $g'^k g_i$ has an attracting fixed point in $D'_{i}$. 
This point must be in $\clo(\tilde{\mathcal{O}})$. 

Since the set of attracting fixed points in $C'_i$ is dense in $\Bd C_i \cap P_i$ by Benoist \cite{Ben1},
we obtain $D'_{i} \subset \clo(\tilde{\mathcal{O}})$.

The strict join $D'$ of $\clo(D'_1), .., \clo(D'_{l_0})$ equals $P \cap \clo(\tilde{\mathcal{O}})$, which is $h(\pi_1(\tilde E))$-invariant.
And $D^{\prime o}$ is a properly convex subset. If any point of $D^{\prime o}$ is in $\Bd \tilde{\mathcal{O}}$, 
then $D'$ is a subset of  $\Bd \tilde{\mathcal{O}}$ by Lemma \ref{lem-simplexbd}. 
Then $\torb$ is a contained in $\bv_{\tilde E} \ast D'$. 
Then $\bGamma$ acts on a strict join. 
By Proposition \ref{prop-joinred}, $\bGamma$ is virtually reducible, a contradiction. 
Therefore, $D^{\prime o} \subset \tilde{\mathcal{O}}$, and $\tilde E$ is a totally geodesic end.

(iv) Let $P$ be the minimal totally geodesic subspace containing all of $P_1, \dots, P_{l_0}$. 
The hyperspace $P$ separates $\tilde{\mathcal{O}}$ into two parts, ones in the p-end-neighborhood $U$ and the subspace outside it. 
Clearly $U$ covers $\Sigma_{\tilde E}$ times an interval 
by the action of $h(\pi_1(\tilde E))$ and the boundary of $U$ goes to a compact orbifold 
projectively diffeomorphic to $\Sigma_{\tilde E}$.

We find a reflection $R$ fixing every points of $D$ and sending $\bv_{\tilde E}$
to its antipode $\bv_{\tilde E-}$. 
Also, there is a projective map $S_{\lambda}$ fixing every point of $D$ and fixing $\bv_{\tilde E}$ with 
two positive eigenvalues $\lambda, 1/\lambda^{n}$. 
Let $F$ be a fundamental domain of $\torb$.  
Call that $\bGamma_{\tilde E}$ acts cocompactly on $D^{o}$.
For an arbitrary neighbourhood $N \subset \torb$ of $D^{\prime o} \cap F$, we can choose sufficiently large $\lambda > 0$ so that 
$S_{\lambda}\circ R(B) \cap F$ is in $N$. 
Since 
\[(S_{\lambda}\circ R )\circ g = g\circ (S_{\lambda} \circ R) \hbox{ for } g \in \bGamma_{\tilde E}\] 
by the matrix forms, 
$S_{\lambda}\circ R(B)$ is $\bGamma_{\tilde E}$-invariant and 
and $S_{\lambda}\circ R(B) \subset \torb$. 
Now, $B \cup S_{\lambda}\circ R(B)$ bounds a strict lens. 


(v) Let $U$ be the p-end-neighborhood of $\bv_{\tilde E}$ obtained in (iv). 
For each $i$, we can find a sequence $g_j$ in the virtual center so that 
\[g_j|  \clo(D'_1)*\cdots *\clo(D'_{i-1})*\clo(D'_{i+1})* \cdots * \clo(D'_{l_0})\] converges to the identity. 
Therefore, we obtain
\[\bv_{\tilde E} * \clo(D'_1)*\cdots *\clo(D'_{i-1})*\clo(D'_{i+1})* \cdots * \clo(D'_{l_0})= \Bd \torb \cap \clo(U)\] 
by the eigenvalue conditions of the virtual center obtained in (iii) and Lemma \ref{lem-centerprojection}.
Hence, (v) follows easily now.

(vi) follows by an argument similar to the proof of Theorem \ref{thm-lensclass}. 

\end{proof}


\subsection{Technical propositions.}

By the following, the first assumption of Theorem \ref{thm-redtot} are needed only
for the conclusion of the theorem to hold. 

\begin{proposition} \label{prop-joinred} 
If a group $G$ of projective automorphisms 
acts on a strict join $A= A_1 \ast A_2$ for two compact convex sets $A_1$ and $A_2$, then $G$ is virtually reducible. 
\end{proposition} 
\begin{proof} 
We prove for $\SI^n$.
Let $x_1, \dots, x_{n+1}$ denote the homogeneous coordinates. 
There is at least one set of strict join sets $A_{1}, A_{2}$. 
We choose a maximal number collection of compact convex sets $A'_1, \dots, A'_m$ so that 
$A$ is a strict join $A'_1 \ast \cdots \ast A'_m$. 
Here, we have $A'_i \subset S_i$ for a subspace $S_i$ corresponding to 
a subspace $V_i \subset \bR^{n+1}$ that form independent set of subspaces. 

We claim that $g \in G$ permutes the collection $\{A'_1, \dots, A'_m\}$:
Suppose not. We give coordinates so that $A'_i$ satisfies $x_j = 0$ for $j \in I_i$ for some indices
and $x_i \geq 0$ for elements of $A$. 
Then we form a new collection of nonempty sets 
\[J':= \{ A'_i \cap g(A'_j)| 0 \leq i, j \leq n, g \in G\}\]
with more elements. 
 Since \[A = g(A) = g(A'_1) \ast \cdots \ast g(A'_n),\] using coordinates we can show 
that each $A'_i$ is a strict join of nonempty sets in \[J'_i := \{ A'_i \cap g(A'_j)| 0 \leq j \leq n, g \in G\}.\]
 $A$ is a strict join of the collection of the sets in $J'$, a contraction to the maximal property. 
 
Hence, by taking a finite index subgroup $G'$ of $G$ acting trivially on the collection, 
$G'$ is reducible. 
\end{proof}

\begin{proposition} \label{prop-decjoin} 
 Suppose that a set $G$ of projective automorphisms  in $\SI^n$ {\rm (}resp. in $\bR P^n${\rm )}
 acts on subspaces $S_1, \dots, S_{l_0}$ and a properly convex domain $\Omega \subset \SI^n$ 
 {\rm (}resp.  $\subset \bR P^n${\rm )}, corresponding 
 to subspaces $V_1, \dots, V_{l_0}$ so that $V_i \cap V_j =\{0\}$ for $i \ne j$ 
 and $V_1 \oplus \cdots \oplus V_{l_0} = \bR^{n+1}$. Let $\Omega_i : = \clo(\Omega) \cap S_i$. 
 We assume that  
 \begin{itemize}
\item for each $S_i$, $G_i := \{g|S_i| g \in G\}$ forms a bounded set of automorphisms and 
\item for each $S_i$, there exists a sequence $\{g_{i, j} \in G\}$ with largest norm eigenvalue $\lambda_{i, j}$ restricted at $S_i$
 has the property $\{\lambda_{i, j}\} \ra \infty$ as $j \ra \infty$. 
 \end{itemize}
 Then $\clo(\Omega) = \Omega_1 * \cdots * \Omega_{l_0}$ for $\Omega_j \ne \emp, j=1, \dots, l_0$. 
 \end{proposition} 
 \begin{proof}
  We will prove for $\SI^n$ but the proof for $\bR P^n$ is identical.  
 First, $\Omega_i \subset \clo(\Omega)$ by definition. 
 Since the element of a strict join has a vector that is a linear combination of elements of 
 the vectors in the directions of $\Omega_1, \dots, \Omega_{l_0}$, 
 Hence, we obtain \[\Omega_1 \ast \cdots \ast \Omega_{l_0} \subset \clo(\Omega)\] since
 $\clo(\Omega)$ is convex. 
 
 Let $z = [ \vec{v}_z]$ for a vector $\vec{v}_z$ in $\bR^{n+1}$.
 We write $\vec{v}_z= \vec{v}_1 + \cdots + \vec{v}_{l_0}$, $\vec{v}_j \in V_j$ for each $j$, $j=1, \dots, l_0$, which is a unique sum. 
Then $z$ determines $z_i = [v_i]$ uniquely. 

Let $z $ be any point.  
 We choose a subsequence of $\{g_{i, j}\}$ so that $\{g_{i, j}|S_i\}$ converges to a projective automorphism 
 $g_{i, \infty}: S_i \ra S_i$ and $\lambda_{i, j} \ra \infty$ as $j \ra \infty$. 
  Then $g_{i, \infty}$ also acts on $\Omega_i$. 
And $g_{i, j}(z_i) \ra g_{i, \infty}(z_i) = z_{i, \infty}$  for a point $z_{i, \infty} \in S_i$.
We also have
 \begin{equation} \label{eqn-lim}
 z_i = g_{i, \infty}^{-1}(g_{i, \infty}(z_i)) = g_{i, \infty}^{-1}(\lim_j g_{i, j}(z_i)) = g_{i, \infty}^{-1}(z_{i, \infty}).
 \end{equation}

Now suppose $z \in \clo(\Omega)$. 
We have   $g_{i, j}(z) \ra z_{i, \infty}$ by the eigenvalue condition. Thus, 
we obtain $z_{i, \infty} \in \Omega_i$ as $z_{i, \infty}$ is the limit of a sequence of orbit points of $z$. 
 Hence we also obtain $z_i \in \Omega_i$ by equation \eqref{eqn-lim}. 
 We obtain $\Omega_i \ne \emp$.
 This shows that $\clo(\Omega) = \Omega_1 * \cdots * \Omega_{l_0}$.

 
 \end{proof}


For the proof of the following, we will use Theorem \ref{thm-redtot}(i)-(iv). 
We need the lemma for Theorem \ref{thm-redtot}(v) only. 
 
 \begin{lemma} \label{lem-centerprojection}
Assume as in Theorem \ref{thm-redtot}.  
Assume $\torb \subset \SI^n$ {\rm (} resp. $\torb \subset \bR P^n${\rm )}. 
Suppose that $\tilde E$ is a generalized lens-type R-end, and 
$\tilde E$ is virtually factorable. 
Then 
for every sequence $\{g_j\}$ of distinct elements of the virtual center $\bZ^{l_0-1}$, 
we have 
\[\frac{\lambda_1(g_j)}{\lambda_{\bv_{\tilde E}}(g_j)}  \ra  \infty, 
 \frac{\lambda_n(g_j)}{\lambda_{\bv_{\tilde E}}(g_j)}  \ra  0\]
for the largest norm $\lambda_{1}(g)$ of the eigenvalues of $g$ and 
the least norm $\lambda_{n}(g)$ of those of $g$. 
\end{lemma}
\begin{proof} 
Since $\tilde E$ is virtually factorable, it has an invariant totally geodesic surface $S_{\tilde E}$ as in Theorem \ref{thm-redtot}. 

If for a sequence $g_j$ of $\bZ^l - \{\Idd\}$, 
\[\left\{\left|\frac{\lambda_1(g_j)}{\lambda_{\bv_{\tilde E}}(g_j)}\right|\right\}\] 
the subsequence converges to $0$, then $g_{j}(x)$ for some $x \in L$ converges to $\bv_{\tilde E}$. 
This contradicts the disjointedness of $L$ to $\bv_{\tilde E}$.
Thus, we assume that the sequence converges to a positive constant. 

Suppose that for a sequence $g_j$ of $\bZ^l - \{\Idd\}$, 
\[\left\{\left|\frac{\lambda_1(g_j)}{\lambda_{\bv_{\tilde E}}(g_j)}\right|\right\}\] 
is bounded above. 
We assume without loss of generality that $\lambda_1(g_j)$ occurs for a fixed collection $C'_i$, $i \in I$, by taking a subsequence of 
$\{g_j\}$ if necessary.  
Then $\{g_j\}$ acts as a bounded set of projective automorphisms of $\ast_{i \in I} C'_i$. 
Since $g_j$ acts trivially on each $D'_j$ for each $j$ for all $j \not\in I$ by Theorem \ref{thm-redtot}(i). 
Again by Proposition \ref{prop-decjoin}, $\clo(\Omega)$ is a nontrivial strict join
$(\ast_{i \in I} C'_i) \ast (\ast_{i \not\in I} D'_{J})$ by considering $\{g_{j}\} \cup \{g_{j}^{-1}\}$
since each sequence $\{g_{j}^{-1}\}$ has a subsequence with largest eigenvalue in the join 
$ \ast_{i \in K } D'_{J})$ for a collection $K \subset I^{c}$. 
Now apply this to $\clo(\torb)$ which must be a joined set. 


 \end{proof}




\section{Duality and lens-type T-ends}\label{sec-dualT}

We first discuss the duality map. We show a lens-cone p-end neighborhood  of a p-R-end 
is dual to a lens p-end neighborhood of a p-T-end. 
Using this we prove Theorem \ref{thm-equ2} dual to Theorem \ref{thm-equ}, i.e., 
Theorem \ref{thm-secondmain}. 

\subsection{Duality map.} 

The Vinberg duality diffeomorphism induces a one-to-one correspondence between p-ends of $\torb$ and $\torb^*$ 
by considering the dual relationship $\bGamma_{\tilde E}$ and $\bGamma^*_{\tilde E'}$ for each pair of 
p-ends $\tilde E$ and $\tilde E'$ with dual p-end fundamental groups. (See Section \ref{I-sec-duality} of \cite{EDC1}.)

Given a properly convex domain $\Omega$ in $\SI^n$ (resp. $\bR P^n$), 
we recall the {\em augmented boundary} of $\Omega$
\begin{align} 
\Bd^{\Ag} \Omega  &:= \{ (x, h)| x \in \Bd \Omega, x\in h, \nonumber \\ 
 & h \hbox{ is an oriented supporting hyperplane of } \Omega \}  \subset \SI^{n}\times \SI^{n \ast}.
\end{align}
This is a closed subspace. 
Each $x \in \Bd \Omega$ has at least one supporting hyperspace, 
an oriented hyperspace is an element of $\SI^{n \ast}$ since it is represented as a linear functional,   
and an element of $\SI^n$ represent an oriented hyperspace in $\SI^{n \ast}$. 

We recall a {\em duality map} 
\begin{equation} \label{eqn:dualmap}  
{\mathcal{D}}_{\Omega}: \Bd^{\Ag} \Omega \leftrightarrow \Bd^{\Ag} \Omega^* 
\end{equation}
given by sending $(x, h)$ to $(h, x)$ for each $(x, h) \in \Bd^{\Ag} \Omega$. 
This is a diffeomorphism since $\mathcal{D}$ has an inverse given by 
switching factors. 

A convex domain $\Omega$ is {\em strictly convex} at a  point $p \in \Bd \Omega$ if
there is no straight segment $s$ in $\Bd \Omega$ with $p \in s$. 
For later purposes, we need 
\begin{lemma}\label{lem-predual}
Let $\Omega^*$ be the dual of a properly convex domain $\Omega$ in $\SI^n$ or $\bR P^n$. 
Then 
\begin{itemize}
\item[{\rm (i)}] $\Bd \Omega$ is $C^1$ and strictly convex at a point $p \in \Bd \Omega$ if and only if 
$\Bd \Omega^*$  is $C^1$ and strictly convex at the unique corresponding point $p^{*}$. 
\item[{\rm (ii)}] $\Omega$ is an ellipsoid if and only if so is $\Omega^*$. 
\item[{\rm (iii)}] $\Bd \Omega^*$ contains a properly convex domain $D = P \cap \Bd \Omega^*$ open in a totally geodesic hyperplane $P$ 
if and only if $\Bd \Omega$ contains 
a vertex $p$ with $R_p(\Omega)$ a properly convex domain. 
In this case, $\mathcal{D}$ sends the pair of $p$ and the associated supporting hyperplanes of $\Omega$
 to the pairs of the totally geodesic hyperplane containing $D$ and points of $D$.  
Moreover, $D$ and $R_p(\Omega)$ are properly
 convex and are projectively diffeomorphic to dual domains. 
\end{itemize}
\end{lemma}
\begin{proof} 
(i) $\Bd \Omega$ near $p$ is a graph of a function $f: B \ra \Bd \Omega$ where 
$B$ is an open set in a hyperspace supporting $\Omega$ at $p$.  
The $C^{1}$-condition implies that $Df:B \ra S(\bR^{n+1\ast})$ is well-defined. 
If $Df$ is not injective in any neighborhood of $p$, we can deduce that there exists a set of 
identical supporting hyperplanes $P$ with distinct supporting points at $\Bd \Omega$. 
$P \cap \Bd \Omega$ is a nontrivial convex set of dimension $> 0$, 
and $\Omega$ is not strictly convex at $p$. 
Hence, $Df$ is injective in a neighborhood of $p$. 
Now, we can apply the inverse function to obtain that $\Bd \Omega^{\ast}$ is $C^{1}$ also. 
It must be strictly convex at $p^{*}$ since otherwise the supporting hyperspaces must be identical along 
a line in $\Bd \Omega$, and the inverse map is not injective. 
The converse also follows by switching the role of $\Omega$ and $\Omega^{\ast}$. 

(ii) This is trivial. 

(iii) Suppose that $R_{p}(\Omega)$ is properly convex. 
We consider the set of hyperplanes supporting $\Omega$ at $p$. This forms a properly convex domain 
as we can see the space as the projectivization of the space of linear functionals supporting $C(\Omega)$: 

Let $v$ be the vector in $\bR^{n+1}$ in the direction of $p$. Then the set of supporting linear functionals
of $C(\Omega)$.  Let $V$ be a complementary space of $v$ in $\bR^{n+1}$. 
Let $A$ be given as $V + v$. 
We choose $V$ so that $C_{v}:= C(\Omega) \cap A$ is a bounded convex domain in $A$. 
We give $A$ a linear structure so that $v$ corresponds to the origin. 
Let $A^{\ast}$ denote the dual linear space.
The set of linear functionals positive on $C(\Omega)$ and $0$ at $v$ is identical with 
that of linear functionals on the linearized $A$ positive on $C_{v}$:
we define 
\begin{align} 
C(D) &:= \{ f\in \bR^{n+1\ast}| f| C(\Omega) > 0, f(v) = 0 \} \subset \bR^{n+1\ast} \nonumber \\ 
= \widehat C_{v}^{\ast} & :=\{ g \in A^{*}| g| C_{v} > 0 \}.
\end{align}
The equality follows by the decomposition $\bR^{n+1} = \{tv| t \in \bR\} \oplus V$.
Define $R'_{v}(C_{v})$ as the equivalence classes of properly convex segments in $C_{v}$ ending at $v$ 
where two segments are equivalent if they agree in an open neighborhood of $v$. 
$R_{p}(\Omega)$ is identical with $R'_{v}(C_{v})$ by projectivization $\bR^{n+1} \ra \SI^{n}$.  
Hence $R'_{v}(C_{v})$ is a properly convex open domain in $\mathcal{S}(A)$. 
Since $R'_{v}(C_{v})$ is properly convex, 
the interior of the spherical projectivization $\mathcal{S}(\widehat C_{v}^{\ast}) \subset \mathcal{S}(A^{\ast})$ 
is dual to the properly convex domain $R'_{v}(C_{v}) \subset \mathcal{S}(A)$.  

Define $D:= \mathcal{S}(C(D)) \subset \SI^{n\ast}$. 
Since $R'_{v}(C_{v})$ corresponds to $R_{p}(\Omega)$,  and 
$\mathcal{S}(\widehat C_{v}^{\ast})$ corresponds to $D$, the conclusion follows. 
\end{proof} 


\begin{remark}\label{rem-duallens}
For an open subspace $A \subset \Bd \Omega$ that is smooth and strictly convex, 
$\mathcal{D}$ induces a well-defined map 
\[A \subset \Bd \Omega \ra A' \subset \Bd \Omega^{\ast}\]
since each point has a unique supporting hyperplane 
for an open subspace $A'$.
The image of the map $A'$ is also smooth and strictly convex by Lemma \ref{lem-predual}. 
We will simply say that $A'$ is the {\em image}  of $\mathcal D$. 
\end{remark}

\begin{figure}
\centering
\includegraphics[trim = 10mm 80mm 1mm 10mm, clip, width=10cm, height=5cm]{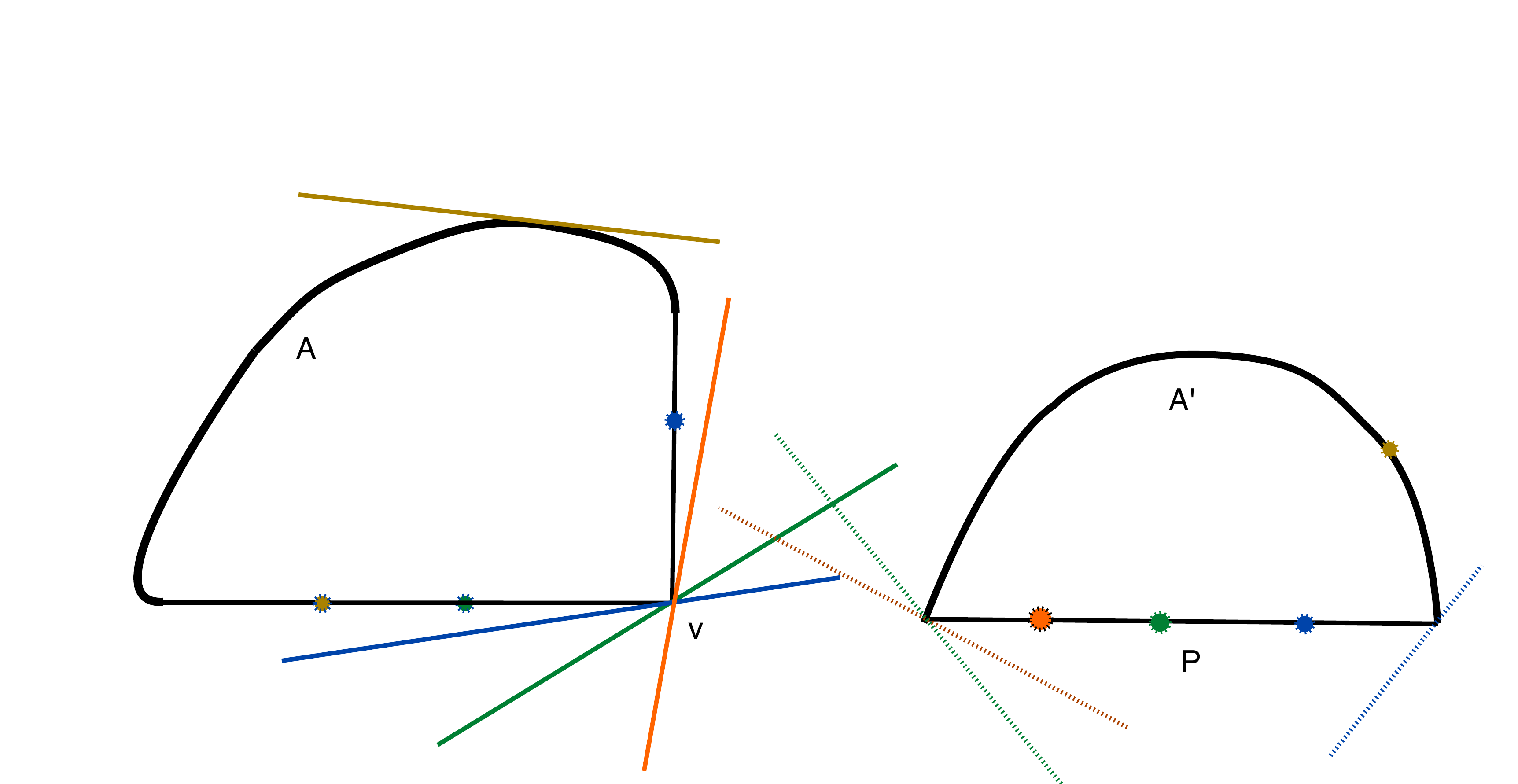}

\caption{The figure for Corollary \ref{cor-duallens}. }

\label{fig-duallens}
\end{figure}

We will need the corollary about the duality of lens-cone and lens-neighborhoods. 
Recall that given a properly convex domain $D$ in $\SI^{n}$ or $\bR P^{n}$, 
the dual domain is the closure of the open set given by the collection of (oriented) hyperplanes in $\SI^{n}$ or $\bR P^{n}$ not meeting $\clo(D)$.

\begin{corollary} \label{cor-duallens} 
The following hold: 
\begin{itemize} 
\item Let $L$ be a lens and $v \not\in L$ so that $v \ast L$ is a properly convex lens-cone. 
Suppose the smooth strictly convex 
boundary component $A$ of $L$ is tangent to a segment from $v$ at each point of $\Bd A$
and $v\ast L = v \ast A$. 
Then the dual domain of $\clo(v\ast L)$ is the closure of a component $L_{1}$ of $L' - P$ where $L'$ is a lens and $P$ is a hyperspace meeting $L^{\prime o}$ but not meeting the boundary of $L'$ and 
$\Bd \partial L_{1} \subset P$. 
\item Conversely, we are given a lens $L'$ and $P$ is a hyperspace meeting $L^{\prime o}$ 
but not meeting the boundary of $L'$. 
Let $L_{1}$ be a component of $L'-P$ with smooth strictly convex boundary 
$\partial L_{1}$ so that $\Bd \partial L_{1}\subset P$. 
The dual of the closure of a component $L_{1}$ of $L' -P$ is 
the closure of $v \ast L$ for a lens $L$ and $v \not\in L$ so that $v\ast L$ is a properly convex lens-cone. 
The outer boundary component $A$ of $L$ is tangent to a segment from $v$ at each point of $\Bd A$
and $v\ast L = v\ast A$. Moreover, $v \not\in \clo(A)$.
\end{itemize}
\end{corollary} 
\begin{proof} 
Let $A$ denote the boundary component of $L$ so that $\{v\} \ast L = \{v\} \ast A$. 
We will determine the dual domain $D$ of
 $\clo(\{v\} \ast L)$ by finding the boundary of $D$ using the duality map $\mathcal{D}$. 
The set of hyperplanes supporting $\clo(v\ast L)$ at $v$ forms a properly totally geodesic domain $D_{1}$ in $\SI^{n\ast}$
contained in a hyperplane $P$ dual to $v$ by Lemma \ref{lem-predual}. 
Also the set of hyperplanes supporting $\clo(\{v\} \ast L)$ at points of $A$ goes to the strictly convex hypersurface $A'$ in $\partial D_{1}$
by Lemma \ref{lem-predual} since $\mathcal{D}$ is a diffeomorphism. 
(See Remark \ref{rem-duallens} and Figure \ref{fig-duallens}.)
$\Bd (v\ast A) - A$ is a union of segments from $v$. 
The supporting hyperplanes containing the segments 
go to points in $\partial D_{1}$. 
Each point of $\clo(A') - A'$ is a limit of a sequence $\{p_{i}\}$ of points of $A'$, corresponding to a sequence of 
supporting hyperspheres $\{h_{i}\}$ to $A$. 
The tangency condition of $A$ and $\Bd A$ implies that the limit hypersphere contains the segment in $S$ from $v$. 
Thus, $\clo(A') - A'$ equals the set of hyperspheres containing the segments in $S$ from $v$. 
Thus, it goes to a point of $\partial D_{1}$. Thus, $\Bd A' = \partial D_{1}$. Let $P$ be the unique hyperplane containing $D_{1}$. 
Then $\partial D = A' \cup D_{1}$. The points of $\Bd A$ go to a supporting hyperplane at points of $\Bd A'$
distinct from $P$. 
Let $L^{\ast}$ denote the dual domain of $\clo(L)$. 
Since $\clo(L) \subset \clo(\{v\}\ast L)$, we obtain $D \subset L^{\ast}$ by equation \eqref{eqn-reversal}.
Since 
\[\partial D \subset A' \cup P, \hbox{ and } A' \subset L^{\ast},\] 
$D$ is the closure of the component of $L^{\ast} - P$. 
Moreover, $A' = \partial L_{1}$ for a component $L_{1}$ of $L' - P$. 

The second item is proved similarly to the first. Then $\partial L_{1}$ goes to a hypersurface $A$ in the
boundary of the dual domain
$D'$ of $\clo(L_{1})$ under $\mathcal{D}$. 
Again $A$ is a smooth strictly convex boundary. 
Since $\Bd \partial L_{1} \subset P$ and $L_{1}$ is a component of $L'-P$, we have
$\Bd L_{1} - \partial L_{1} = \clo(L_{1}) \cap P$. 
This is a totally geodesic properly convex domain $D_{1}$.


If $l \subset P$ be a supporting $n-2$-dimensional space of $D_{1}$, then
a space of hyperplanes containing $l$ forms a projective geodesic in $\SI^{n\ast}$. 
An {\em $L_{1}$-parameter} $P_{t}$ with ends $P_{0}, P_{1}$ is a parameter satisfying 
\[P_{t} \cap P = P_{0 }\cap P,  P_{t}\cap L_{1}^{o} = \emp \hbox{ for  all } t \in [0, 1].\]

There is a one-to-one correspondence 
\[ \{P'| P' \hbox{ is a hyperspace that supports } L_{1} \hbox{ at points of } \partial D_{1} \} 
\leftrightarrow v \ast \Bd A: \]

Every supporting hyperplane $P'$ to $L_{1}$ at points of $\partial D_{1}$ is contained in  a $L_{1}$-parameter $P_{t}$ 
with $P_{0} = P', P_{1} = P $.
$v$ is the dual to $P$ in $\SI^{n\ast}$. 
Each of the path $P_{t}$ is a geodesic segment in $\SI^{n\ast}$ with an endpoint $v$. 

By duality map $\mathcal{D}$, $\Bd D'$ is a union of $A$ and the union of these segments. 
Given any hyperplane $P'$ disjoint from $L_{1}^{o}$, we find a one-parameter family of hyperplanes containing $P' \cap P$. 
Thus, we find a one-parameter family $P_{t}$ 
with $P_{0} = P', P_{1} = P$.
Since the hyperplanes are disjoint from $L_{1}$, the segment is in $D'$. 
Since $D'$ is a properly convex domain,  
we can deduce that $D'$ is the closure of the cone  $\{v\}\ast A$. 

Let $L''$ be the dual domain of $\clo(L')$.  
Since $\clo(L') \supset L_{1}$, we obtain $L'' \subset D'$ by equation \eqref{eqn-reversal}. 
Since $\partial L_{1}\subset L'$, we obtain $A \subset L''$ by the duality map $\mathcal{D}$.  
We obtain that $L^{\prime \prime o} \cup A \subset \{v\}\ast A$.  

Let $B$ be the image of the other boundary component $B'$ of $L'$ under $\mathcal{D}$. 
We take a supporting hyperplane $P_{y}$ at $y\in B'$. 
Then we find a one-parameter family $P_{t}$ of hyperplanes containing $P_{y} \cap P$
with $ P_{0}=P_{y}, P_{1}=P$.
This parameter goes into the segment from $v$ to a point of $A$ under the duality. 
Thus, each segment from $v$ to a point of $A$ meets $B$. 
Thus, $L^{\prime \prime o} \cup A \cup B$ is a lens of the lens cone $ \{v\}\ast A$.
This completes the proof. 
\end{proof}

\subsection{The duality of T-ends and properly convex R-ends.}  \label{sub-dualend}

%



Let $\Omega$ be the properly convex domain covering $\orb$. 
For a T-end $E$, the totally geodesic ideal boundary $\Sigma_{E}$ of $E$ 
is covered by a properly convex open domain in $\Bd \Omega$ corresponding 
to a p-T-end $\tilde E$.  We denote it by $S_{\tilde E}$. We call it 
the {\em ideal boundary} of $\tilde E$. 


\begin{proposition}\label{prop-dualend} 
Let $\orb$ be a strongly tame properly convex real projective orbifold with R-ends or T-ends.
Then the dual real projective orbifold $\orb^*$ is also strongly tame and has the same number of ends so that 
\begin{itemize} 
\item there exists a one-to-one correspondence $\mathcal{C}$ between the set of ends of $\orb$ and the set of ends of $\orb^*$. 
\item $\mathcal{C}$ restricts to such a one  between the subset of horospherical ends of $\orb$ and the subset of horospherical ones of $\orb^*$.
\item $\mathcal{C}$ restricts to such a one  between the set of 
T-ends of $\orb$ with the set of ends of properly convex R-ends of $\orb^*$.
The ideal boundary $S_{\tilde E}$ for a p-T-end $\tilde E$ 
is projectively diffeomorphic to the properly convex open domain dual to the domain
$\tilde \Sigma_{\tilde E^*}$ for the corresponding p-R-end $\tilde E^*$ of $\tilde E$. 
\item $\mathcal{C}$ restricts to such a one  between 
the subset of all properly convex R-ends of $\orb$ and the subset of all T-ends of $\orb^*$. 
Also, $\tilde \Sigma_{\tilde E}$ of a p-R-end is projectively dual to the ideal boundary $S_{\tilde E^*}$ 
for the corresponding dual p-T-end $\tilde E^*$ of $\tilde E$. 
\end{itemize} 
\end{proposition}
\begin{proof}
We prove for the $\SI^n$-version.
By the Vinberg duality diffeomorphism of Theorem \ref{I-thm-dualdiff} of \cite{EDC1}, $\orb^*$ is also strongly tame. 
Let $\torb$ be the universal cover of $\orb$. Let $\torb^*$ be the dual domain. 
The first item follows by the fact that 
this diffeomorphism sends pseudo-ends neighborhoods to pseudo-end neighborhoods. 

Let $\tilde E$ be a horospherical p-R-end with $x$ as the end vertex. 
Since there is a subgroup of a cusp group acting on $\clo(\torb)$ with $x$ fixed by \cite{EDC1},   
the intersection of the unique supporting hyperspace $h$ with $\clo(\torb)$ is a singleton $\{x\}$. 
The dual subgroup is also a cusp group and acts on $\clo(\torb^*)$ with $h$ fixed.
So the corresponding $\torb^*$ has the dual hyperspace $x^*$ of $x$ as the unique intersection 
at $h^*$ dual to $h$ at $\clo(\torb^*)$. Hence $x^*$ is a horospherical end.

A p-R-end $\tilde E$ of $\torb$ has a p-end vertex $\bv_{\tilde E}$. 
$\tilde \Sigma_{\tilde E}$ is a properly convex domain in $\SI^{n-1}_{\bv_{\tilde E}}$. 
The space of supporting hyperplanes of $\torb$ at 
$\bv_{\tilde E}$ forms a properly convex domain of dimension $n-1$ 
since they correspond to hyperplanes in $\SI^{n-1}_{\bv_{\tilde E}}$
not intersecting $\tilde \Sigma_{\tilde E}$. 
Under the duality map ${\mathcal{D}}_{\torb}$, $(\bv_{\tilde E}, h)$ for a supporting hyperplane $h$ is sent 
to $(h^{\ast}, \bv_{\tilde E}^{\ast})$. Lemma \ref{lem-predual} 
shows that 
$h^{\ast}$ is a point in a properly convex $n-1$-dimensional domain 
$\Bd \torb^{\ast} \cap P$ for $P = \bv_{\tilde E}^{\ast}$, a hyperplane. 
Thus, $\tilde E^{\ast}$ is a totally geodesic end with $\tilde \Sigma_{\tilde E^{\ast}}$ dual to $S_{\tilde E}$. 
This proves the third item. The fourth item follows similarly. 
\end{proof} 

\begin{remark}\label{rem:comp} 
We also remark that  the map induced on the set of pseudo-ends of $\torb$ to that of $\torb^{*}$ by
${\mathcal{D}}_{\torb}$ is compatible with the Vinberg diffeomorphism.
This easily follows by Proposition 6.7 of \cite{wmgnote} and 
the fact that the level set $S_{x} \subset \bR^{n+1}$ of the Koszul-Vinberg function is asymptotic to 
the boundary of $\torb$. 
Thus, the hyperspace in $\bR^{n+1}$ corresponding to the
supporting hyperplane of a p-end vertex is approximated by a tangent hyperplane to 
$S_{x}$ in $\bR^{n+1}$. $\mathcal{D}_{\torb}$ sends a point $p$ of $S_{x}$ to the linear form corresponding to 
the tangent hyperplane of $S_{x}$ at $p$. 
(See Chapter 6 of Goldman \cite{wmgnote}.)
\end{remark}



$\mathcal{C}$ restricts to a correspondence between the lens-type R-ends with 
lens-type T-ends. See Corollary \ref{cor-duallens2}  for detail. 

\begin{proposition}\label{prop-dualend2}
Let $\orb$ be a strongly tame properly convex real projective orbifold. 
The following conditions are equivalent\,{\rm :} 
\begin{itemize} 
\item[{\rm (i)}] A properly convex R-end of $\orb$ satisfies the uniform middle-eigenvalue condition.
\item[{\rm (ii)}] The corresponding totally geodesic end of $\orb^*$ satisfies this condition.
\end{itemize} 
\end{proposition}
\begin{proof}
The items (i) and (ii) are equivalent by considering equation \eqref{eqn-umec}.
\end{proof}


We now prove the dual to Theorem \ref{thm-equ}. For this we do not need the triangle condition or the
reducibility of the end. 

\begin{theorem}\label{thm-equ2}
Let $\orb$ be a properly convex real projective orbifold. 
Assume that the holonomy group is strongly irreducible.
Let $S_{\tilde E}$ be a totally geodesic ideal boundary of a p-T-end $\tilde E$ of $\torb$. 
Then the following conditions are equivalent\,{\rm :} 
\begin{itemize} 
\item[{\rm (i)}] $\tilde E$ satisfies the uniform middle-eigenvalue condition.
\item[{\rm (ii)}] $S_{\tilde E}$ has a lens neighborhood in an ambient open manifold containing $\torb$
and hence $\tilde E$ has a lens-type p-end-neighborhood in $\torb$. 
\end{itemize} 
\end{theorem}
\begin{proof}
It suffices to prove for $\SI^n$.
Assuming (i), the existence of a lens neighborhood follows from Theorem \ref{thm-lensn}. 

Assuming (ii), we obtain a totally geodesic $(n-1)$-dimensional properly convex domain $S_{\tilde E}$ in
a subspace $\SI^{n-1}$ where $\bGamma_{\tilde E}$ acts on. 
Let $U$ be the two-sided properly convex neighborhood of it where $\bGamma_{\tilde E}$ acts on. 
Then since $U$ is a two-sided neighborhood, the supporting hemisphere at each point of 
$\clo(S_{\tilde E})-S_{\tilde E}$ is now transversal to $\SI^{n-1}$. 
Let $P$ be the hyperplane containing $S_{\tilde E}$, and 
let $U_{1}$ be the component of $U - P$. Then the dual $U_{1}^{\ast}$ is 
a lens-cone by the second part of Corollary \ref{cor-duallens}. 
The dual $U^*$ of $U$  is a lens contained in a lens-cone $U_{1}^{\ast}$ where 
$\bGamma_E$ acts on $U^{\ast}$. 
We apply the part (i) $\Rightarrow$ (ii) of Theorem \ref{thm-equ}. 
\end{proof}

Theorems \ref{thm-equ} and \ref{thm-equ2} and Propositions \ref{prop-dualend} and \ref{prop-dualend2} imply 
\begin{corollary}\label{cor-duallens2} 
Let $\orb$ be a strongly tame properly convex real projective orbifold and 
let $\orb^{\ast}$ be its dual orbifold. 
The dual end correspondence $\mathcal{C}$ restricts to a correspondence between the generalized lens-type R-ends with 
lens-type T-ends with admissible end fundamental groups. 
If $\orb$ satisfies the triangle condition or every end is virtually factorable, $\mathcal{C}$ restricts to 
a correspondence between the lens-type R-ends with 
lens-type T-ends with admissible end fundamental groups. 
\end{corollary} 


\begin{proof}[{\sl Proof of Corollary \ref{cor-Coxeter}}.] 
Let $\tilde E$ be a  p-R-end. Under the premise, $\lambda_{v_{\tilde E}}(g) = 1$ for a p-end vertex $v_{\tilde E}$ of $\tilde E$.
Suppose that $\bGamma_{\tilde E}$ is irreducible. 
Suppose that $\tilde E$ is properly convex. 
By Theorem \ref{thm-eignlem} and Remark \ref{rem-eigenlem}, $\bGamma_{\tilde E}$ 
satisfies the uniform middle eigenvalue 
condition.  Theorem \ref{thm-equ} implies the result. 

If $E$ is a T-end, Theorem \ref{thm-equ2} implies the result similarly. 
\end{proof}


\section{Application: The openness of the lens properties, and expansion and shrinking of end neighborhoods} \label{sec-results}

We will list a number of properties that we will need later. 
(These are not essential in this paper itself.) 
We show the openness of the lens properties, i.e., the stability for properly convex radial ends
and totally geodesic ends. We can find an increasing sequence of 
horoball p-end-neighborhoods, lens-type p-end-neighborhoods for radial or totally geodesic p-ends that exhausts 
$\torb$. We also show that the p-end-neighborhood always contains a horoball p-end-neighborhood 
or a concave p-end neighborhood. 
Finally, we discuss how to get rid of T-ends as boundary components. 

\subsection{The openness of lens properties.}


A {\em radial affine connection} is an affine connection on $\bR^{n+1} -\{O\}$ invariant under 
the radial dilatation $S_t: \vec{v} \ra t\vec{v}$ for every $t > 0$. 


As conditions on representations of $\pi_1(\tilde E)$, the condition for 
generalized lens-shaped ends and one for lens-shaped ends are the same. 
Given a holonomy group of $\pi_1(\tilde E)$ acting on a generalized lens-shaped cone p-end neighborhood,  
the holonomy group satisfies the uniform middle eigenvalue condition by Theorem \ref{thm-equ}. 
We can find a lens cone by choosing our orbifold to be ${\mathcal T}_{\bv_{\tilde E}}^{o}/\pi_1(\tilde E)$ 
by Proposition \ref{prop-convhull2}.

A segment is {\em radial} if it is a segment from $\bv_{\tilde E}$. 

\begin{theorem}\label{thm-qFuch}
Let $\orb$ be a properly convex real projective orbifold. 
Assume that the holonomy group is strongly irreducible.
Assume that the universal cover $\torb$ is a subset of $\SI^n$ {\rm (}resp.\, $\bR P^n${\rm ).} 
Let $\tilde E$ be a properly convex p-R-end of the universal cover $\torb$. 
Let $\Hom_E(\pi_1(\tilde E), \SLnp)$ {\rm (}resp. $\Hom_E(\pi_1(\tilde E), \PGL(n+1, \bR))${\rm )} be the space of representations of the 
fundamental group of an $(n-1)$-orbifold $\Sigma_{\tilde E}$ with an admissible fundamental group. 
Then 
\begin{itemize}
\item[{\rm (i)}] $\tilde E$ is a generalized lens-type R-end if and only if $\tilde E$ is a strictly generalized lens-type R-end.
\item[{\rm (ii)}] The subspace of generalized lens-shaped  representations of an R-end is open. 
\end{itemize}
Finally, if $\orb$ satisfies the triangle condition or every end is virtually factorable, then we can replace the word {\em generalized lens-type}
to {\em lens-type} in each of the above statements. 
\end{theorem}
\begin{proof} 
(i) If $\pi_1(\tilde E)$ is hyperbolic, then the equivalence is given in Theorem \ref{thm-lensclass} (i), and 
if $\pi_1(\tilde E)$ is a virtual product of hyperbolic groups and abelian groups, then it is in Theorem \ref{thm-redtot} (iv).

(ii) Let $\mu$ be a representation $\pi_1(\tilde E) \ra \SLnp$ associated with a generalized lens-cone.  
By Theorem \ref{thm-equiv}, 
we obtain a lens domain $K$ in ${\mathcal T}_{\bv_{\tilde E}}$ with smooth convex boundary components 
$A \cup B$ since ${\mathcal T}_{\bv_{\tilde E}}$ itself satisfies the triangle condition although it is not properly convex. 
(Note we don't need $K$ to be in $\torb$ for the proof.)

$K/\mu(\pi_1(\tilde E))$ is a compact orbifold whose boundary is the union of two closed $n$-orbifold components
$A/\mu(\pi_1(\tilde E)) \cup B/\mu(\pi_1(\tilde E))$.
Suppose that $\mu'$ is sufficiently near $\mu$. 
We may assume that $\bv_{\tilde E}$ is fixed by conjugating $\mu'$ by a bounded projective transformation. 
By considering the radial segments in $K$, we obtain a foliation by radial lines in $K$ also. 
By Proposition \ref{prop-koszul}, applying Proposition \ref{prop-lensP} to the both boundary components of 
the lens, we obtain a lens-cone in ${\mathcal T}_{\bv_{\tilde E}}$. 
This implies that the sufficiently small change of holonomy keep $\tilde E$ to have a concave p-end neighborhood. 
This completes the proof of (ii). 


The final statement follows by Lemma \ref{lem-genlens}. 

\end{proof} 


\begin{theorem}\label{thm-qFuch2}
Let $\orb$ be a strongly tame properly convex real projective orbifold. 
Assume that the holonomy group is strongly irreducible.
Assume that the universal cover $\torb$ is a subset of $\SI^n$ {\rm (}resp. of $\bR P^n${\rm ).}
Let $\tilde E$ be a p-T-end of the universal cover $\torb$. 
Let $\Hom_E(\pi_1(\tilde E), \SLnp)$  {\rm (}resp. $\Hom_E(\pi_1(\tilde E), \PGL(n+1, \bR))${\rm )} 
be the space of representations of the 
fundamental group of an $n$-orbifold $\Sigma_{\tilde E}$ with an admissible fundamental group. 
Then the subspace of  lens-shaped  representations of a p-T-end is open. 
\end{theorem}
\begin{proof} 
By Theorem \ref{thm-equ2}, the condition of the lens p-T-end
is equivalent to the uniform middle eigenvalue condition for the end. 
By Proposition \ref{prop-dualend2} and Theorems \ref{thm-secondmain} and \ref{thm-qFuch}
complete the proof. 

\end{proof} 




\begin{corollary}\label{cor-mideigen}
We are given a properly convex end $\tilde E$ of a strongly tame properly convex orbifold $\orb$
with the admissible end fundamental group. 
Suppose that the holonomy group of $\orb$ is strongly irreducible. 
Assume that $\torb \subset \SI^n$ {\rm (}resp.\, $\torb \subset \bR P^n${\rm ).}  
Then the subset of 
\[\Hom_E(\pi_1(\tilde E), \SL_\pm(n+1, \bR)) \hbox{ {\rm (}resp.}\, \Hom_E(\pi_1(\tilde E), \PGL(n+1, \bR)) ) \]
consisting of  representations satisfying the uniform middle-eigenvalue condition is open.
\end{corollary} 
\begin{proof} 
For p-R-ends, this follows by Theorems \ref{thm-equ} and \ref{thm-qFuch}.
For p-T-ends, this follows by dual results: Theorem \ref{thm-equ2} and Theorems \ref{thm-qFuch2}. 
\end{proof}

\subsection{The end and the limit sets.}

\begin{definition} \label{defn-limitset}
\begin{itemize}
\item Define the {\em limit set} $\Lambda(\tilde E)$ of a p-R-end $\tilde E$ with a generalized p-end-neighborhood to be 
$\Bd D - \partial D$ for a generalized lens $D$ of $\tilde E$ in $\SI^n$ {\rm (}resp. $\bR P^n${\rm ).}
\item The {\em limit set} $\Lambda(\tilde E)$ of a p-T-end $\tilde E$ of lens type to be 
$\clo(S_{\tilde E})-S_{\tilde E}$ for the ideal totally geodesic boundary component $S_{\tilde E}$ of $\tilde E$. 
\item The limit set of a horospherical end is the set of the end vertex. 
\end{itemize} 
\end{definition}

\begin{corollary}\label{cor-independence} 
Let $\mathcal{O}$ be a strongly tame $n$-orbifold. 
Suppose that the holonomy group is strongly irreducible.
Let $U$ be a p-end-neighborhood of $\tilde E$ where $\tilde E$ is a lens-type p-T-end 
or a generalized lens-type or lens-type or horospherical p-R-end with admissible end fundamental groups. 
Then $\clo(U) \cap \Bd \torb$
equals $\clo(S_{\tilde E})$ or $\clo(S(\bv_{\tilde E}))$ or $\{\bv_{\tilde E}\}$ depending on 
whether $\tilde E$ is a lens-type p-T-end or a generalized lens-type or lens-type or horospherical p-R-end, 
this set is independent of the choice of $U$ 
and so is the limit set $\Lambda(\tilde E)$ of $\tilde E$.
\end{corollary}
\begin{proof} 
Let $\tilde E$ be a generalized lens-type p-R-end. Then by Theorem \ref{thm-equ}, $\tilde E$ satisfies the uniform middle eigenvalue condition. 
Suppose that $\pi_1(\tilde E)$ is not virtually factorable. 
Let $K^b$ denote $\Bd {\mathcal T}_{\bv_{\tilde E}} \cap K$ for a distanced minimal compact convex set $K$ where $\bGamma_{\tilde E}$ acts on.
Proposition \ref{prop-orbit} shows that the 
limit set is determined by a set $K^b$ in $\bigcup S(v_{\tilde E})$ since $S(v_{\tilde E})$ is an $h(\pi_1(\tilde E))$-invariant set. 
We deduce that $\clo(U) \cap \Bd \torb = \bigcup S(v_{\tilde E})$. 

Also, $\Lambda(\tilde E) \supset K^b$ since $\Lambda(\tilde E)$ is a $\pi_1(\tilde E)$-invariant compact set in 
$\Bd {\mathcal T}_{\bv_{\tilde E}} - \{v_{\tilde E}, v_{\tilde E-}\}$. 
By Proposition \ref{prop-orbit}, each point of $K^b$ is a limit of some $g_i(x)$ for $x \in D$ for a generalized lens. 
Since $D$ is  $\pi_1(\tilde E)$-invariant compact set, we obtain $K^b \subset \Lambda(\tilde E)$. 

Suppose now that $\pi_1(\tilde E)$ acts reducibly. Then by Theorem \ref{thm-redtot}, $\tilde E$ is a totally geodesic p-R-end. 
Proposition \ref{prop-orbit} again implies the result. 

Let $\tilde E$ be a p-T-end. Theorem \ref{thm-qFuch2}(i) implies 
\[\clo(A) - A \subset \clo(S_{\tilde E}) \hbox{ for } A= \Bd L \cap \torb\]
for a lens neighborhood $L$ by the strictness of the lens. 
Thus, $\clo(U) \cap \Bd \torb$  equals $\clo(S_{\tilde E})$. 

For horospherical, we simply use the definition to obtain the result. 
\end{proof}

\begin{definition} 
An {\em SPC-structure} or a {\em stable  irreducible properly-convex real projective structure} on an $n$-orbifold 
is a real projective structure so that the orbifold
with stable and strongly irreducible holonomy. That is, it 
is projectively diffeomorphic to a quotient orbifold of 
a properly convex domain in $\bR P^n$ by a discrete group
of projective automorphisms that is stable and  strongly irreducible.
\end{definition}

\begin{definition}
Suppose that $\mathcal{O}$ has an SPC-structure. Let $\tilde U$ be 
the inverse image in $\tilde{\mathcal{O}}$ of the union $U$ of some choice of a collection of disjoint end neighborhoods of $\orb$.  \index{SPC-structure}
If every straight arc in the boundary of the domain $\tilde{\mathcal{O}}$ 
and every non-$C^1$-point  is contained in the closure of a component of $\tilde U$ for some choice of $U$, 
then $\mathcal{O}$ is said to be {\em strictly convex} with respect to the collection of the ends.  \index{convex!strictly}
And $\mathcal{O}$ is also said to have a {\em strict SPC-structure} with respect to the collection of ends. \index{SPC-structure!strict}
\end{definition}

\begin{corollary} \label{cor-strictconv} 
Suppose that $\mathcal O$ is a strongly tame strictly SPC-orbifold. 
Assume that the holonomy group of $\pi_1(\orb)$ is strongly irreducible.  
Let $\torb$ is a properly convex domain in $\bR P^n$ \rlp resp. in $\SI^n$\rrp \, covering $\orb$.
Choose any disjoint collection of end neighborhoods in $\orb$. Let $U$ denote their union. \index{$p_{\orb}$}
Let $p_{\orb}: \torb \ra \orb$ denote the universal cover. 
Then any segment or a non-$C^1$-point of $\Bd \torb$ is contained in the closure of a component of $p_{\orb}^{-1}(U)$ for 
any choice of $U$. 
\end{corollary}
\begin{proof}
%
By the definition of a strict SPC-orbifold, any segment or a non-$C^1$-point has to be in the closure of 
a p-end neighborhood.  Corollary \ref{cor-independence} proves the claim. 
\end{proof}


\subsection{Expansion of admissible p-end-neighborhoods.} 


\begin{lemma}\label{lem-expand}  
Let $\mathcal{O}$ have a noncompact strongly tame properly convex real projective structure $\mu$. 
Assume that the holonomy group is strongly irreducible.
\begin{itemize} 
\item Let $U_1$ be a p-end neighborhood of a horospherical or a lens-type p-R-end $\tilde E$ with the p-end vertex $v${\rm ;} 
or 
\item Let $U_1$ be a lens-type p-end neighborhood of a p-T-end $\tilde E$.
\end{itemize} 
Let $\bGamma_{\tilde E}$ denote the admissible p-end fundamental group corresponding to $\tilde E$. 
Then we can construct a sequence of lens-cone or lens p-end neighborhoods $U_{i}$, $i=1, 2, \dots, $ where 
$U_{i} \subset U_{j} \subset  \torb$ for $i < j$ where 
the following hold\,{\rm :} 
\begin{itemize} 
\item Given a compact subset of $\tilde{\mathcal{O}}$, there exists an integer $i_0$ such that 
$U_i$ for $i > i_0$ contains it. 
\item The Hausdorff distance between $U_i$ and $\tilde{\mathcal{O}}$ can be made as small as possible, i.e., 
\[ \forall \eps > 0, \exists\, J, J > 0, \hbox{ so that }  \bdd_H (U_i, \torb) < \epsilon \hbox{ for } i > J. \]
\item There exists a sequence of convex open p-end neighborhoods $U_i$ of $\tilde E$ in $\tilde{\mathcal{O}}$ 
so that $(U_i - U_j)/\bGamma_{\tilde E}$ for a fixed $j$ and $i> j$ is diffeomorphic to a product of an open interval with 
the end orbifold. 
\item We can choose $U_i$ so that $\Bd U_i \cap \torb$ is smoothly embedded and strictly convex with 
$\clo(\Bd U_i) - \torb \subset \Lambda(\tilde E)$. 
\end{itemize}
\end{lemma}
\begin{proof} 
%
Suppose that $\tilde E$ is a lens-type R-end first. 
Let $U_1$ be a  lens-cone.
Take a union of finitely many geodesic leaves $L$ from $\bv_{\tilde E}$ in $\torb$ 
of $d_{\torb}$-length $t$ outside the lens-cone $U_1$ and
take the convex hull of $U_1$ and $\bGamma_{\tilde E}(L)$ in $\tilde{\mathcal{O}}$. 
Denote the result by $\Omega_t$. Thus, the endpoints of $L$ not equal to $\bv_{\tilde E}$ are in $\torb$.

We claim that 
\begin{itemize}
\item $\Bd \Omega_t \cap \tilde{\mathcal{O}}$ is a connected $(n-1)$-cell, 
\item $\Bd \Omega_t \cap \tilde{\mathcal{O}}/\bGamma_{\tilde E}$ is 
a compact $(n-1)$-orbifold diffeomorphic to $\Sigma_{\tilde E}$, and  
\item $\Bd U_1 \cap \torb$ bounds 
a compact orbifold diffeomorphic to the product of a closed interval with  
$(\Bd \Omega_t \cap \tilde{\mathcal{O}})/\bGamma_{\tilde E}$: 
\end{itemize} 
First, each leaf of $g(l), g\in \bGamma_{\tilde E}$ for $l$ in $L$ is so that any converging subsequence of 
$\{g_i(l)\}, g_i\in \bGamma_{\tilde E}$, converges to a segment in $S(v)$ for an infinite collection of $g_i$. 
This follows since a limit is a segment in $\Bd \tilde{\mathcal{O}}$ with an endpoint $v$ 
and must belong to $S(v)$ by Theorem \ref{I-thm-affinehoro} of \cite{EDC1}. 

Let $S_1$ be the set of segments with endpoints in $\bGamma_{\tilde E}(L) \cup \bigcup S(v)$.
We define inductively $S_i$ to be the set of simplices with sides in $S_{i-1}$. 
Then the convex hull of $\bGamma_{\tilde E}(L)$ in $\clo(\torb)$ is a union of $S_1 \cup \cdots \cup S_n$. 

We claim that for each maximal segment $s$ in $\clo(\torb)$ from $v$ 
not in $S(v)$, $s^o$ meets $\Bd \Omega_t \cap \torb$ at a unique point: 
Suppose not. 
Then let $v'$ be its other endpoint of $s$ in  $\Bd \tilde{\mathcal{O}}$ 
with $s^{o}\cap \Bd \Omega_t \cap \torb=\emp$. 
Thus, $v' \in \Bd \Omega_{t}$. 

Now, $v'$ is contained in the interior of a simplex $\sigma$ in $S_i$ for some $i$.
Since $\sigma^o \cap \Bd \torb \ne \emp$, $\sigma\subset \Bd \torb$ by Lemma \ref{lem-simplexbd}.
Since the endpoints $\bGamma_{\tilde E}(L)$ are in $\torb$, the only possibility is that 
the vertices of $\sigma$ are in $\bigcup S(v)$. 
Also, $\sigma^{o}$ is transversal to radial rays 
since otherwise $v'$ is not in $\Bd \torb$. 
Thus, $\sigma^{o}$ projects to an open simplex of same dimension in $\tilde \Sigma_{\tilde E}$.
Since $U_1$ is convex and contains $\bigcup S(v)$ in its boundary, 
$\sigma$ is in  the lens-cone $\clo(U_1)$. 
Since a lens-cone has boundary a union of a strictly convex open hypersurface $A$ and $\bigcup S(v)$, 
and $\sigma^{o}$ cannot meet $A$ tangentially, it follows that 
$\sigma^{o}$ is in the interior of the lens-cone.
and no interior point of $\sigma$ is in $\Bd \tilde{\mathcal{O}}$, a contradiction. 
Therefore, each maximal segment $s$ from $v$ meets the boundary 
$\Bd \Omega_t \cap \tilde{\mathcal{O}}$ exactly once. 

As in Lemma \ref{lem-infiniteline}, $\Bd \Omega_t \cap \torb$ contains no line segment ending in $\Bd \torb$.  
The strictness of convexity of $\Bd \Omega_t \cap \torb$ follows as 
by smoothing as in the proof of Proposition \ref{prop-convhull2}. 
By taking sufficiently many leaves for $L$ with $d_{\torb}$-lengths $t$ sufficiently large, we can show that any compact subset is
inside $\Omega_t$. From this, the final item follows. 
The first three items now follow if $\tilde E$ is an R-end. 

Suppose now that $\tilde E$ is horospherical and
$U_1$ is a horospherical p-end neighborhood. 
We can smooth the boundary to be strictly convex. 
Call the set $\Omega_t$ where $t$ is a parameter $\ra \infty$ measuring the distance from $U_1$. 
$\bGamma_{\tilde E}$ is in a parabolic subgroup of a conjugate of $\SO(n, 1)$ by Theorem \ref{I-thm-comphoro} of \cite{EDC1}.  
By taking $L$ sufficiently densely, we can choose similarly to above a sequence $\Omega_i$ of strictly convex horospherical open sets at $v$ 
so that eventually any compact subset of $\tilde{\mathcal{O}}$ is in it for sufficiently large $i$.


Suppose now that $\tilde E$ is totally geodesic. Now we use the dual domain $\torb^*$ and the group $\bGamma_{\tilde E}^*$. 
Let $\bv_{\tilde E^*}$ denote the vertex dual to $S_{\tilde E}$. 
By the diffeomorphism induced by great segments with endpoints $\bv_{\tilde E}^*$, we obtain 
 \[(\Bd \torb^* - \bigcup S(\bv_{\tilde E^*}))/\bGamma_{\tilde E}^* \cong \Sigma_{\tilde E}/\bGamma_{\tilde E}^*,\] 
a compact orbifold.
Then we obtain $U_i$ containing $\torb^*$ in ${\mathcal{T}}_{\tilde E}$
by taking finitely many hypersphere $F_{i}$ disjoint from $\torb^*$ but meeting
${\mathcal{T}}_{\tilde E}$. Let $H_{i}$ be the open hemisphere containing $\torb^*$ bounded by $F_{i}$. 
Then we form $U_1 := \bigcap_{g\in \bGamma_{\tilde E}} g(H_i)$. 
By taking more hyperspheres, we obtain a sequence 
\[U_1 \supset U_2 \supset \cdots \supset U_i \supset U_{i+1} \supset \cdots \supset \torb^* \]
so that $\clo(U_{i+1}) \subset U_i$ and 
\[\bigcap_i \clo(U_i) =\clo(\torb^*). \] 
That is for sufficiently large hyperplanes, we can make 
$U_i$ disjoint from any compact subset disjoint from $\clo(\torb^*)$.
Now taking the dual $U_i^*$ of $U_i$ and by equation \eqref{I-eqn-dualinc} we obtain
\[ U_1^* \subset U_2^* \subset \cdots \subset U_i^* \subset U_{i+1}^* \subset \cdots \subset \torb.\]
Then $U_{i}^* \subset \torb$ is an increasing sequence eventually containing all compact subset of $\torb$. 
This completes the proof for the first three items.

The fourth item follows from Corollary \ref{cor-independence}.
\end{proof} 


\subsection{Convex hulls of ends.} \label{subsub-convh}

We will sharpen Corollary \ref{cor-independence} and the convex hull part in Lemma \ref{lem-expand}.   

One can associate a {\em convex hull} of a p-end $\tilde E$ of $\tilde{\mathcal{O}}$ as follows: 
\begin{itemize}
\item For horospherical p-ends, the convex hull of each is defined to be the set of the end vertex actually. 
\item The convex hull of a totally geodesic p-end $\tilde E$ of lens-type is the closure $\clo(S_{\tilde E})$
the totally geodesic ideal boundary component $S_{\tilde E}$ corresponding to $\tilde E$. 
\item For a generalized lens-type p-end $\tilde E$,  the convex hull 
$I(\tilde E)$ of $\tilde E$ is the convex hull of $\bigcup S(\bv_{\tilde E})$ in $\clo(\torb)$.  
\end{itemize} 
The first two equal $\clo(U) \cap \Bd \torb$ for any p-end neighborhood $U$ of $\tilde E$ by
Corollary \ref{cor-independence}.

Corollary \ref{cor-independence}  and Proposition \ref{prop-I} imply that 
the convex hull of an end  is well-defined.

For a lens-shaped p-end $\tilde E$ with a p-end vertex $\bv_{\tilde E}$,  the {\em proper convex hull} $I(\tilde E)$ is defined as 
\[CH(\bigcup S(\bv_{\tilde E})) \cap \tilde{\mathcal{O}}.\] 
We can also characterize it as the intersection 
\[\bigcap_{U_{1}\in \mathcal U}CH(\clo(U_1)) \cap \tilde{\mathcal{O}}\] for 
the collection $\mathcal U$ of p-end neighborhoods $U_1$ of $\bv_{\tilde E}$
by (iv) and (v) of Proposition \ref{prop-I}.

\begin{proposition}\label{prop-I} 
Let $\orb$ be a strongly tame properly convex real projective orbifold with radial ends or totally geodesic ends 
of lens-type and satisfy {\em (IE)} and {\em (NA)}.
Assume that the holonomy group of $\pi_1(\orb)$ is strongly irreducible. 
Let $\tilde E$ be a radial lens-shaped p-end and $v$ an associated 
p-end vertex. Let $I(\tilde E)$ be the convex hull of $\tilde E$. 
\begin{itemize}
\item[{\rm (i)}] $\Bd I(\tilde E) \cap \tilde{\mathcal{O}}$ is contained in the union of a lens part of a lens-shaped p-end neighborhood. 
\item[{\rm (ii)}] $I(\tilde E)$ contains any concave p-end-neighborhood of $\tilde E$ and 
\[I(\tilde E) \cap \torb  = CH(\clo(U)) \cap \tilde{\mathcal{O}}\] for a concave p-end neighborhood $U$ of $\tilde E$. 
Thus, $I(\tilde E)$ has a nonempty interior. 
\item[{\rm (iii)}] Each segment from $v$ maximal in $\tilde{\mathcal{O}}$ 
meets the set $\Bd I(\tilde E)\cap  \tilde{\mathcal{O}}$ at most once and
$\Bd I(\tilde E) \cap  \tilde{\mathcal{O}}/\bGamma_v$ is an orbifold isotopic to $E$
for the end fundamental group $\bGamma_v$ of $v$. 
\item[{\rm (iv)}] There exists a nonempty interior of the convex hull $I(\tilde E)$ of $\tilde E$ 
where $\bGamma_v$ acts so that $I(\tilde E) \cap \tilde{\mathcal{O}}/\bGamma_v$ is diffeomorphic to the end orbifold times an interval. 
\end{itemize}
\end{proposition}
\begin{proof}
(i) We define $S_1$ as the set of $1$-simplices with endpoints in segments in $\bigcup S(v)$ and we inductively define
$S_i$ to be the set of $i$-simplices with faces in $S_{i-1}$. 
Then 
\[I(\tilde E) = \bigcup_{\sigma \in S_1 \cup S_2 \cup \cdots \cup S_n} \sigma.\] 
Notice that $\Bd I(\tilde E)$ is the union 
{\large 
\[\bigcup_{\sigma \in S_1 \cup S_2 \cup \cdots \cup S_n, \, \sigma \subset \Bd I(\tilde E)} \sigma\]
}
since each point of $\Bd I(\tilde E)$ is contained in the interior of a simplex which lies in $\Bd I(\tilde E)$ by the convexity of $I(\tilde E)$. 

If $\sigma \in S_1$ with $\sigma \subset \Bd I(\tilde E)$, then its endpoint must be in an endpoint of a segment in $\bigcup S(v)$:
otherwise, $\sigma^{o}$ is in the interior of $I(\tilde E)$. 
If an interior point of $\sigma$ is in a segment in $S(v)$, then the vertices of $\sigma$ are in 
$\bigcup S(v)$ by the convexity of $\clo(R_v(\torb))$. 
Hence, if $\sigma^o \subset \Bd I(\tilde E) \cap \torb$ meets $\tilde{\mathcal{O}}$, then 
$\sigma^o$ is contained in the lens-shaped domain $L$ as the vertices of $\sigma$ is in $\Bd L - \partial L$ 
by the convexity of $L$. 
Now by induction on $S_i$, $i > 1$, we can verify (i)
since any simplex with boundary in the union of subsimplices in the lens-domain is in the lens-domain
by convexity.

(ii) Since $I(\tilde E)$ contains the segments in $S(v)$ and is convex, and so does a concave p-end neighborhood $U$, 
we obtain $\Bd U \subset I(\tilde E)$: 
Otherwise, let $x$ be a point of $\Bd U \cap \Bd I(\tilde E) \cap \torb$ where some neighborhood 
in $\Bd U$ is not in $I(\tilde E)$. Then since $\Bd U$ is a union of a convex hypersurface 
$\Bd U \cap \torb$ and $S(v)$, 
each supporting hyperspace at $x$ of the convex set $\Bd U \cap \torb$
 meets a segment in $S(v)$ in its interior. 
This is a contradiction since $x$ must be then in $I(\tilde E)^{o}$. 
Thus, $U \subset I(\tilde E)$. Thus, $CH(\clo(U)) \subset I(\tilde E)$. 
Conversely, since $\clo(U) \supset \bigcup S(v)$ by Theorems \ref{thm-lensclass} and \ref{thm-redtot}, we obtain that 
$CH(\clo(U)) \supset I(\tilde E)$. 

(iii) $\Bd I(\tilde E) \cap  \tilde{\mathcal{O}}$ is a subset of a lens part of a p-end neighborhood by (iii). 
Each point of it meets a maximal segment from $v$ in the end but not in $S(v)$ at exactly one point since 
a maximal segment must leave the lens cone eventually.
Thus $\Bd I(\tilde E) \cap  \tilde{\mathcal{O}}$ is homeomorphic to an $(n-1)$-cell and the result follows. 

(iv) This follows from (iii) since we can use rays from $x$ meeting $\Bd I(\tilde E) \cap  \tilde{\mathcal{O}}$ at unique points 
and use them as leaves of a fibration. 


\end{proof}


\begin{figure}
\centerline{\includegraphics[height=7cm]{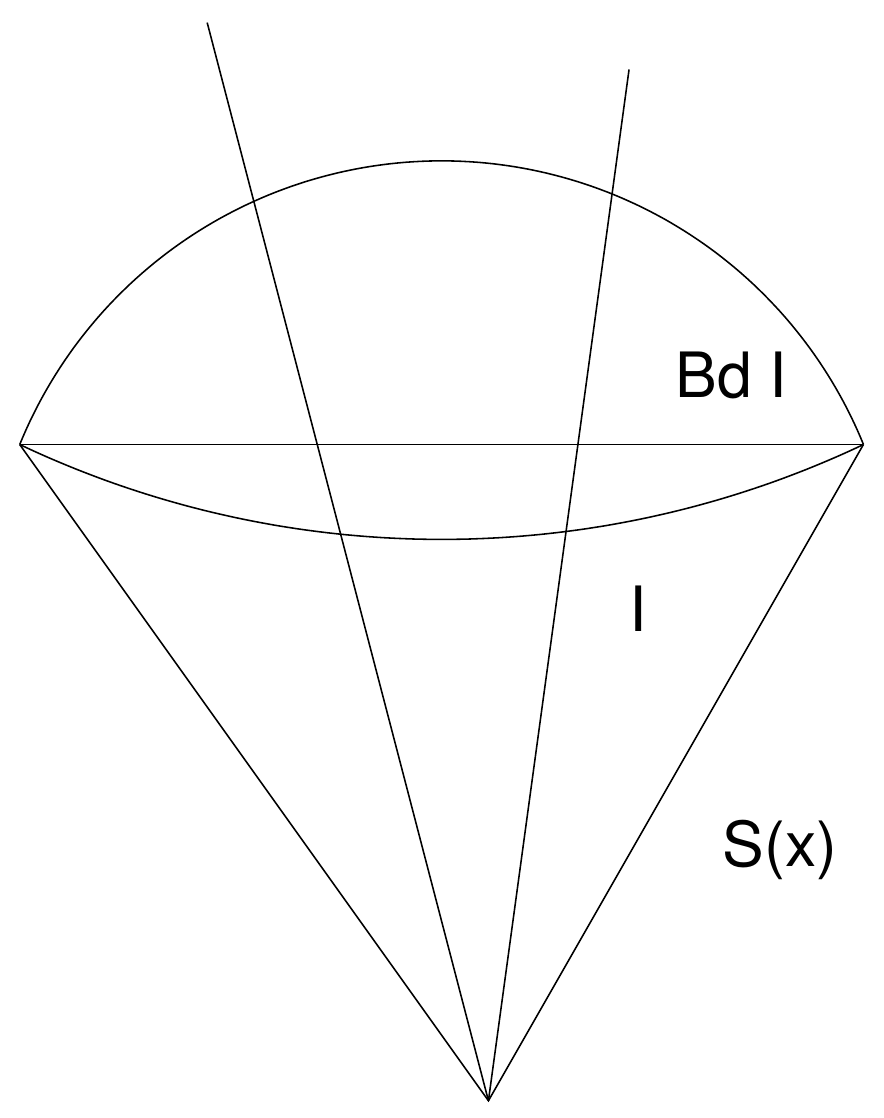}}
\caption{The structure of a lens-shaped p-end.}
\label{fig:lenslem2}
\end{figure}

\subsection{Shrinking of lens and horospherical p-end-neighborhoods.} 

We now discuss the ``shrinking'' of p-end-neighborhoods. These repeat some results. 
\begin{corollary} \label{cor-shrink} 
Suppose that $\orb$ is a properly convex real projective orbifold 
and let $\torb$ be a properly convex domain in $\SI^n$ {\rm (}resp. $\bR P^n${\rm )} covering $\orb$. 
Assume that the holonomy group is strongly irreducible.
Then the following statements hold\,{\rm :} 
\begin{itemize} 
\item[{\rm (i)}] If $\tilde E$ is a horospherical p-R-end, 
every p-end-neighborhood of $\tilde E$ contains a horospherical p-end-neighborhood.
\item[{\rm (ii)}] Suppose that $\tilde E$ is a generalized lens-shaped or lens-shaped p-R-end. 
Let $I(\tilde E)$ be the convex hull of $\bigcup S(\bv_{\tilde E})$, 
and let $V$ be a  p-end-neighborhood $V$ where $(\Bd V \cap \torb)/\pi_1(\tilde E)$ is a compact orbifold. 
If $V^o \supset I(\tilde E) \cap \torb$, 
$V$ contains a lens-cone p-end neighborhood of $\tilde E$, 
and a lens-cone contains $\torb$ properly. 
\item[{\rm (iii)}] If $\tilde E$ is a generalized lens-shaped p-R-end or satisfies the uniform middle eigenvalue condition, 
every p-end-neighborhood of $\tilde E$ contains a concave p-end-neighborhood.
\item[{\rm (iv)}] Suppose that $\tilde E$ is a p-T-end of lens type or satisfies the uniform middle eigenvalue condition.
Then every p-end-neighborhood contains a lens p-end-neighborhood $L$ with strictly convex boundary in $\torb$. 
\end{itemize} 
\end{corollary} 
\begin{proof}
Let us prove for $\SI^n$. 

(i) Let $v_{\tilde E}$ denote the p-R-end vertex corresponding to $\tilde E$. 
By Theorem \ref{I-thm-comphoro}, we obtain a conjugate $G$ of a parabolic subgroup of $\SO(n, 1)$ as the finite index subgroup of $h(\pi_1(\tilde E))$
acting on $U$, a p-end-neighborhood of $\tilde E$. We can choose a $G$-invariant ellipsoid of $\bdd$-diameter $\leq \eps$ for any $\eps > 0$ 
in $U$ containing $v_{\tilde E}$.    

(ii) This follows from Proposition \ref{prop-convhull2} since 
the convex hull of $\bigcup S(\bv_{\tilde E})$ has the right properties. 

(iii) Suppose that we have a lens-cone $V$
that is a p-end-neighborhood equal to $L \ast v_{\tilde E} \cap \torb$ 
where $L$ is a generalized lens bounded away from $v_{\tilde E}$. 

By taking smaller $U$ if necessary, we may assume that $U$ and $L$ are disjoint. 
Since $\Bd U/h(\pi_1(\tilde E))$ and $L/h(\pi_1(\tilde E))$ are compact, $\eps > 0$. 
Let
\[L' := \{x \in V | d_{V}(x, L) \leq \eps\}.\] 
Since a lower component of $\partial L$ is strictly convex, 
we can show that $L'$ is a generalized lens by Lemma \ref{lem-nhbd}. 
Clearly, $h(\pi_1(\tilde E))$ acts on $L'$. 

We choose sufficiently large $\eps'$ so that
 $\Bd U \cap \torb \subset L'$, and hence $V-L'\subset U$ form a concave p-end-neighborhood as above. 


(iv) The existence of a lens-type p-end neighborhood of $S_{\tilde E}$ follows from Theorem \ref{thm-lensn}. 
\end{proof}

\subsection{T-ends and the ideal boundary.} \label{subsec-totdual}

We discuss more on T-ends. 
For T-ends, by the lens condition, we only consider the ones that have lens neighborhoods in some ambient orbifolds, 
First, we discuss the extension to bounded orbifolds. 

\begin{theorem}\label{thm-totgeoext} 
Suppose that $\mathcal O$ is a strongly tame properly convex real projective orbifold 
with generalized lens or horospherical ends and satisfy {\em (IE)}. 
Assume that the holonomy group of $\pi_1(\orb)$ is strongly irreducible. 
Let $E$ be a lens-shaped p-T-end, and let $\Sigma_E$ be a totally geodesic hypersurface that is the ideal 
boundary corresponding to $E$. 
 Let $L$ be a lens-shaped end neighborhood of $\Sigma_E$ in an ambient real projective orbifold containing $\orb$. 
 Then 
 \begin{itemize} 
\item $L \cup \orb$ is a properly convex real projective orbifold 
 and has a strictly convex boundary component 
 corresponding to $E$. 
\item  Furthermore if $\orb$ is strictly SPC and $\tilde E$ is a hyperbolic end, then so is $L \cup \orb$
 which now has one more boundary component and one less T-ends. 
 \end{itemize} 
\end{theorem} 
\begin{proof} 
It is sufficient to prove for $\SI^n$ cases here. 
Let $\torb$ be the universal cover of $\orb$ which we can identify with a properly convex bounded domain in 
an affine subspace. Then $\Sigma_E$ corresponds to a p-T-end $\tilde E$ and 
 to a totally geodesic hypersurface $S= S_{\tilde E}$. And $L$ is covered by 
 a lens $\tilde L$ containing $S$. The p-end fundamental group $\pi_1(\tilde E)$ acts on
 $\torb$ and $\tilde L_1$ and $\tilde L_2$ the two components of $\tilde L - S_{\tilde E}$ in $\torb$ and
 outside $\torb$ respectively. 
 




\begin{definition}\label{defn-asymp} 
Let $\bR^n$ denote the affine subspace in $\SI^n$ with boundary $\SI^{n-1}_\infty$. 
Suppose that $\Omega$ is a properly convex open domain in $\SI^{n-1}_\infty$. 
Let $\Omega_1$ be a properly convex open domain with $\Bd \Omega_1 \supset \clo(\Omega)$ in $\bR^n$. 
The supporting hyperplanes at $p \in \Lambda = \clo(\Omega)-\Omega$ 
contains a hyperplane of codimension-two supporting $\Omega$.
Let \[A_p :=\{ H| H \hbox{ is a supporting hyperspace of $\Omega_1$ at $p$ in $\bR^n$} \}. \]
An {\em asymptotic supporting hyperplane} $h$ at a point $p$ of $\Lambda$ is a supporting hyperplane at $p$
so that there exists no other element $h'$ of $A_p$ with 
\[\clo(h) \cap \SI^{n-1}_\infty  = \clo(h') \cap \SI^{n-1}_\infty \]
closer to $\Omega_1$ from a point of $\Bd \Omega_1 - \clo(\Omega)$
(using minimal distance between a point and a set).
\end{definition} 

\begin{lemma} \label{lem-commsupp} 
Suppose that $S_{\tilde E}$ is the totally geodesic ideal boundary of 
a lens-type T-end $\tilde E$ of a strongly tame real projective orbifold $\orb$
and $\pi_1(\tilde E)$ is nontrivial hyperbolic. 
\begin{itemize} 
\item Given a $\pi_1(\tilde E)$-invariant properly convex open domain $\Omega_1$ containing $S_{\tilde E}$ in the boundary, 
at each point of $\Lambda$, there exists a unique asymptotic supporting hyperplane. 
\item At each point of $\Lambda$, the hyperspace supporting any $\pi_1(\tilde E)$-invariant 
properly convex open set $\Omega$ containing $S_{\tilde E}$  is unique. 
\item We are given two $\pi_1(\tilde E)$-invariant properly convex open domains $\Omega_1$ containing $S_{\tilde E}$ in the boundary
and $\Omega_2$ containing $S_{\tilde E}$ in the boundary from the other side. 
Then $\Omega_1 \cup \Omega_2$ is a convex domain
with \[\clo(\Omega_1) \cap \clo(\Omega_2) = \clo(S_{\tilde E})\] and their asymptotic supporting hyperplanes at each point of $\Lambda$ coincide.
\end{itemize} 
\end{lemma} 
\begin{proof} 
Let $A$ denote the affine subspace that is the complement in $\SI^n$ of the hyperspace containing $S_{\tilde E}$. 
Because $\pi_1(\tilde E)$ acts on a lens-type domain, 
the dual group of $h(\pi_1(\tilde E))$ is the holonomy group 
of a lens-type p-R-end by Corollary \ref{cor-duallens}. 
By Theorem \ref{thm-equ}, $h(\pi_1(\tilde E))$ satisfies the uniform middle eigenvalue condition. 

If $\Omega_1$ has an asymptotic supporting half-space $H(x)$ for each $x \in \Lambda$ containing $\Omega_1$. 
$H(x)$ is uniquely determined by $\pi_1(\tilde E)$ and $x$ by Lemma \ref{lem-inde} and its proof. 


The third item follows since the asymptotically supporting hyperplane at each point of $\clo(S_{\tilde E}) - S_{\tilde E}$ 
to $\Omega_1$ and $\Omega_2$ have to agree by Lemma \ref{lem-inde}(ii). 
The convexity follows easily from this. Also, the second item follows. 


\end{proof}

We continue with the proof of Theorem \ref{thm-totgeoext}.  
Suppose that $\pi_1(\tilde E)$ is hyperbolic. 
By Lemma \ref{lem-commsupp}, $\tilde L_2 \cup S_{\tilde E} \cup \torb$ is a convex domain. 
If $\tilde L_2 \cup \torb$ is not properly convex, then it is a union of 
two cones over $S_{\tilde E}$ over 
of $[\pm v_x ] \in \bR^{n+1}, [v_x] = x$. 
This means that $\torb$ has to be a cone contradicting the irreducibility of $h(\pi_1(\orb))$. 
Hence, it follows that $\tilde L_2 \cup \torb$ is properly convex. 

Suppose that $\orb$ is strictly SPC and $\pi_1(\tilde E)$ is hyperbolic. 
Then every segment in $\Bd \torb$ or a non-$C^1$-point in $\Bd \torb$
is in the closure of one of the p-end neighborhood.
$\Bd \tilde L_2 - \clo(S_{\tilde E})$ does not contain any segment in it or a non-$C^1$-point. 
$\Bd \torb - \clo(S_{\tilde E})$ does not contain any segment or a non-$C^1$-point outside 
the union of the closures of p-end neighborhoods. 
$\Bd(\torb \cup \tilde L_2 \cup S_{\tilde E})$ is $C^1$
at each point of 
$\Lambda(\tilde E) := \clo(S_{\tilde E}) - S_{\tilde E}$
by the uniqueness of the supporting hyperplanes of Lemma \ref{lem-commsupp}. 

Recall that $S_{\tilde E}$ is strictly convex 
since $\pi_1(\tilde E)$ is a hyperbolic group. (See Theorem 1.1 of \cite{Ben1}.)
Thus, $\Lambda$ does not contain a segment, and hence,
$\Bd(\torb \cup \tilde L_2 \cup S_{\tilde E})$ does not contain one. 
Therefore, $L_2 \cup \orb$ is strictly convex relative to the ends.

Suppose now that $\pi_1(\tilde E)$ is virtually factorable. 
Then the dual of the p-T-end is a radial p-end by Proposition \ref{prop-dualend}. 
The dual p-R-end has a p-end neighborhood 
that is contained in a strict join with a vertex $x$ with a properly convex open domain $K$ in a hyperplane $V$. 
$\clo(K)$ is a strict join $C_1 \ast \cdots \ast C_k$ for 
properly compact convex domains $C_i$, for $i=1, \dots, k$ by Theorem \ref{thm-redtot}. 

Recall that $\torb$ contains an open one-sided properly convex p-end neighborhood $D$ of $S_{\tilde E}$. 
By equation \eqref{I-eqn-dualinc} of \cite{EDC1}, the dual $D^*$ of $D$ contains the dual $\torb^*$ of $\torb$.
Let $x$ be a dual point to the hyperplane containing ideal boundary component $S_{\tilde E}$.
$D^*$ is  the interior of a lens-cone with end vertex $x$ by Corollary \ref{cor-duallens}. 
By Theorem \ref{thm-redtot},  $D^{\ast}$ is a totally geodesic lens-cone with end vertex $x$. 
$D^*$  is contained  in the union $U$ of two strict joins $x\ast K \cup x_- \ast K$. 
Thus, $\torb^* \subset x\ast K \cup x_- \ast K$. 
However, $D^*$ contains $x \ast K$.

The set of supporting hyperspaces at the vertex $x$ is projectively
isomorphic to the dual $K$ of $\clo(S_{\tilde E})$ by Proposition \ref{prop-dualend}.
Let $V$ be the hyperspace containing $K$. 
Since $D^*$ contains $x \ast K$, $D$ is contained in 
$(x \ast K)^{*} = a \ast \clo(S_{\tilde E})$ for the point $a$ dual to the hyperplane $V$
by equation \eqref{I-eqn-dualjoin} of \cite{EDC1}. 
Therefore, the dual $\torb$ of $\torb^*$ is contained in the 
the cone $\clo(S_{\tilde E}) \ast a$ for some point $a$ dual to the hyperplane $V$. 

Now, $\tilde L_2$ is a subset of $\clo(S_{\tilde E}) \ast a_-$ sharing boundary $\clo(S_{\tilde E})$ with $\torb$
since we can treat $\tilde L_2$ as $\torb$ in the above arguments. 
Since both share $S_{\tilde E}$ and are in $S_{\tilde E} \ast a \cup S_{\tilde E} \ast a_-$, 
the convexity of the union $\tilde L_2 \cup \torb$ follows. 
The proper convexity follows also as above. 

Since $\tilde L_{2}\cup \torb$ has a Hilbert metric, the action is properly discontinuous. 
\end{proof}

\section{Application: The strong irreducibility of the real projective orbifolds.}\label{sec-strirr}

The main purpose of this section is to prove 
Theorem \ref{thm-sSPC}, the strong irreducibility result. 
In particular, we don't assume the holonomy group of $\pi_{1}(\orb)$ is strongly irreducible for results from now on. 
But we will discuss the convex hull of the ends first. 
We show that the closure of convex hulls of p-end neighborhoods are disjoint in $\Bd \torb$. 
The infinity of the number of these will show the strong irreducibility.

\subsection{The limit sets and convex hull of ends, mc-p-end neighborhoods}

The mc-p-end neighborhood will be useful in other papers. 

\begin{definition}\label{defn-lambda}
 Let $\tilde E$ be a lens-type R-end. Let $L$ be the lens-cone p-end neighborhood of ${\tilde E}$. \index{end!mc-p-end neighborhood}
 Let $CH(\Lambda(\tilde E))$ denote the convex hull of $\Lambda(\tilde E)$. 
Let $U'$ be any p-end neighborhood $U'$ of $\tilde E$ containing $CH(\Lambda(\tilde E)) \cap \torb$. 
We define a {\em maximal concave p-end neighborhood} or {\em mc-p-end-neighborhood} $U$ 
 to be one of the two components of $U' - CH(\Lambda(\tilde E))$ containing 
a p-end neighborhood of $\tilde E$. 
The {\em closed maximal concave p-end neighborhood} is 
$\clo(U) \cap \torb$. 
An $\eps$-$d_{\torb}$-neighborhood $U''$ of a maximal concave p-end neighborhood is called 
an {\em $\eps$-mc-p-end-neighborhood}.  \index{end!$\eps$-mc-p-end neighborhood}.
\end{definition} 
In fact, these are independent of choices of $U'$. 
Note that a maximal concave p-end neighborhood $U$ is uniquely determined since $\Lambda(\tilde E)$ is. 

Each radial segment $s$ in $\torb$ from $\bv_{\tilde E}$ meets $\Bd U \cap \torb$ at a unique point
since $s \cap \Bd U$ is in a disk $D$ supporting $CH(\Lambda(\tilde E))$ with $\partial D \subset S(\bv_{\tilde E})$. 

\begin{lemma} \label{lem-mcc}
Let $D$ be an $i$-dimensional totally geodesic compact convex domain, $i \geq 1$. 
Let $\tilde E$ be a generalized lens-type p-R-end with the p-end vertex $v_{\tilde E}$. 
Suppose $\partial D \subset \bigcup S(v_{\tilde E})$. Then $D \subset V$ for 
a maximal concave p-end neighborhood $V$, and 
for sufficiently small $\eps>0$, an $\eps$-$d_{\torb}$-neighborhood of 
$D^o $ is contained in $V'$ for any $\eps$-mc-p-end neighborhood $V'$.
\end{lemma}
\begin{proof} 
Assume that $U$ is a generalized lens-cone of $v_{\tilde E}$. 
Then $\Lambda$ is the set of endpoints of segments in $S_{v_{\tilde E}}$ with $v_{\tilde E}$ removed. 
Let $P$ be the subspace spanned by $D \cup \{v_{\tilde E}\}$. 
Since $\partial D, \Lambda \cap P  \subset \bigcup S(v_{\tilde E}) \cap P$, 
and $\partial D \cap P$ is closer than $\Lambda \cap P$ from $v_{\tilde E}$, 
it follows that 
$P\cap \clo(U) -D$ has a component $C_1$ containing $v_{\tilde E}$ and a component $C_2$ contains $\Lambda \cap P$.
Hence $\clo(C_2) \supset CH(\Lambda) \cap P$ by the convexity of $\clo(C_2)$. 
Since $CH(\Lambda)\cap P$ is a convex set in $P$, we have one of the two possibilities
\begin{itemize}
\item $D$ is disjoint from $CH(\Lambda)^o$ or 
\item $D$ contains $CH(\Lambda) \cap P$.
\end{itemize} 
Let $V$ be an mc-p-end neighborhood of $U$. 
Since $\clo(V)$ contains the closure of the component of $U - CH(\Lambda)$ whose closure contains $v_{\tilde E}$, 
it follows that $\clo(V)$ contains $D$. 

Since $D$ is in $\clo(V)$, the boundary $\Bd V' \cap \torb$ of the 
$\eps$-mc-p-end neighborhood $V'$ do not meet $D$. Hence $D^o \subset V'$. 
\end{proof} 


\begin{corollary} \label{cor-mcn} 
Let $\orb$ be a  properly convex real projective orbifold with  
 lens-shaped R-ends, lens-type T-ends,
or horospherical ends,
and satisfies {\em (IE)} and {\em (NA)}.
Let $\tilde E$ be a generalized lens-type R-end. 
Then 
\begin{itemize}
\item[{\rm (i)}] A concave p-end neighborhood of $\tilde E$ is always a subset of an mc-p-end-neighborhood of the same p-R-end. 
\item[{\rm (ii)}] The closed mc-p-end-neighborhood of $\tilde E$ is 
the closure in $\torb$ of a union of all concave end neighborhoods of $\tilde E$.
\item[{\rm (iii)}] The mc-p-end-neighborhood  of $\tilde E$ is a proper p-end neighborhood, and covers  an end-neighborhood with compact boundary in $\orb$. 
\item[{\rm (iv)}] An $\eps$-mc-p-end-neighborhood of $\tilde E$ for sufficiently small $\eps > 0$ is a proper p-end neighborhood. 
\item[{\rm (v)}] For sufficiently small $\eps> 0$, 
the image end-neighborhoods in $\orb$ of $\eps$-mc-p-end neighborhoods of p-R-ends
are mutually disjoint. 
\end{itemize}
\end{corollary} 
\begin{proof}
(i) Since the limit set $\Lambda(\tilde E)$ is in any generalized lens by Corollary \ref{cor-independence}, 
a generalized lens-cone p-end neighborhood $U$ of $\tilde E$ contains $CH(\Lambda) \cap \torb$. 
Hence, a concave end neighborhood is contained in an mc-p-end-neighborhood. 



(ii) 
Let $V$ be an mc-p-end neighborhood of $\tilde E$.
Then define $S$ to be the set of endpoints in $\clo(\torb)$ 
of maximal segments in $V$ from $v_{\tilde E}$ in directions of $S_{\tilde E}$. 
Then $S$ is diffeomorphic to $S_{\tilde E}$ by the map induced by radial segments 
as shown in the paragraph before 
Thus, $S/\pi_1(\tilde E)$ is a compact set since $S$ is contractible and $S_{\tilde E}/\pi_1(\tilde E)$ is a $K(\pi_1(\tilde E))$-space. 
We can $d_{\torb}$-approximate $S$ by the piecewise linear boundary component $S_{\eps}$ outwards
of a generalized lens as in Section \ref{subsub:umecorbit}
since $\tilde E$ has the uniform middle-eigenvalue condition. 
We smooth this component. 
A component $U-S_{\eps}$ is a concave p-end neighborhood. 
(ii) follows from this. 

(iii) Since a concave p-end neighborhood is a proper p-end neighborhood by Theorems \ref{thm-lensclass}(iv) and \ref{thm-redtot}(vi), 
we obtain  
\[g(V) \cap V = \emp \hbox{ or } g(V) = V \hbox{ for } g \in \pi_1(\orb) \hbox{ by (ii).} \] 

Suppose that $g(\clo(V) \cap \torb) \cap \clo(V) \ne \emp$. Then $g(V) = V$ and $g \in \pi_1(\tilde E)$:
Otherwise, $g(V) \cap V =\emp$, and $g(\clo(V) \cap \torb)$ meets $\clo(V)$ in a totally geodesic hypersurface $S$ equal to $CH(\Lambda)^o$
by the concavity of $V$. Hence for every $g \in \pi_1(\orb)$, $g(S) = S$, since $S$ is a maximal 
totally geodesic hypersurface in $\torb$, 
and $g(V) \cup S \cup V = \torb$ since these are subsets of a properly convex domain $\torb$.
Then $\pi_1(\orb)$ acts on $S$ and $S/G$ is homotopy 
equivalent to $\torb/G$ for a finite-index torsion-free subgroup $G$ of $\pi_1(\orb)$ by Selberg's lemma. 
This contradicts the condition {\rm (IE)}.
Hence, we conclude that $g(V\cup S) \cap V \cup S = \emp$ or $g(V\cup S) = V \cup S$ for $g \in \pi_{1}(\orb)$.



Now suppose that $S \cap \Bd \torb \ne \emp$. 
Let $S'$ be a maximal totally geodesic domain in $\clo(V)$ supporting $S$. 
Then $S' \subset \Bd \torb$ by convexity and Lemma \ref{lem-simplexbd}, 
meaning that $S'=S \subset \Bd \torb$.  
In this case, $\torb$ is a cone over $S$ and the end vertex $v_{\tilde E}$ of $\tilde E$.
For each $g \in \pi_1(\orb)$, $g(V) \cap V \ne \emp$ meaning $g(V)=V$ since $g(v_{\tilde E})$ is on $\clo(S)$. 
Thus, $\pi_1(\orb) = \pi_1(\tilde E)$. 
This contradicts the infinite index condition of $\pi_1(\tilde E)$. 

We showed that $\clo(V) \cap \torb = V \cup S$. 
Thus, an mc-p-end-neighborhood $\clo(V) \cap \torb$ is a proper end neighborhood of $\tilde E$
with compact imbedded boundary $S/\pi_1(\tilde E)$. 
Therefore we can choose positive $\eps$ so that an $\eps$-mc-p-end-neighborhood is a proper p-end neighborhood also.
This proves (iv). 

(v) For two mc-p-end neighborhoods $U$ and $V$ for different p-R-ends, we have $U \cap V =\emp$
by (iii). 

We showed that $\clo(V) \cap \torb$ for an mc-p-end-neighborhood $V$ covers an end neighborhood in $\orb$. 
Suppose that $U$ is another mc-p-end neighborhood different from $V$.
Similar to above (v), we obtain $\clo(U) \cap \clo(V) \cap \torb = \emp$.  

Since the closures of mc-p-end neighborhoods with different p-ends are disjoint, 
and these have compact boundary components, 
the final item follows. 
\end{proof}

\subsection[The strong irreducibility]{The strong irreducibility and stability of the holonomy group of properly convex strongly tame
orbifolds.}

For the following, we need a stronger condition of lens-type ends
to obtain the disjointedness of the closures of p-end neighborhoods. 
\begin{corollary} \label{cor-disjclosure} 
Let $\orb$ be a  strongly tame  properly convex real projective orbifold with 
generalized lens-shaped R-ends, lens-type T-ends,
or horospherical ends, 
and satisfy {\em (IE)} and {\em (NA)}.
Let $\mathcal U$ be the collection of the components of the inverse image in $\torb$ 
of the union of disjoint collection of  end neighborhoods of $\orb$. 
Now replace each of the p-end neighborhoods of radial lens-type 
of collection $\mathcal U$ by a concave p-end neighborhood by Corollary \ref{cor-shrink} {\rm (iii).} 
Then the following statements hold\,{\em :} 
\begin{itemize} 
\item[{\rm (i)}] Given horospherical, concave, or one-sided lens p-end-neighborhoods $U_1$ and $U_2$ contained in $\bigcup \mathcal U$, 
we have $U_1 \cap U_2 =\emp$ or $U_1= U_2$. 
\item[{\rm (ii)}] Let $U_1$ and $U_2$ be in $\mathcal U$. Then 
$\clo(U_1) \cap \clo(U_2) \cap  \Bd \tilde{\mathcal{O}} = \emp$ 
or $U_1 = U_2$ holds. 
\end{itemize}
\end{corollary}
\begin{proof} 
%
%
(i) Suppose that $U_1$ and $U_2$ are p-end neighborhoods of p-R-ends. 
 Let $U'_1$ be the interior of the associated generalized lens-cone of $U_1$ in $\clo(\torb)$ and $U'_2$ be that of $U_2$. 
Let $U''_i$ be the concave p-end-neighborhood of $U'_i$ for $i=1,2$ that covers an end neighborhood in $\orb$ 
by Corollary \ref{cor-shrink} (iii).
Since the neighborhoods in $\mathcal U$ are mutually disjoint, 
\begin{itemize}
\item $\clo(U''_1) \cap \clo(U''_2) \cap \torb = \emp$ or
\item $U''_1 = U''_2$. 
\end{itemize} 

(ii) Assume that $U''_i \in {\mathcal{U}}$, $i=1, 2$, and $U''_1 \ne U''_2$. 
Suppose that the closures of $U''_1$ and $U''_2$ intersect in $\Bd \torb$.
Suppose that they are both p-R-end neighborhoods. 
Then 
the respective closures of convex hulls $I_1$ and $I_2$ as obtained by Proposition \ref{prop-I} intersect as well. 
Take a point $z \in \clo(U''_1) \cap \clo(U''_2) \cap \Bd \torb$. 
Let $p_1$ and $p_2$ be the respective p-end vertices of $U'_1$ and $U'_2$. 
We assume that $\ovl{p_1p_2}^o \subset \torb$. 
Then $\ovl{p_1z}\in S(p_1)$ and $\ovl{p_2z} \in S(p_2)$ and these segments
are maximal since otherwise $U''_1 \cap U''_2 \ne \emp$.
The segments intersect transversally at $z$ 
since otherwise we violated the maximality in Theorems \ref{thm-lensclass} and 
\ref{thm-redtot}.
We obtain a triangle $\tri(p_1p_2z)$ in $\clo(\torb)$ with vertices $p_1, p_2, z$. 

Suppose now that $\ovl{p_1p_2}^o \subset \Bd \torb$. 
We need to perturb $p_1$ and $p_2$ inside $\Bd \torb$ 
by a small amount so that $\ovl{p_{1}p_{2}} \subset \torb$. 
Let $P$ be the $2$-dimensional plane containing $p_{1}, p_{2}, z$. 
Consider a disk $P \cap \clo(\torb)$ containing $p_{1}, p_{2}, z$ in the boundary. 
However, the disk has an angle $\leq \pi$ at $z$ since $\clo(\torb)$ is properly convex. 
We will denote the disk by $\tri(p_1p_2z)$ and $p_{1}, p_{2}, z$ are considered as vertices. 

We define a convex curve $\alpha_i := \tri(p_1p_2z) \cap \Bd I_i$ with an endpoint $z$ for each $i$, $i=1,2$. 
Let $\tilde E_i$ denote the p-R-end corresponding to $p_i$. 
Since $\alpha_i$ maps to a geodesic in $R_{p_i}(\torb)$, 
there exists a foliation $\mathcal{T}$ of $\tri(p_1p_2z)$ 
by maximal segments from the vertex $p_1$.  
There is a natural parametrization of the space of leaves by $\bR$ 
as the space is projectively equivalent to an open interval using the Hilbert metric of
the interval. We parameterize $\alpha_i$ by these parameters 
as $\alpha_i$ intersected with a leaf is a unique point. 
They give the geodesic length parameterizations under the Hilbert metric  of $R_{p_i}(\torb)$
for $i=1, 2$. 

\begin{figure}
\centerline{\includegraphics[height=6cm]{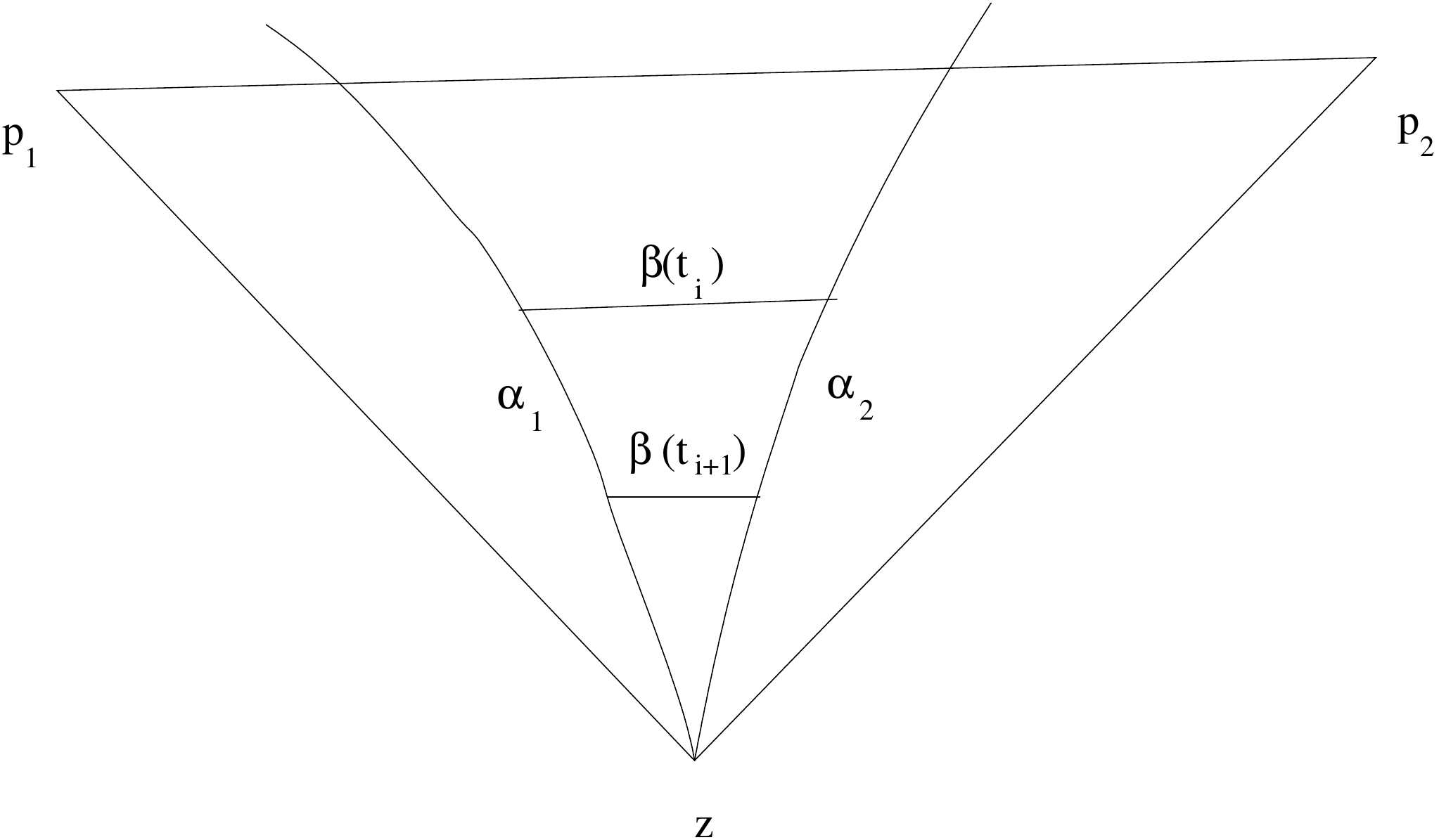}}
\caption{The diagram of the quadrilateral bounded by $\beta(t_i), \beta(t_{i+1}), \alpha_1, \alpha_2$.}
\label{fig:bounded}
\end{figure}

We now show that an infinite-order element of $\pi_1(\tilde E_1)$ is the same as one in $\pi_1(\tilde E_2)$:
By convexity, either $\alpha_2$ goes into $I_1$ and not leave again or $\alpha_2$ is disjoint from $l_1$.
Suppose that $\alpha_2$ goes into $I_1$ and not leave it again. 
Since $\Bd I_2/\pi_1(\tilde E_2)$ is compact, there is a sequence $t_i$ so that 
the image of $\alpha_2(t_i)$ converges to a point of $\Bd I_1/\pi_1(\tilde E_1)$. 
Hence, by taking a short path between $\alpha_2(t_i)$s, 
there exists an essential closed curve $c_2$ in $I_2/\pi_1(\tilde E_2)$ homotopic to 
an element of $\pi_1(\tilde E_1)$. In fact $c_2$ is in a lens-cone end neighborhood of the end corresponding to $\tilde E_1$. 
This contradicts (NA). 
(The element is of infinite order since we can take a finite cover of $\orb$ so that $\pi_1(\orb)$ is torsion-free 
by Selberg's lemma.)

Suppose now that $\alpha_2$ is disjoint from $l_1$. Then $\alpha_1$ and $\alpha_2$ have
the same endpoint $z$ and by the convexity of $\alpha_2$. 
We parameterize $\alpha_i$ so that $\alpha_1(t)$ and $\alpha_2(t)$ are on a line segment 
containing $\ovl{\alpha_1(t)\alpha_2(t)}$ in the triangle 
with endpoints in $\ovl{zp_1}$ and $\ovl{zp_2}$. 

We obtain $d_{\orb}(\alpha_2(t), \alpha_1(t)) \leq C$ for a uniform constant $C$:
We define $\beta(t) := \ovl{\alpha_2(t) \alpha_1(t))}$.
Let $\gamma(t)$ denote the full extension of $\beta(t)$ in $\tri(p_1p_2z)$. 
One can project to the space of lines through $z$, a one-dimensional projective space. 
Then the image of $\beta(t)$ are so that 
the image of $\beta(t')$ is contained in that of $\beta(t)$ if $t < t'$. 
Also, the image of $\gamma(t)$ contains that of $\gamma(t')$ if $t < t'$. 
Thus, we can show by computation that the Hilbert-metric length of the segment 
$\beta(t)$ is bounded above by the uniform constant.

We have a sequence $t_i \ra \infty$ so that 
\[p_{\orb} \circ \alpha_2(t_i) \ra x, d_{\orb}(p_{\orb}\circ \alpha_2(t_{i+1}), p_{\orb}\circ \alpha_2(t_i)) \ra 0, x \in \orb. \]
So we obtain a closed curve $c_{2, i}$ in $\orb$
obtained by taking a short path jumping between the two points. By taking a subsequence, 
the image of $\beta(t_i)$ in $\orb$ geometrically converges to a segment of Hilbert-length $\leq C$.
As $i\ra \infty$, we have $d_{\orb}(p_{\orb} \circ \alpha_1(t_i), p_{\orb} \circ \alpha_1(t_{i+1})) \ra 0$ by extracting a subsequence.  
There exists a closed curve $c_{1, i}$ in $\orb$ again by taking a short jumping path. 
We see that $c_{1, i}$ and $c_{2, i}$ are homotopic in $\orb$ since we can use the image of the disk 
in the quadrilateral bounded by
$\ovl{\alpha_2(t_i) \alpha_2(t_{i+1})}, \ovl{\alpha_1(t_i) \alpha_1(t_{i+1})}, \beta(t_i), \beta(t_{i+1})$ 
and the connecting thin strips between the images of $\beta_{t_i}$ and $\beta_{t_{i+1}}$ in $\orb$. 
This again contradicts (NA). 



Now, consider when $U_1$ is a one-sided lens-neighborhood of a p-T-end and 
let $U_2$ be a concave p-R-end neighborhood of a p-R-end of $\torb$.
Let $z$ be the intersection point in $\clo(U_1) \cap \clo(U_2)$. 
We can use the same reasoning as above by choosing any $p_1$ in $S_{\tilde E_1}$ 
so that $\ovl{p_1z}$ passes the interior of $\tilde E_1$. Let $p_2$ be the p-R-end vertex of $U_2$. 
Now we obtain the triangle with vertices $p_1, p_2$, and $z$ as above. Then the arguments are analogous
and obtain infinite order elements in $\pi_1(\tilde E_1) \cap \pi_1(\tilde E_2)$. 

Next, consider when $U_1$ and $U_2$ are one-sided  lens-neighborhoods of p-T-ends respectively.  
Using the intersection point $z$ of $\clo(U_1) \cap \clo(U_2) \cap \torb$
and we choose $p_i$ in $\Bd \tilde E_i$ so that $\ovl{zp_i}$ passes the interior of $S_{\tilde E_i}$ for $i=1, 2$. 
Again, we obtain a triangle with vertex $p_1, p_2, $ and $z$, and find a contradiction as above. 

We finally consider when $U$ is a horospherical p-R-end. Since $\clo(U) \cap \Bd \torb$ is a unique point, 
(iii) of Theorem \ref{I-thm-affinehoro} of \cite{EDC1} implies the result.

\end{proof}





We modify Theorem \ref{thm-redtot} 
by replacing some conditions. 
In particular, we don't assume $h(\pi_{1}(\orb))$ is strongly irreducible. 
\begin{lemma}\label{lem-redtot2}
Let $\orb$ be a strongly tame 
properly convex real projective orbifold  
and satisfy {\em (IE)} and {\em (NA)}.
Let $\tilde E$ be a virtually factorable admissible p-R-end of $\torb$ of generalized lens-type.
Then
\begin{itemize}
\item there exists a totally geodesic hyperspace $P$ on which $h(\pi_1(\tilde E))$ acts, 
\item $D:=P \cap \torb$ is a properly convex domain, 
\item $D^{o} \subset \torb$, and 
\item $D^{o}/\pi_1(\tilde E)$ is a compact orbifold. 
\item Also, each element of $g \in \pi_1(\tilde E)$ acts as nonidentity on a subspace properly containing $v$. 
\end{itemize} 
\end{lemma}
\begin{proof}
The proof of Theorem \ref{thm-redtot}   shows that  
\begin{itemize}
\item either $\clo(\torb)$ is a strict join or
\item the conclusion of Theorem \ref{thm-redtot}   holds.
\end{itemize} 
In both cases, $\pi_1(\tilde E)$ acts on a totally geodesic convex compact domain $D$ of codimension $1$.
$D$ is the intersection $P_{\tilde E} \cap \clo(\torb)$ for a $\pi_1(\tilde E)$-invariant subspace $P_{\tilde E}$. 
Suppose that $D^{o}$ is not a subset of $\torb$. Then by Lemma \ref{lem-simplexbd}, 
$D\subset \Bd \torb$. 
In the former case, we can show that $\clo(\torb)$ is the join $v_{\tilde E} \ast D$. 

For each $g \in \pi_{1}(\tilde E)$ satisfying $g(v_{\tilde E}) \ne v_{\tilde E}$, we have $g(D) \ne D$
since $g(v_{\tilde E})\ast g(D) = v_{\tilde E}\ast D$. 
$g(D) \cap D$ is a proper compact convex subset
of $D$ and $g(D)$. 
Moreover, 
\[\clo(\torb) = v_{\tilde E}\ast g(v_{\tilde E}) \ast (D\cap g(D)).\]
We can continue as many times as there is a mutually distinct collection of vertices of form 
$g(v_{\tilde E})$. Since this process must stop, 
we have a contradiction since 
by Condition (IE), there are infinitely many distinct end vertices of form $g(v_{\tilde E})$ for $g \in \pi_{1}(\orb)$. 

Now, we go to the alternative case. 
Then $D^{o} \subset \torb$.  
The last part follows again from the proof of Theorem \ref{thm-redtot} (ii). 
The virtually reducible cases don't happen as above. 
\end{proof}

\begin{proof}[{\sl Proof of Theorem \ref{thm-sSPC}}.]
We need to prove for $\PGL(n+1, \bR)$ only for strong irreducibility. 
Let $h:\pi_1(\orb) \ra \PGL(n+1, \bR)$ be the holonomy homomorphism. 
Suppose that $h(\pi_1(\orb))$ is virtually reducible. Then we can choose a finite cover 
$\orb_1$ so that $h(\pi_1(\orb_1))$ is reducible. 

We denote $\orb_1$ by $\orb$ for simplicity. 
Let $S$ denote a proper subspace where $\pi_1(\orb)$ acts on. 
Suppose that $S$ meets $\torb$. 
Then $\pi_1(\tilde E)$ acts on a properly convex open domain $S\cap \torb$ for each p-end 
$\tilde E$. 
Thus, $(S\cap \torb)/\pi_1(\tilde E)$ is a compact orbifold homotopy equivalent to one of the end orbifold. 
However, $S \cap \torb$ is $\pi_1(\tilde E)$-invariant and cocompact 
for each p-end $\tilde E$.
Each p-end fundamental group $\pi_{1}(\tilde E)$ is virtually identical to
any other p-end fundamental group. 
This contradicts (IE). 
Therefore, 
\begin{equation}\label{eqn:K}
K:=S \cap \clo(\torb) \subset \Bd \torb. 
\end{equation} 

(A) We show that $K:= \clo(\torb) \cap S \ne \emp$: 
Let $\tilde E$ be a p-end. If $\tilde E$ is horospherical, $\pi(\tilde E)$ acts on a great sphere $\hat S$ tangent 
to an end vertex. 
Since $S$ is $\bGamma$-invariant, 
$S$ has to be a subspace in $\hat S$ containing the end vertex 
by Theorem \ref{I-thm-affinehoro}(iii) of \cite{EDC1}. 
This implies that every horospherical p-end vertex is in $S$. Since there is no nontrivial segment in $\Bd \torb$ containing
a horospherical p-end vertex of Theorem \ref{I-thm-affinehoro}(iv) of \cite{EDC1}, the p-end vertex 
is $\bGamma$-invariant. This contradicts the condition (IE). 

Suppose that $\tilde E$ is a p-R-end of generalized lens-type. Then by the existence of 
attracting subspaces of some elements of $\bGamma_{\tilde E}$, we have
\begin{itemize}
\item either $S$ passes the end vertex $\bv_{\tilde E}$ or 
\item there exists a subspace $S'$ containing $S$ and $\bv_{\tilde E}$ that is $\bGamma_{\tilde E}$-invariant. 
\end{itemize} 
Now consider the first case, we have $S \cap \clo(\torb) \ne \emp$. 

In the second case, 
$S'$ corresponds to a proper-invariant subspace in $\SI^{n-1}_{\bv_{\tilde E}}$
and $S$ is a hyperspace of dimension $n-1$ disjoint from $\bv_{\tilde E}$. 
Thus, $\tilde E$ is a virtually factorable p-R-end. 
By Lemma \ref{lem-redtot2} 
and Proposition 1.1 of \cite{Ben5} and the uniform middle eigenvalue condition, 
we obtain  some attracting fixed points 
in the limit sets of $\pi_1(\tilde E)$. 
Considering that $\pi_1(\tilde E)$ has nontrivial diagonalizable 
elements, we obtain $S \cap \clo(L) \ne \emp$

If $\tilde E$ is a p-T-end of lens-type, we can apply a similar argument using the attracting fixed points. 
Therefore, $S \cap \clo(\torb)$ is a subset $K$ of $\Bd \torb$ of $\dim K \geq 0$ and is not empty. 
In fact, we showed that the closure of each p-end neighborhood meets $K$. 




(B)  By taking a dual orbifold if necessary, 
we assume without loss of generality that there exists a p-R-end $\tilde E$ 
of generalized lens-type with a radial p-end vertex $\bv_{\tilde E}$. 

As above in (A), suppose that $\bv_{\tilde E} \in K$. 
There exists $g \in \pi_{1}(\orb)$, $g(\bv_{\tilde E}) \ne \bv_{\tilde E}$, 
and $g(\bv_{\tilde E}) \in K \subset \Bd \torb$. 
Since $g(\bv_{\tilde E})$ is outside the lens-cone or 
the generalized lens-cone of $\tilde E$, 
$K$ meets $\clo(L)$ for the lens or generalized lens $L$ of $\tilde E$. 

If $\bv_{\tilde E} \not\in K$, then again $K \cap \clo(L) \ne \emp$ as in (A) using attracting fixed points of 
some elements of $\pi_{1}(\tilde E)$. 
Hence, we conclude $K \cap \clo(L) \ne \emp$ for the lens $L$ of $\tilde E$. 

Let $\Sigma_{\tilde E}$ denote $D^{o}$ from Lemma \ref{lem-redtot2}. 
Since $K \subset \Bd \orb$,  $K$ cannot contain $\Sigma_{\tilde E}$. 
Thus, $K \cap \clo(\Sigma_{\tilde E})$ is a proper subspace of $\clo(\Sigma_{\tilde E})$, 
$\tilde E$ must be a virtually factorable end. 

By Lemma \ref{lem-redtot2}, 
there exists a totally geodesic domain $\Sigma_{\tilde E}$ in the lens-part. 
The p-end neighborhood of $\bv_{\tilde E}$ equals $U_{\bv_{\tilde E}}:=(\bv_{\tilde E} \ast \Sigma_{\tilde E})^{o}$. 
Since $\pi_{1}(\tilde E)$ acts reducibly, 
$\clo(\Sigma_{\tilde E})$ is a join $D_{1}\ast \cdots \ast D_{n}$. 
$K \cap \clo(U_{\bv_{\tilde E}})$ contains a join $D_{J}:= \ast_{i\in J}D_{i}$ for a proper subcollection 
$J$ of $\{1, \dots, n\}$. Moreover, $K \cap \clo(\Sigma_{\tilde E}) = D_{J}$. 

Since $g(U_{\bv_{\tilde E}})$ is a p-end neighborhood of $g(\bv_{\tilde E})$, we obtain $g(U_{\bv_{\tilde E}}) = U_{g(\bv_{\tilde E})}$.
Since $g(K) = K$ for $g \in \Gamma$, we obtain that 
\[K \cap g(\clo(\Sigma_{\tilde E}))  = g(D_{J}).\] 

Lemma \ref{lem-redtot2} implies that 
\begin{align} \label{eqn:Uv}
U_{g(\bv_{\tilde E})} \cap U_{\bv_{\tilde E}} = \emp \hbox{ for } g \not\in \pi_{1}(\tilde E) \hbox{ or } \nonumber \\ 
U_{g(\bv_{\tilde E})} = U_{\bv_{\tilde E}} \hbox{ for } g \in \pi_{1}(\tilde E) 
\end{align} 
by the similar properties of $S(g(\bv_{\tilde E}))$ and $S(\bv_{\tilde E})$ and the fact that
$\Bd U_{\bv_{\tilde E}} \cap \torb$ and $\Bd U_{g(\bv_{\tilde E})}\cap \torb$ are 
totally geodesic domains. 

Let $\lambda_{J}(g)$ denote the $(\dim D_{J}+1)$-th root of the norm of the determinant of the submatrix of $g$
associated with $D_{J}$ for the unit norm matrix of $g$.
Since the strict lens-type ends satisfy the uniform middle eigenvalue condition by Theorem \ref{thm-redtot}, 
a sequence of virtually cental elements $\gamma_{i}\in \pi_{1}(\tilde E)$ so that 
\begin{align} 
\gamma_{i}| D_{J} \ra \Idd, \gamma_{i}| D_{J^{c}} \ra \Idd \hbox{ for the complement } J^{c}:= \{1, 2, \dots, n\} - J, \nonumber \\
\frac{\lambda_{J}(\gamma_{i})}{\lambda_{\bv_{\tilde E}}(\gamma_{i})} \ra \infty, \frac{\lambda_{J^{c}}(\gamma_{i})}{\lambda_{\bv_{\tilde E}}(\gamma_{i})} \ra 0, 
\frac{\lambda_{J}(\gamma_{i})}{\lambda_{J^{c}}(\gamma_{i})} \ra \infty. 
\end{align}  

Since $\bv_{\tilde E}, D_{J}\subset K$, the eigenvalue condition implies that 
one of the following holds: 
\[K = D_{J}, K = \bv_{\tilde E} \ast D_{J} \hbox{ or } K = \bv_{\tilde E}\ast D_{J} \cup \bv_{\tilde E-}\ast D_{J}\]
by the invariance of $K$ under $\gamma_{i}^{-1}$
and the fact that $K \cap \clo(\Sigma_{\tilde E}) = D_{J}$. 
Since $K \subset \clo(\torb)$, the third case is not possible. 
We obtain \[K = D_{J} \hbox{ or } K= \{\bv_{\tilde E}\} \ast D_{J}.\] 



Consider the second case. 
Let $g$ be an arbitrary element of $\pi_{1}(\orb) - \pi_{1}(\tilde E)$. 
Since $D_{J} \subset K$, we obtain $g(D_{J}) \subset K$. 
Recall that $U_{\bv_{\tilde E}} \cup S(\bv_{\tilde E})^{o}$ is a neighborhood of points of $S(\bv_{\tilde E})^{o}$. 
Thus, $g(U_{\bv_{\tilde E}} \cup S(\bv_{\tilde E})^{o})$ is a neighborhood of points of $g(S(\bv_{\tilde E})^{o})$. 
$D_{J}^{o}$ is in the closure of $U_{\bv_{\tilde E}}$. 

If $D_{J}^{o}$ meets 
\[g(\bv_{\tilde E} \ast D_{J} - D_{J}) = g(U_{\bv_{\tilde E}}\cup S(\bv_{\tilde E})^{o}) \supset g(S(\bv_{\tilde E})^{o}),\] 
then $U_{\bv_{\tilde E}} \cap g(U_{\bv_{\tilde E}}) \ne \emp$, 
and $S(\bv_{\tilde E})^{o} \cap g(S(\bv_{\tilde E})^{o}) \ne \emp$ since 
these are components of $\torb$ with some totally geodesic hyperspaces removed.
Hence, $\bv_{\tilde E} = g(\bv_{\tilde E})$ by Theorems \ref{thm-lensclass} 
and \ref{thm-redtot}. Finally, we obtain $D_{J} = g(D_{J})$ as
$K = \bv_{\tilde E} \ast D_{J} = g(\bv_{\tilde E}) \ast g(D_{J})$. 

If $D_{J}^{o}$ is disjoint from $g(\bv_{\tilde E} \ast D_{J} - D_{J})$, then $g(D_{J}) \subset D_{J}$. 
Since $D_{J}$ and $g(D_{J})$ are intersections of a hyperplane with $\Bd \torb$,  
we obtain $g(D_{J}) = D_{J}$. 

In both cases, we conclude $g(D_{J}) = D_{J}$ for $g \in \pi_{1}(\orb)$.



This implies $g(D_{J}) = D_{J}$ for $g \in \pi_{1}(\orb)$. 
Since $\bv_{\tilde E}$ and $g(\bv_{\tilde E})$ are not equal for $g \in \pi_{1}(\orb) - \pi_{1}(\tilde E)$, 
we obtain a triangle $\tri$ with vertices $\bv_{\tilde E}, g(\bv_{\tilde E}), x \in D_{J}$. 
Then as in the part (ii) of the proof of Corollary \ref{cor-disjclosure}, 
we obtain the existence of essential annulus. (For this argument, 
we did not need the assumption on strong irreducibility of $h(\pi_{1}(\orb))$.)

Therefore, we deduced that the $h(\pi_{1}(\orb))$-invariant subspace $S$ does not exist.

\end{proof}

\appendix 
\renewcommand*{\thesection}{\Alph{section}}

\section{The affine action dual to the tubular action} \label{app-dual}

In this section we will show the asymptotic niceness of the the affine actions. 
The main tools will be Anosov flows on the unit tangent bundles as in Goldman-Labourie-Margulis \cite{GLM}. 
We will introduce a flat bundle and decompose it in an Anosov type way. 
We will prove the Anosov type property. Then we will find an invariant section.  
We will prove the asymptotic niceness using 
the sections. 

Let $\Gamma$ be an affine group acting on the affine space $A^n$ with boundary $\Bd A^n$ in $\SI^n$, 
i.e., an open hemisphere.
Let $U'$ be a properly convex invariant $\Gamma$-invariant domain 
with boundary in a properly convex domain 
$\Omega \subset \Bd A^n$.

In this section, we will work with $\SI^n$ only, while
the $\bR P^n$ versions are clear enough. 

Each element of $g \in \Gamma$ is of the form 
\begin{equation}\label{eqn-bendingm4} 
\left(
\begin{array}{cc}
\frac{1}{\lambda_{{\tilde E}}(g)^{1/n}} \hat h(g)          &       \vec{b}_g     \\
\vec{0}          &     \lambda_{{\tilde E}}(g)                  
\end{array}
\right)
\end{equation}
where $\vec{b}_g$ is $n\times 1$-vector and $\hat h(g)$ is an $n\times n$-matrix of determinant $\pm 1$
and $\lambda_{{\tilde E}}(g) > 0$.
In the affine coordinates, it is of the form 
\begin{equation}\label{eqn-affact} 
x \mapsto \frac{1}{\lambda_{{\tilde E}}(g)^{1+ \frac{1}{n}}} \hat h(g) x + \frac{1}{\lambda_{{\tilde E}}(g)} \vec{b}_g. 
\end{equation}
Recall that if there exists a uniform constant $C > 0$ so that 
\[C^{-1} \leng(g) \leq \log \frac{\lambda_1(g)}{\lambda_{{\tilde E}}(g)} \leq C \leng(g), \quad
g \in \bGamma_{\tilde E} -\{\Idd\},\] 
then $\Gamma$ is said to satisfy 
the {\em uniform middle-eigenvalue condition}.

In this appendix, it is sufficient for us to prove when $\Gamma$ is a hyperbolic group
when $\Omega$ must be strictly convex by Theorem 1.1 of \cite{Ben1}. 

\begin{theorem}\label{thm-asymnice}
We assume that $\Gamma$ is a hyperbolic group. 
Let $\Omega$ be a properly convex domain in $\Bd A^n$. 
Let $\Gamma$ have a properly convex affine action on the affine space $A^n$, $A^n \subset \SI^n$,
acting on a properly convex domain $U \subset A^n$
so that $\clo(U) \cap \Bd A^n = \clo(\Omega)$. 
Suppose that $\Omega/\Gamma$ is a closed $(n-1)$-dimensional orbifold and 
$\Gamma$ satisfies the uniform middle-eigenvalue condition. 
Then $\Gamma$ is asymptotically nice with the properly convex 
open domain $U$, 
and the asymptotic hyperspace at 
each boundary point of $\Omega$ is uniquely determined and is transversal to $\Bd A^n$. 
\end{theorem}

In the case when the linear part of the affine maps are unimodular, 
Theorem 8.2.1 of Labourie \cite{Lab} shows that such a domain $U$ exists but without showing the asymptotic niceness. 
In general, we think that the existence of the domain $U$ can be obtained but the proof 
is much longer. 
Here, we are in an easier case when a domain $U$ is given without the properties. 

(It is fairly easy to show that this holds also for virtual products of hyperbolic and abelian groups as well
by Proposition \ref{prop-dualend2} and Theorem \ref{thm-distanced}.) 

\subsection{The Anosov flow.}\label{sub-anosov}

We generalize the work of Goldman-Labourie-Margulis \cite{GLM}: Assume 
as in the premise of Theorem \ref{thm-asymnice}. 
Since $\Omega$ is properly convex, 
$\Omega$ has a Hilbert metric. 
Let $U\Omega$ denote the unit tangent bundle over $\Omega$.
This has a smooth structure as a quotient space of $T\Omega - O/\sim$ where 
\begin{itemize}
\item $O$ is the image of the zero-section, and 
\item $\vec{v} \sim \vec{w}$ if $\vec{v}$ and $\vec{w}$ are over the same point of $\Omega$
and $\vec{v} = s \vec{w}$ for a real number $s > 0$.
\end{itemize} 

Assume $\Gamma$ as above. 
Since $\Sigma:= \Omega/\Gamma$ is a properly convex real projective orbifold, 
$U\Sigma := U\Omega/\Gamma$ is a compact smooth orbifold again. 
A geodesic flow on $U\Omega/\Gamma$ is Anosov 
and hence topologically mixing. Hence, the flow is nonwondering everywhere. (See \cite{Ben1}.)
$\Gamma$ acts irreducibly on $\Omega$, and $\Bd \Omega$ is $C^1$. 

Let $h: \Gamma \ra \Aff(A^n)$ denote the representation 
as described in equation \eqref{eqn-affact}.  
We form the product $U\Omega \times A^n$ that is an affine bundle over $U\Omega$. 
We take the quotient $\tilde \bA := U\Omega \times A^n$ by the diagonal action 
\[g(x, \vec u)= (g(x), h(g) \vec u) \hbox{ for } g \in \Gamma, x \in U\Omega, \vec u \in A^n.\] 
We denote the quotient by $\bA$ fibering over
the smooth orbifold $U\Omega/\Gamma$ with fiber $A^n$. 

Let $V^n$ be the vector space associated with $A^n$.
Then we can form $\tilde \bV:= U\Omega \times V^n$ and take the quotient under 
the diagonal action:
\[g(x, \vec u)= (g(x), {\mathcal L}\circ  h(g) \vec u) \hbox{ for } g \in \Gamma, x\in U\Omega, \vec u \in V^n\]
where $\mathcal L$ is the homomorphism taking the linear part of $g$.  
We denote by $\bV$ the fiber bundle over $U\Omega/\Gamma$ 
with fiber $V^n$. 

We recall the trivial product structure. 
$U\Omega \times A^n$ is a flat $A^n$-bundle over $U\Omega$ with a flat affine connection $\nabla^{\tilde \bA}$, 
and $U\Omega \times V^n$ has a flat linear connection $\nabla^{\tilde \bV}$.
The above action preserves the connections. 
We have a flat affine connection $\nabla^{\bA}$ on the bundle $\bA$ over $U\Sigma$
and a flat linear connection $\nabla^{\bV}$ on the bundle $\bV$ over $U\Sigma$.

We give a decomposition of $\tilde \bV$ into three parts $\tilde \bV_+, \tilde \bV_0, \tilde \bV_-$: 
For each vector $\vec u \in U\Omega$, we find the maximal oriented geodesic 
$l$ ending at two points $\partial_+ l, \partial_- l \in \Bd \Omega$. They correspond to 
the $1$-dimensional vector subspaces $V_+(\vec u)$ and $V_-(\vec u) \subset V$. 
Recall that $\Bd \Omega$ is $C^{1}$ since $\Omega$ is strictly convex (see \cite{Ben1}).
There exists a unique pair of supporting hyperspheres $H_+$ and $H_-$ in $\Bd A^n$ at each of $\partial_+ l$ and $\partial_- l$. 
We denote by $H_0 = H_+ \cap H_-$. It is a codimension $2$ great sphere in $\Bd A^n$
and corresponds to a vector subspace $V_0$ of codimension-two in $\bV$. 
For each vector $\vec u$, we find the decomposition of $V$ as $V_+(\vec u) \oplus V_0(\vec u) \oplus V_-(\vec u)$
and hence we can form the subbundles $\tilde \bV_+, \tilde \bV_0, \tilde \bV_-$ over $U\Omega$
where \[\tilde \bV = \tilde \bV_+ \oplus \tilde \bV_0 \oplus \tilde \bV_-.\] 
The map $U\Omega \ra \Bd \Omega$ by sending a vector to the endpoint of the geodesic tangent to it is $C^{1}$.
The map $\Bd \Omega \ra \mathcal{H}$ sending a boundary point to its supporting hyperspace 
in the space $\mathcal{H}$ of hyperspaces in $\SI^{n}$ is continuous.   
Hence $\tilde \bV_{+}, \tilde \bV_{0},$ and $\tilde \bV_{-}$ are $C^{0}$-bundles. 
Since the action preserves the decomposition of $\tilde \bV$, 
$\bV$ also decomposes as \[ \bV = \bV_{+} \oplus \bV_0 \oplus \bV_-.\] 

We can identify $\Bd A^n = {\mathcal{S}}(V^{n})$ where $g$ acts by  ${\mathcal{L}}(g) \in \GL(n, \bR)$. 

For each complete  geodesic $l$ in $\Omega$, let $\vec l$ denote the set of unit vectors on $l$ in one-directions.  
On $\vec{l}$, we have a decomposition
\begin{align*}
&\tilde \bV|\vec l = \tilde \bV_{+}|\vec l \oplus \tilde \bV_{0}|\vec l \oplus \tilde \bV_{-}|\vec{l}  \hbox{ of form }  \\
&\vec l \times V_{+}(\vec u), \vec l \times V_{0}(\vec u), \vec l\times V_{-}(\vec u)
\hbox{ for a vector $\vec u$ tangent to $l$}.
\end{align*} 
where  we recall
\begin{itemize}
\item $V_{+}(\vec u)$ is the vectors in direction of the forward end point of $\vec{l}$ 
\item $V_{-}(\vec u)$ is the vectors in direction of the backward end point of $\vec{l}$ 
\item $V_{0}(\vec u)$ is the vectors in directions of $H_{0}= H_{+}\cap H_{-}$ for $\partial l$. 
\end{itemize}
That is, these bundles are constant bundles. 

If $g \in \Gamma$ acts on a complete geodesic $l$ with a unit vector $\vec u$, then $V_+(\vec u)$ and $V_-(\vec u)$ corresponding to endpoints of $l$ 
are eigenspaces of the largest norm  $\lambda_1(g)$ of the eigenvalues 
and the smallest norm $\lambda_n(g)$ of the eigenvalues of the linear part ${\mathcal{L}}(g)$ of $g$. 
Hence on $V_+(\vec u)$, $g$ acts by expending by $\lambda_1(g)$ 
and on $V_-(\vec u)$, $g$ acts by contracting by $\lambda_n(g)$. 

There exists a flow $\hat \Phi_t: U\Omega \ra U\Omega$ for $t \in \bR$ given 
by sending $\vec v$ to the unit tangent vector to at $\alpha(t)$ where 
$\alpha$ is a geodesic tangent to $\vec v$ with $\alpha(0)$ equal to the base point of $\vec v$.

We define a flow on $\tilde \Phi_t: \tilde \bA \ra \tilde \bA$ by considering a unit speed geodesic flow line $\vec{l}$ in $U\Omega$ 
and considering $\vec{l} \times E$ and acting trivially on the second factor as we go from $\vec v$ to $\hat \Phi_t(\vec v)$
(See remarks in the beginning of Section 3.3 and equations in Section 4.1 of \cite{GLM}.)
Each flow line in $U\Sigma$ lifts to a flow line on $\bA$ from every point in it. This induces a flow
 $\Phi_t: \bA \ra \bA$.  
 
We define a flow on $\tilde \Phi_t: \bV \ra \bV$ by considering a unit speed geodesic flow line $\vec{l}$ in $U\Omega$ and 
and considering $\vec{l} \times V$ and acting trivially on the second factor as we go from $\vec v$ to $\Phi_t(\vec v)$
for each $t$. (This generalizes the flow on \cite{GLM}.)
Also, $\tilde \Phi_t$ preserves $\tilde \bV_+$, $\tilde \bV_0$, and $\tilde \bV_-$ since on the line $l$, the endpoint $\partial_\pm l$ does not change. 
Again, this induces a flow 
\[\Phi_{t}: \bV \ra \bV, \bV_{+} \ra \bV_{+}, \bV_{0} \ra \bV_{0}, \bV_{-} \ra \bV_{-}.\]


We let $||\cdot||_S$ denote some metric on these bundles over $U\Sigma/\Gamma$ defined as a fiberwise inner product:
We chose a cover of $\Omega/\Gamma$ by compact sets $K_i$ 
and choosing a metric over $K_i \times A^n$ and use the partition of unity. 
This induces a fiberwise metric on $\bV$ as well. 
Pulling the metric back to $\tilde \bA$ and $\tilde \bV$, we obtain a fiberwise metrics
to be denoted by $||\cdot||_{S}$.

As in Section 4.4 of \cite{GLM}, 
$\bV = \bV_+ \oplus \bV_0 \oplus \bV_-$.
By the uniform middle-eigenvalue condition,  
$\bV$ has a fiberwise Euclidean metric $g$ with the following properties: 
\begin{itemize} 
\item the flat linear connection $\nabla^{\bV}$ is bounded with respect to $g$.
\item hyperbolicity: There exists constants $C, k > 0$ so that 
\begin{align} \label{eqn-Anosov1} 
 || \Phi_t ({\vec{v}})||_S \geq \frac{1}{C} \exp(kt) ||{\vec{v}}||_S \hbox{ as } t \ra \infty
\end{align}
for $\vec{v} \in \bV_+$ and 
\begin{align} \label{eqn-Anosov2} 
 || \Phi_t (\vec{v})||_S \leq C \exp(-kt) ||\vec{v}||_S \hbox{ as } t \ra \infty
\end{align}
for $v \in \bV_-$.
\end{itemize} 

Proposition \ref{prop-contr} proves this property by taking $C$ sufficiently large according to $t_1$, 
which is a standard technique. 


\subsection{The proof of the Anosov property.} \label{subsub-Anosov}


We can apply this to $\bV_-$ and $\bV_+$ by possibly reversing the direction of 
the flow. 
The Anosov property follows from the following proposition. 

Let $\bV_{-,1}$ denote the subset of $\bV_-$ of the unit length under $||\cdot||_S$.

\begin{proposition} \label{prop-contr} 
Let $\Omega/\bGamma$ be a closed real projective orbifold with hyperbolic group. 
Then there exists a constant $t_1$ so that 
\[ || \Phi_t(\bv)||_S \leq \tilde C || \bv||_S, \bv \in \bV_- \hbox{ and } || \Phi_{-t}(\bv)||_S \leq \tilde C ||\bv||_S, \bv \in \bV_+\] 
for $t \geq t_1$ and a uniform $\tilde C$, $0 < \tilde C < 1$.
\end{proposition}
\begin{proof} 
It is sufficient to prove the first part of the inequalities since we can substitute $t \ra -t$
and switching $\bV_{+}$ with $\bV_{-}$ as the direction of the vector changed to the opposite one. 

Let $\bV_{-,1}$ denote the subset of $\bV_-$ of the unit length under $||\cdot||_S$.
By following Lemma \ref{lem-S_0}, the uniform convergence implies that for given $0< \eps < 1$, 
for every vector $\bv$ in $\bV_{-, 1}$, there exists a uniform $T$ so that for $t > T$, 
$\Phi_t(\bv)$ is in an $\eps$-neighborhood $U_\eps(S_0)$ of the image $S_0$ of the zero section. 
Hence, we obtain that $\Phi_{t}$ is uniformly contracting near $S_{0}$, which implies the result. 
\end{proof} 

The line bundle $\bV_-$ lifts to $\tilde \bV_-$ where each unit vector $\bu$ on $\Omega$ 
one associates the line $\bV_{-,\bu}$ corresponding to the starting point in $\Bd \Omega$ of the oriented geodesic $l$ tangent to it.
$\tilde \bV_{-}| \vec l$ equals $\vec l \times \bV_{-, \bu}$. 
$\Phi_{t}$ lifts to a parallel translation or constant flow $\tilde \Phi_{t}$
of form \[ (\bu, \vec v) \ra (\hat \Phi_{t}(\bu), \vec v).\]

\begin{figure}
\centering
\includegraphics[height=8cm]{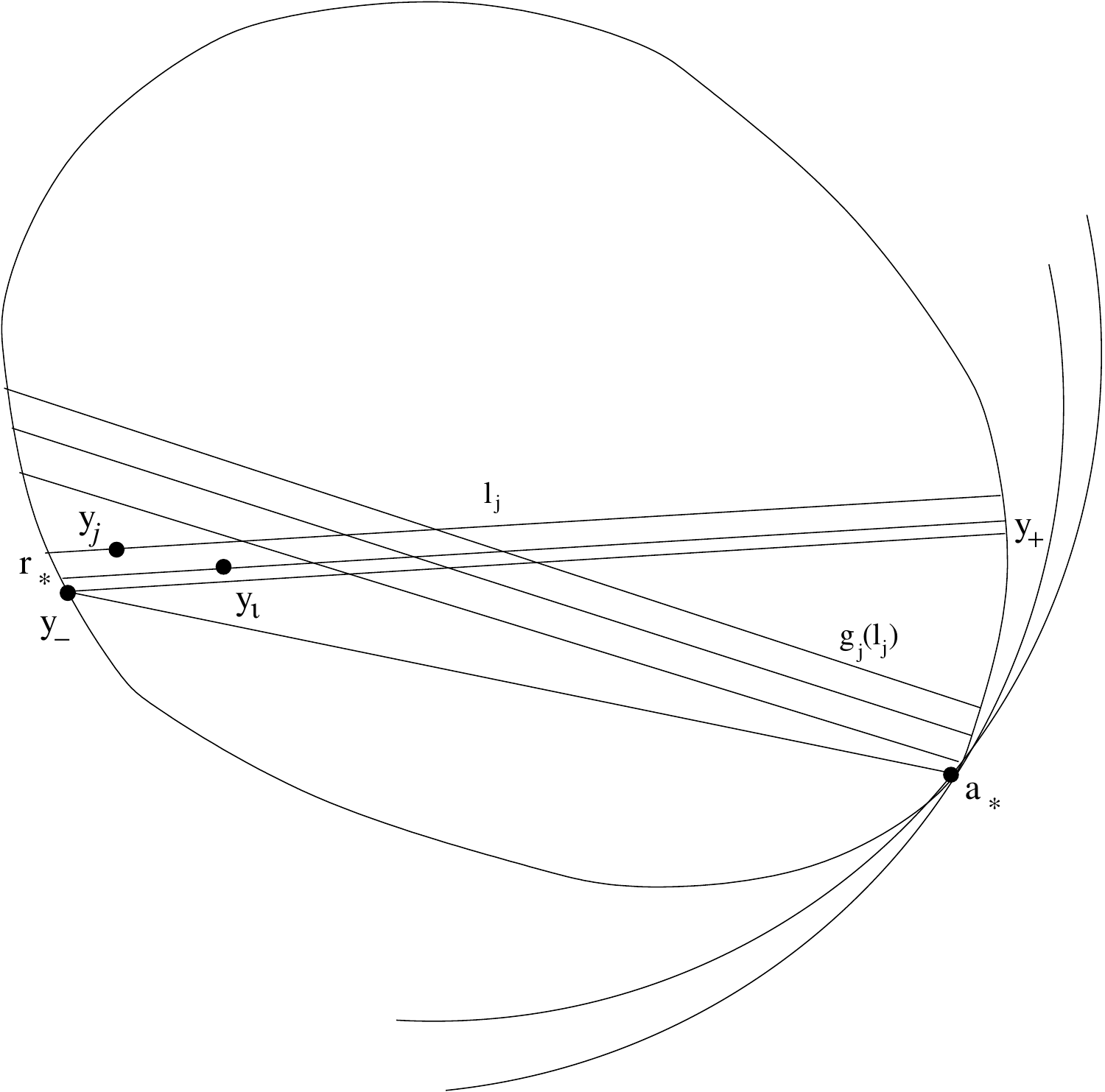}
\caption{The figure for Lemma \ref{lem-S_0}.}

\label{fig:contr}

\end{figure}



Let $P:U\Omega \ra \Omega$ be a projection of the unit tangent bundle to the base space. 

\begin{lemma} \label{lem-S_0} 
$||\Phi_{t}||_S  \ra 0 $ uniformly as $t \ra \infty$. 
\end{lemma}
\begin{proof} 
%
Let $F$ be a fundamental domain of $U\Omega$ under $\Gamma$.
It is sufficient to prove this for $\tilde \Phi_{t}$ on the fibers of over 
$F$ of $U\Omega$ with a fiberwise metric $||\cdot||_{S}$.

We choose an arbitrary sequence $\{x_i\}$, $\{x_i\} \ra x$ in $F$.
For each $i$, let $\bv_{-, i}$ be a Euclidean unit vector in  $V_{-,i}:= V_{-}(x_{i})$ for the unit vector $x_i \in U\Omega$.
That is, $\bv_{-, i}$ is in the $1$-dimensional subspace in $\bR^{n}$,  
corresponding to the endpoint of the geodesic determined by $x_i$ in $\Bd \Omega$. 

We will show that $||\tilde \Phi_{t_i}(x_{i}, \bv_{-,i})||_S \ra 0$ for any sequence $t_i \ra \infty$. 
This is sufficient to prove the uniform convergence to $0$ by the compactness of  $\bV_{-,1}$. 
(Here, $[\bv_{-, i}]$ is an endpoint of $l_i$ in the direction given by $x_i$.)

For this, we just need to show that any sequence of $\{t_i\} \ra \infty$ has a subsequence $\{t_j\}$ so that 
$||\tilde \Phi_{t_j}((x_{i}, \bv_{-,  j}))||_S \ra 0$.
This follows since if the uniform convergence did not hold, then we can easily find a sequence with out such subsequences. 

Let $y_i := \hat \Phi_{t_i}(x_i)$ for the lift of the flow $\hat \Phi$. 
By construction, we recall that each $P(y_i)$ is in the geodesic $l_i$. 
Since we have the sequence of vectors $x_i \ra x$, $x_{i}, x \in F$, 
we obtain that $l_i$ geometrically converges to a line $l_\infty$ passing $P(x)$ in $\Omega$. 
Let $y_+$ and $y_-$ be the endpoints of $l_\infty$ where $\{P(y_i)\} \ra y_-$.
Hence, \[[\bv_{+, i}] \ra y_{+}, [\bv_{-, i}] \ra y_{-}.\] 

Find a deck transformation $g_i$ so that 
$g_i(y_i) \in F$ and $g_i$ acts on the line bundle $\tilde \bV_-$ by the linearization of the matrix of form of equation \eqref{eqn-bendingm4}: 
\begin{align} 
g_{i} &: \bV_{-} \ra \bV_{-} \hbox{ given by } \nonumber \\
 (y_{i}, \bv) & \ra (g_{i}(y_{i}), {\mathcal L}(g_{i})(\bv)) \hbox{ where } \nonumber \\
{\mathcal L}(g_i) &:= \frac{1}{\lambda_{{\tilde E}}(g_i)^{1+\frac{1}{n}}} \hat h(g_i) : V_{-}(y_i) = V_{-}(x_{i}) \ra V_{-}(g_i(y_i)).
\end{align}
\begin{itemize}
\item[{\rm (Goal)}] We will show $\{(g_{i}(y_{i}), {\mathcal L}(g_i)(\bv_{-, i}))\} \ra 0$ under $||\cdot||_S$. 
This will complete the proof since $g_{i}$ acts as isometries on $\bV_{-}$ with $||\cdot||_{S}$. 
\end{itemize} 

Since $g_i(l_i) \cap F \ne \emp$, we choose a subsequence of $g_i$ and relabel it $g_i$ so that
$\{g_i(l_i)\}$ converges to a nontrivial line in $\Omega$.

We choose a subsequence of $\{g_i\}$ so that the sequences $\{a_i\}$ and $\{r_i\}$
are convergent for the attracting fixed point $a_i \in \clo(\Omega)$ 
and the repelling fixed point $r_i \in \clo(\Omega)$ of each $g_i$. 
Then 
\[\{a_i\} \ra a_\ast \hbox{ and } r_i \ra r_\ast \hbox{ for } a_\ast, r_\ast \in \Bd \Omega.\] 
(See Figure \ref{fig:contr}.) Also, it follows that  for every compact $K \subset \clo(\Omega) - \{r_\ast\}$,
\begin{equation}\label{eqn-giK} 
g_i|K \ra \{a_\ast\} 
\end{equation} 
uniformly as in the proof of Theorem 5.7 of \cite{conv}.

Suppose that $a_{\ast} = r_{\ast}$. Then we choose an element $g \in \Gamma$ so that $g(a_{\ast}) \ne r_{\ast}$
and replace the sequence by $\{g g_i\}$ and replace $F$ by $F \cup g(F)$. 
The above uniform convergence condition still holds. 
Then the new attracting fixed points $a'_i \ra g(a_\ast)$ 
and the sequence $\{r'_i\}$ of repelling fixed point $r'_i$ of $gg_i$ converges to $r_\ast$ also
by Lemma \ref{lem-gatt}.
Hence, we may assume without loss of generality that \[a_\ast \ne r_\ast\]
by replacing our sequence $g_i$.  

Suppose that both $y_+, y_- \ne r_\ast$. Then $\{g_i(l_i)\}$ converges to a singleton $\{a_\ast\}$ by 
equation \eqref{eqn-giK} and this cannot be. 
If \[r_\ast = y_+ \hbox{ and } y_- \in \Bd \Omega - \{r_\ast\},\] then 
$g_i(y_i) \ra a_\ast$ by equation \eqref{eqn-giK} again. 
Since $g_i(y_i) \in F$, this is a contradiction. 
Therefore \[r_\ast = y_- \hbox{ and } y_+ \in \Bd \Omega-\{r_\ast\}.\] 

Let $d_i$ denote the other endpoint of $l_i$ from $[\bv_{-, i}]$. 
\begin{itemize} 
\item Since $[\bv_{-, i}] \ra y_-$ and $l_i$ converges to a nontrivial line $l_\infty$, it follows that 
$\{d_i\}$ is in a compact set in $\Bd \Omega -\{y_-\}$.
\item Then $\{g_i(d_i)\} \ra a_{\ast}$ as $\{d_i\}$ is in a compact set in $\Bd \Omega -\{y_-\}$.
\item Thus, $\{g_i([\bv_{-, i}])\} \ra y' \in \Bd \Omega$ where $a_\ast \ne y'$ holds since 
$\{g_i(l_i)\}$ converges to a nontrivial line in $\Omega$. 
\end{itemize} 


Also, $g_i$ has an invariant great sphere $\SI^{n-2}_i \subset \Bd A^n$ 
containing the attracting fixed point $a_i$ and supporting $\Omega$ at $a_{i}$.
Thus, $r_i$ is uniformly bounded at a distance from $\SI^{n-2}_i$ since $\{r_i\} \ra y_-= r_\ast$
and $a_{i} \ra a_{\ast}$ with $\SI^{n-2}_{i}$ geometrically converging to a supporting sphere $\SI^{n-2}_{\ast}$
at $a_{\ast}$. 

Let $||\cdot||_E$ denote the standard Euclidean metric of $\bR^{n}$. 
\begin{itemize}
\item Since $P(y_i) \ra y_-$, it follows that $P(y_i)$ 
is also uniformly bounded away from $a_i$ and the tangent sphere $\SI^{n-1}_i$ at $a_i$. 
\item Since $[\bv_{-, i}] \ra y_-$,
the vector $\bv_{-, i}$ has the component $\bv_i^p$ parallel to $r_i$ and the component $\bv_i^S$ 
in the direction of $\SI^{n-2}_i$ where $\bv_{-, i} = \bv_i^p + \bv_i^S$. 
\item Since $r_i \ra r_\ast= y_-$ and $[\bv_{-, i}] \ra y_-$, we obtain $\bv_i^S \ra 0$ and that $\bv_i^p$ is uniformly bounded in $||\cdot||_E$. 
\item $g_i$ acts by preserving the directions of $\SI^{n-2}_i$ and $r_i$. 
\end{itemize} 
Since $\{g_i([\bv_{-, i}])\}$ converging to $y'$ is bounded away from $\SI^{n-2}_i$ uniformly, we have that 
\begin{itemize} 
\item the Euclidean norm of \[\frac{{\mathcal L}(g_i)(\bv_i^S)}{||{\mathcal L}(g_i)(\bv_i^p)||_E}\] is bounded above uniformly. 
\end{itemize} 
Since $r_i$ is a repelling fixed point of $g_i$ and $||\bv_i^p||_E$ is uniformly bounded above, 
we have $\{{\mathcal L}(g_i)(\bv_i^p)\} \ra 0$. 
\[\{{\mathcal L}(g_i)(\bv_i^p)\} \ra 0 \hbox{ implies } \{{\mathcal L}(g_i)(\bv_i^S)\} \ra 0\]
for $||\cdot||_{E}$.  
Hence, we obtain $\{{\mathcal L}(g_i)(\bv_{-, i})\} \ra 0$ under $||\cdot||_E$.

Recall that $\tilde \Phi_{t}$ is the identity map on the second factor of $U\Omega \times V_{-}$. 
\[g_{i}(\tilde \Phi_{t_{i}}(x_{i}, \bv_{-,i})) = (g_{i}(y_{i}), {\mathcal L}(g_{i})(\bv_{-, i}))\]
 is a vector over the compact fundamental domain $F$ of $U\Omega$. 
Since \[(g_{i}(y_{i}), \mathcal{L}(g_{i})(\bv_{-,i}))\] is a vector over the compact fundamental domain $F$ of $U\Omega$
with \[|| \mathcal{L}(g_{i})(\bv_{-,i}))||_{E} \ra 0,\] 
we conclude that $\{||\tilde \Phi_{t_i}(x, \bv_{-, i})||_S \}\ra 0$:
For the compact fundamental domain $F$, the Euclidean metric $||\cdot||_{S}$ and 
the Riemannian metric $||\cdot||_{S}$ 
of $\tilde \bV_-$ are related by a bounded constant on the compact set $F$. 
\end{proof}






\subsection{The neutralized section.}

A section $s:U\Sigma \ra \bA$ is {\em neutralized} if 
\begin{equation}\label{eqn-neu}
\nabla^{\bA}_{\phi} s \in \bV_0 . 
\end{equation}
We denote by $\Gamma(\bV)$ the space of sections 
$U\Sigma \ra \bV$ and by $\Gamma(\bA)$ the space of
sections $U\Sigma \ra \bA$. 

Recall from \cite{GLM} the one parameter-group of bounded 
operators $D\Phi_{t, *}$ on $\Gamma(\bV)$ and 
$\Phi_{t, *}$ on $\Gamma(\bA)$. We denote by $\phi$ the vector field generated by 
this flow on $U\Sigma$.
Recall Lemma 8.3 of \cite{GLM} also. 

\begin{lemma}\label{lem-lem83}
 If $\psi \in \Gamma(\bA)$, and 
\[t \mapsto D\Phi_{t, *}(\psi) \] 
is a path in $\Gamma(\bV)$ that is differentiable at $t =0$, then 
\[ \frac{d}{dt}\left|_{t=0}  (D\Phi_t)_* (\psi)  \right. = \nabla^{\bA}_\phi(\psi). \]
\end{lemma} 

Recall that $U\Sigma$ is a recurrent set under the geodesic flow. 
\begin{lemma}\label{lem-exist} 
A neutralized section exists on $U\Sigma$. 
This lifts to a map $\tilde s_0: U\Omega \ra \bA$  
so that $\tilde s_0 \circ \gamma = \gamma \circ \tilde s_0$. 
\end{lemma}
\begin{proof} 
Let $s$ be a continuous section $U\Sigma \ra \bA$. 
We decompose 
\[\nabla^{\bA}_\phi(s) = \nabla^{\bA+}_\phi(s) + \nabla^{\bA 0}_\phi(s) + \nabla^{\bA-}_\phi(s) \in \bV \]
so that $\nabla^{\bA\pm}_\phi(s) \in \bV_\pm$ and $\nabla^{\bA_0}_\phi(s)  \in \bV_0$ hold.
By the uniform convergence property of equations \eqref{eqn-Anosov1} and \eqref{eqn-Anosov2}, 
the following integrals converge to smooth functions over $U\Sigma$. 
Again
\[ s_0 = s + \int_0^\infty (D\Phi_t)_*(\nabla^{\bA-}_\phi(s)) dt - \int_0^\infty (D\Phi_{-t})_*(\nabla^{\bA+}_\phi(s)) dt\] 
is a continuous section and 
$\nabla^{\bA}_\phi(s_0) = \nabla^{\bA_0}_\phi(s_0) \in \bV_0$ as shown in \cite{GLM}. 

Since $U\Sigma$ is connected, there exists a fundamental domain $F$ 
so that we can lift $s_0$ to $\tilde s_0'$ defined on $F$ mapping to $\bA$. 
We can extend $\tilde s_0'$ to $U\Omega \ra \Omega \times E$. 
\end{proof} 

Let $N_2(A^n)$ denote the space of codimension two affine spaces of $A^n$. 
We denote by $G(\Omega)$ the space of maximal oriented geodesics in $\Omega$.
We use the quotient topology on both spaces. 
There exists a natural action of $\Gamma$ on both spaces. 

For each element $g\in \Gamma -\{\Idd\}$, we define $N_2(g)$: 
Now, $g$ acts on $\Bd A^n$ with invariant subspaces corresponding to 
invariant subspaces of the linear part ${\mathcal L}(g)$ of $g$. 
Since $g$ and $g^{-1}$ are positive proximal,  
\begin{itemize}
\item a unique fixed point in $\Bd A^{n}$ corresponds
to the largest norm eigenvector, an attracting fixed point in $\Bd A^n$, 
and 
\item a unique fixed point in $\Bd A^{n}$ corresponds to the smallest norm eigenvector,
a repelling fixed point 
\end{itemize} 
by \cite{Ben1} or \cite{Benasym}.
There exists an ${\mathcal L}(g)$-invariant 
vector subspace $V_g^0$ complementary to the join of the subspace generated 
by these eigenvectors. (This space equals $V_0(\vec u)$ for the unit tangent vector $\vec u$
tangent to the unique maximal geodesic $l_g$ in $\Omega$ on which $g$ acts.)
It corresponds to a $g$-invariant 
subspace $M(g)$ of codimension two in $\Bd A^n$.

Let $\tilde c$ be the geodesic in $U\Sigma$ that is $g$-invariant for $g \in \Gamma$. 
$\tilde s_0(\tilde c)$ lies on a fixed affine space parallel to $V_g^{0}$ by the neutrality, i.e., Lemma \ref{lem-exist}. 
There exists a unique affine subspace $N_2(g)$ of codimension two in $A^n$ 
whose containing $\tilde s_0(\tilde c)$. 
Immediate properties are $N_2(g) = N_2(g^{m}), m \in \bZ$
and that $g$ acts  on $N_{2}(g)$. 

\begin{definition}\label{prop-tau} 
We define $S'(\Bd  \Omega)$ the space of $(n-1)$-dimensional hemispheres
with interiors in $A^n$ each of  
whose boundary in $\Bd A^n$ is a supporting hypersphere in $\Bd A^n$ to $\Omega$. 
We denote by $S(\Bd \Omega)$ the space of pairs 
$(x, H)$ where $H \in S'(\Bd \Omega)$ and $x$ is in the boundary of 
$H$ and $\Bd \Omega$. 
\end{definition}

Define $\Delta$ to be the diagonal set of $\Bd \Omega \times \Bd \Omega$. 
Denote by $\Lambda^* = \Bd \Omega \times \Bd \Omega - \Delta$. 
Let $G(\Omega)$ denote the space of maximal oriented geodesics in $\Omega$. 
$G(\Omega)$ is in a one-to-one correspondence with $\Lambda^*$ by 
the map taking the maximal oriented geodesic to the ordered pair of its endpoints. 

\begin{proposition}\label{prop-mapgh}
\begin{itemize}
\item There exists  a continuous function $\hat s: U\Omega \ra N_2(A^n)$ 
equivariant with respect to $\Gamma$-actions.
\item Given $g \in \Gamma$ and for the unique unit speed geodesic $\vec{l}_g$ in 
$U\Omega$ lying over a geodesic $l_g$ where $g$ acts on, $\hat s(\vec{l}_g) = \{N_2(g)\}$. 
\item This gives a continuous map 
\[\bar s': \Bd \Omega \times \Bd \Omega - \Delta \ra N_2(A^n)\] 
again equivariant with respect to the $\Gamma$-actions. 
There exists a continuous function 
\[\tau:\Lambda^* \ra S(\Bd \Omega).\]
\end{itemize}
\end{proposition}
\begin{proof} 
Given a vector ${\vec u} \in U\Omega$, we find $\tilde s_0(\vec{u})$.
There exists a lift $\tilde \phi_t: U\Omega \ra U\Omega$ of 
the geodesic flow $\phi_t$.
Then $\tilde s_0(\tilde \phi_t({\vec u}))$ is in 
an affine subspace $H^{n-2}$ parallel to $V_0$ for $\vec u$ by the neutrality 
condition equation \eqref{eqn-neu}.
We define $\hat s(\vec u)$ to be this $H^{n-2}$. 

For any unit vector $\vec u'$ on the maximal (oriented) geodesic 
in $\Omega$ determined by $\vec u$, we obtain
$\hat s(\vec u') = H^{n-2}$. 
Hence, this determines the continuous map $\bar s: G(\Omega) \ra N_2(A^n)$. 
The $\Gamma$-equivariance comes from that of $\tilde s_0$. 

For $g \in \Gamma$, $\vec u$ and $g(\vec u)$
lie on the $g$-invariant geodesic $l_g$ provided $\vec u$ is tangent to $l_g$. 
Since $g(\tilde s_0(\vec u)) = \tilde s_0(g(\vec u))$ by equivariance, 
$g(\tilde s_0(\vec u))$ lies on $\hat s(\vec u) = \hat  s(g(\vec u))$ by two paragraphs above. 
We conclude $g(\bar s(l_g)) = \bar s(l_g)$.

The map $\bar s'$ is defined since
$\Bd \Omega \times \Bd \Omega - \Delta$ is in one-to-one correspondence with 
the space $G(\Omega)$. 
The map $\tau$ is defined by taking for each pair $(x, y) \in \Lambda^*$
\begin{itemize} 
\item we take the geodesic $l$ with endpoints $x$ and $y$,
and 
\item taking the hyperspace in $A^n$ containing $\bar s(l)$ and its boundary containing $x$.  
\end{itemize} 
\end{proof}


\subsection{The asymptotic niceness.} \label{sub-asymnice}


We denote by $h(x, y)$ the $(n-1)$-dimensional hemisphere part in $\tau(x, y) = (x, h(x, y))$. 

\begin{lemma}\label{lem-hdisj} 
Let $U$ be a $\bGamma_{\tilde E}$-invariant 
properly convex open domain in $\bR^n$ so that 
$\Bd U \cap \Bd A^n = \clo(\Omega)$. 
Suppose that $x$ and $y$ are fixed points of an element $g$ of $\Gamma$ in $\Bd \Omega$. 
Then $h(x, y)$ is disjoint from $U$. 
\end{lemma} 
\begin{proof}
Suppose not. 
$h(x, y)$ is a $g$-invariant hemisphere, and $x$ is an attracting fixed point of $g$ in it.
(We can choose $g^{-1}$ if necessary.) 
Then $U \cap h(x, y)$ is a $g$-invariant properly convex open domain containing $x$ in its boundary. 

Suppose first that $h(x, y)$ has a fixed point $z$ of $g$ 
with the smallest eigenvalue in $h(x, y)^o$. 
Then the associated eigenvalue to $z$ 
is strictly less than that of $x$ 
by the uniform middle-eigenvalue condition
and hence $z$ is in the closure of the convex open domain $U\cap h(x, y)$. 
$g$ acts on the $2$-sphere $P$ containing $x, y, z$. 
Then the $g$ acts on $P\cap U$ intersecting $\ovl{xz}^{o}$. This set $P \cap U$ 
cannot be properly convex due to the fact that $z$ is a 
saddle-type fixed point. Hence, there exists no fixed point $z$. 

The alternative is as follows: 
$h(x, y)$ contains a $g$-invariant affine subspace $A^{\prime}$ of codimension at least $2$ in $A^{n}$, 
and the fixed point of the smallest eigenvalue in $h(x, y)$  is associated with a point of $\Bd A'$. 
$g| h(x, y)$ has the largest norm eigenvalue at $x, x_{-}$. 
Therefore, we act by  $\langle g \rangle$ on a generic point $z$ of $h(x, y) \cap U$.
We obtain an arc in $h(x, y)$ with endpoints $x$ or $x_{-}$ and 
an endpoint $y'$ in $\Bd A^{\prime} \subset \Bd A^n$. 
Here $y'$ is a fixed point in $h(x, y)$ different from $y$ as $y \not\in h(x, y)$, 
and $y' \in \clo(U)$. It follows $y' \in \clo(\Omega)$. 
$x \in \clo(\Omega)$ implies $x_{-}\not\in \clo(\Omega)$ by the proper convexity. 
$x, y' \in \clo(\Omega)$ implies
$\ovl{x y'} \subset \Bd A^n \subset \clo(\Omega)$.
Finally, $\ovl{x y'} \subset \partial h(x, y)$ for 
the supporting subspace $\partial h(x, y)$ of $\clo(\Omega)$
violates the strict convexity of $\Omega$.
(See Benoist \cite{Ben1}.)


\end{proof} 

The proof of the following lemma is slightly different from that of Theorem 9.1 in \cite{afftame}
since we can use an invariant properly convex domain $U$. 
In Theorem \ref{thm-qFuch2}, we will obtain that this also give us strict lens p-end neighborhoods.

\begin{lemma} \label{lem-inde}
Let $(x, y) \in \Lambda^*$. Then 
\begin{itemize}
\item $\tau(x, y)$ does not depend on $y$ and is unique for each $x$. 
\item $h(x, y)$ contains $\bar s(\overline{xy})$ but is independent of $y$. 
\item $h(x, y)$ is never a hemisphere in $\Bd A^n$ for every 
$(x, y) \in \Lambda^*$.
\item $\tau: \Bd \Omega \ra S(\Bd \Omega)$ is continuous. 
\end{itemize}
\end{lemma}
\begin{proof} 
We claim that for any $x, y $ in $\Bd \Omega$, $h(x, y)$ is disjoint from $U$: 
By Theorem 1.1 of Benoist \cite{Ben1}, the geodesic flow on $\Omega/\Gamma$ is Anosov, and hence 
closed geodesics in $\Omega/\Gamma$ is dense in the space of geodesics
by the basic property of the Anosov flow. 
Since the fixed points are in $\Bd \Omega$, we can find a sequence  $x_i \ra x$ and $y_i \ra y$ where 
$x_i$ and $y_i$ are fixed points of an element $g_i \in \Gamma$ for each $i$. 
If $h(x, y) \cap U \ne \emp$, then $h(x_i, y_i) \cap U \ne \emp$ for 
$i$ sufficiently large by the continuity of the map $\tau$.
This is a contradiction by Lemma \ref{lem-hdisj} 

Also $\Bd A^n$ does not contain $h(x, y)$ 
since $h(x, y)$ contains the $\bar s(\ovl{xy})$ while $y$ is chosen 
$y \ne x$. 

Let $H(x, y)$ denote the half-space bounded by $h(x, y)$ containing $U$. 
$\partial H(x, y')$ is supporting $\Bd \Omega$ and hence is independent of $y'$
as $\Bd \Omega$ is $C^1$. So, we have 
\[H(x, y) \subset H(x, y') \hbox{ or } H(x, y) \supset H(x, y').\] 
For each $x$, we define 
\[ H(x) := \bigcap_{y \in \Bd \Omega -\{x\}} H(x, y). \]
Define $h(x)$ as the boundary $(n-1)$-hemisphere of $H(x)$.  

Now, $U':= \bigcap_{x \in \Bd \Omega} H(x)$ contains $U$ by the above disjointedness.
Since $\Bd \Omega$ has at least $n+1$ points in general position 
and tangent hemispheres, $U'$ is properly convex.
Let $U''$ be the properly convex open domain 
\[\bigcap_{x \in \Bd \Omega} (E - \clo(H(x))).\]
It has the boundary $\mathcal{A}(\clo(\Omega))$ in $\Bd A^n$ for the antipodal map $\mathcal A$.
Since the antipodal set of $\Bd \Omega$ has at least $n+1$ points in general position, 
$U''$ is a properly convex domain.
Note that $U' \cap U'' = \emp$.

If for some $x, y$, $h(x, y)$ is different from $h(x)$, then 
$h(x, y) \cap U'' \ne \emp$.
This is a contradiction by the above part of the proof where $U$ is replaced by $U''$. 
Thus, we obtain $h(x, y) = h(x)$ for all $y \in \Bd \Omega -\{x \}$. 

We show the continuity of $x \mapsto h(x)$:
Let $x_i \in \Bd \Omega$ be a sequence converging to $x \in \Bd \Omega$. 
Then choose $y_i \in \Bd \Omega$ so that 
$y_i \ra y$ and we have $\{h(x_i) = h(x_i, y_i)\}$ converges to 
$h(x, y) = h(x)$ by the continuity of $\tau$. 
Therefore, $h$ is continuous. 
\end{proof} 

\begin{proof}[{\sl Proof of Theorem \ref{thm-asymnice}}.] 
For each point $x \in \Bd \Omega$, an $(n-1)$-dimensional hemisphere $h(x)$ passes $A^n$ 
with $\partial h(x) \subset \Bd A^n$ supporting $\Omega$ by Lemma \ref{lem-inde}. 
Then a hemisphere $H(x) \subset A^n$ is bounded by $h(x)$ and contains $\Omega$. 
The properly convex open domain 
$\bigcap_{x \in \Bd \Omega} H(x)$ contains $U$.
Since $\Bd \Omega$ is $C^{1}$ and strictly convex, 
the uniqueness of $h(x)$ in the proof of Lemma \ref{lem-inde} gives us the unique asymptotic totally geodesic hypersurface. 
\end{proof} 

The following is another version of Theorem \ref{thm-asymnice}. 
We do not assume that $\Gamma$ is hyperbolic here. 

\begin{theorem}\label{thm-lensn} 
Let $\Gamma$ be a discrete group in $\SLnp$ acting on $\Omega$, 
$\Omega \subset \Bd A^n$, so that $\Omega/\Gamma$ is 
a compact orbifold.  
\begin{itemize}
\item Suppose that $\Omega$ has a $\Gamma$-invariant open domain $U$ 
forming a neighborhood of $\Omega$ in $A^n$. 
\item Suppose that $\Gamma$ satisfies the uniform middle eigenvalue condition. 
\item Let $P$ be the hyperplane containing $\Omega$. 
\end{itemize}
Then $\Gamma$ acts on a properly convex domain $L$ in $\SI^n$ 
with strictly convex boundary $\partial L$ such that 
\[ \Omega \subset L \subset U, \partial L \subset \SI^{n}-P. \]
Moreover, $L$ is a lens-shaped neighborhood of $\Omega$ with $\Bd \partial L \subset P$. 
\end{theorem} 
\begin{proof} 
Suppose that $\Gamma$ is not virtually factorable and hyperbolic. 
Define a half-space $H(x) \subset A^n$ bounded by $h(x)$ and containing $\Omega$ in the boundary. 
For each $H(x)$, $x \in \Bd \Omega$, in the proof of Theorem \ref{thm-asymnice}, 
an open $n$-hemisphere $H'(x) \subset \SI^{n}$ satisfies $H'(x) \cap A^{n} = H(x)$. 
Then we define
\[V:=\bigcap_{x \in \Bd \Omega} H'(x) \subset \SI^{n}\] 
is a convex open domain containing $\Omega$
as in the proof of Lemma \ref{lem-inde}.

Suppose that $\Gamma$ is virtually factorable. 
By Theorem \ref{thm-distanced} and Proposition \ref{prop-dualDA},
$\Gamma$ acts on a compact set 
\[\mathcal{H}:=
\{ h| h \hbox{ is a supporting hyperspace at } x \in \Bd \Omega, h \not\subset {\SI^{n-1}_\infty}\}\] 
Let $\mathcal{H}'$ denote the set of hemispheres bounded by an element of $\mathcal{H}$ and 
containing $\Omega$. 
 Then we define
\[V:=\bigcap_{H \in \mathcal{H}'} H \subset \SI^{n}\] 
is a convex open domain containing $\Omega$.
Here again the set of supporting hyperspaces is closed and bounded away from $\SI^{n-1}_{\infty}$. 

First suppose that $V$ is properly convex. 
Then $V$ has a $\Gamma$-invariant 
Hilbert metric $d_V$ that is also Finsler. (See \cite{wmgnote} and \cite{Kobpaper}.)
Then
\[N_\eps =\{ x\in V| d_V(x, \Omega)) < \eps\}\]
 is a convex subset of $V$ by Lemma \ref{lem-nhbd}. 



A compact tubular neighborhood $M$ of $\Omega/\Gamma$ in $V/\Gamma$ is
diffeomorphic to $\Omega/\Gamma \times [-1,1]$. (See Section 4.4.2 of \cite{Cbook}.)
We choose $M$ in $U/\Gamma$. 
Since $\Omega$ is compact, the regular neighborhood has a compact closure. 
Thus, $d_V(\Omega/\Gamma, \Bd M/\Gamma) > \eps_0$ for some $\eps_0 > 0$. 
If $\eps < \eps_0$, then $N_\eps \subset M$. We obtain that $\Bd N_\eps/\Gamma$ is compact.


Clearly, $\Bd N/\Gamma$ has two components in two respective components of $(V - \Omega)/\Gamma$.
Let $F_1$ and $F_2$ be the fundamental domains of both components. 
We procure the set ${\mathcal  H}_{j}$ of
finitely many open hemispheres $H_i$, $H_i \supset \Omega$, 
so that open sets $(\SI^n - \clo(H_i) )\cap N_\eps$ cover $F_j$ for $j=1, 2$. 
By Lemma \ref{lem-locfin}, 
the following is an open set containing $\Omega$
\[W := \bigcap_{g \in \Gamma} \bigcap_{H_{i} \in \mathcal{H}_{1}\cup \mathcal{H}_{2}} g(H_i) \cap V.\]
Since any path in $V$ from $\Omega$ to $\Bd N_\eps$ must meet $\Bd W -P$ first, $N_\eps$ contains
$W$ and $\Bd W$. 
A collection of 
compact totally geodesic polyhedrons meet in angles $< \pi$ and 
comprise $\Bd W/\Gamma$. 
Let $L$ be $\clo(W) \cap \torb$. Then $\partial L$ has boundary only in $\Bd A^{n}$
by Lemma \ref{lem-attracting2} since $\Gamma$ satisfies the uniform middle eigenvalue condition. 
We can smooth $\Bd W$ to obtain a lens-neighborhood $W' \subset W$ of $\Omega$ in $N_\eps$.  


Suppose that $V$ is not properly convex. Then $\Bd V$ contains $v, v_-$.
$V$ is a tube. 
We take any two open hemispheres $S_1$ and $S_2$ containing $\clo(\Omega)$ so that 
$\{v, v_-\} \cap S_1 \cap S_2 = \emp$. 
Then $\bigcap_{g \in \Gamma} g(S_1 \cap S_2) \cap V$ is a properly convex open domain containing $\Omega$.
and we can apply the same argument as above.

\end{proof} 

\begin{lemma}\label{lem-attracting2} 
Let $\Gamma$ be a discrete group in $\SLnp$ acting on $\Omega$, 
$\Omega \subset \Bd A^n$, so that $\Omega/\Gamma$ is 
a compact orbifold.  Suppose that $\Gamma$ satisfies the uniform middle eigenvalue condition. 
\begin{itemize}
\item Suppose that the supporting hyperspheres 
are at uniformly bounded distances from the hypersphere containing $\Omega$
\item Suppose that $\gamma_i$ is a sequence of elements of $\Gamma$ acting on $\Omega$. 
\item The sequence of attracting fixed points $a_i$ and the sequence of  repelling fixed points $b_i$ are so that 
$a_i \ra a_\infty$ and $b_i \ra b_\infty$ where $a_\infty, b_\infty$ are in $\clo(\Omega) - \Omega$. 
\item Suppose that the sequence $\{\lambda_i\}$ of eigenvalues where 
$\lambda_i$ corresponds to $a_i$ converges to $+\infty$. 
\end{itemize} 
Then for a properly convex open domain $V$ containing $\Gamma$ of the affine action
the point $\{a_\infty\}$ is the limit of $\{\gamma_i(J)\}$ for any compact subset $J \subset V$. 
\end{lemma} 
\begin{proof} 
The proof is similar to that of Lemma \ref{lem-attracting}. Here we can use the fact that the supporting hyperspheres 
are at uniformly bounded distances from the hypersphere containing $\Omega$.
The eigenvalue estimations are similar. 
\end{proof} 

\begin{lemma}\label{lem-locfin}
Let $\Gamma$ be a discrete group of projective automorphisms of 
a properly convex domain $V$ and a domain $\Omega \subset V$ of dimension $n-1$. 
Assume that $\Omega/\Gamma$ is compact. 
Suppose that $\Gamma$ satisfies the uniform middle eigenvalue condition. 
Let $P$ be a subspace of $\SI^{n}$ so that $P \cap \clo(\Omega) = \emp$. 
Then $\{g(P)\cap V|g \in \Gamma\}$ is a locally finite collection of closed sets in $V$. 
\end{lemma}
\begin{proof} 
Suppose not. Then there exists a sequence $x_{i}\in P$ and $g_{i}\in \Gamma$ so that 
$g_{i}(x_{i}) \in F$ for a compact set $F \subset V$. 
Then Lemma \ref{lem-attracting2} applies. $\{g_{i}^{-1}(F)\}$ accumulates only to $\Bd \Omega$. 
Since $x_{i} \in P \cap V$, this is a contradiction. 
\end{proof} 

\section{The characterization of quasi-lens p-R-end-neighborhoods.} \label{app-quasi-lens} 

We introduce the weak uniform eigenvalue condition. Then we model the quasi-lens p-R-end neighborhood
by a group property. Finally, we will prove the main result Proposition \ref{prop-quasilens2}. 

Let us give some definitions generalizing the conditions of the main part of the paper: 

A quasi-lens cone is a properly convex cone of form $p\ast S$ for a strictly convex open hypersurface $S$ 
so that $\partial (\{p\} \ast S -\{p\})= S$ and $p \in \clo(S) - S$
and the space of directions from $p$ to $S$ is a properly convex domain in $\SI^{n-1}_p$.

 An R-end $\tilde E$ is {\em lens-shaped} (resp. {\em totally geodesic cone-shaped},
{\em generalized lens-shaped}, {\em quasi-lens shaped}) 
 if it has a pseudo-end-neighborhood that is a lens-cone (resp. a cone over a totally-geodesic domain, 
 a concave pseudo-end-neighborhood, or a quasi-lens cone.)
 Here, we require that $\bGamma_{\tilde E}$ acts on the lens of the lens-cone.  
 
In Definition \ref{defn-umec}, we replace the condition by the follow: 
\begin{itemize} 
\item If $\lambda_{\bv_{\tilde E}}(g)$, $g \in \bGamma_{\tilde E}$ has the largest norm among 
eigenvalues, then it has to be of multiplicity $\geq 2$, 
\item the uniform middle eigenvalue condition for each hyperbolic $\Gamma_i$, i.e., the condition (ii).
\end{itemize}
Then we say that $\bGamma_{\tilde E}$ satisfies the {\em weakly uniform middle-eigenvalue conditions}. 

This is the last remaining case for the properly convex ends with weak uniform middle eigenvalue conditions. 
We will only prove for $\SI^n$. 

\begin{definition}\label{defn-quasilens}
Let $U$ be a totally geodesic lens cone p-end-neighborhood of a p-R-end 
in a subspace $\SI^{n-1}$ with vertex $\bv$. Let $G$ denote the p-end fundamental group
satisfying the weak uniform middle eigenvalue condition. 
\begin{itemize}
\item Let $D$ be an open totally geodesic $n-2$-dimensional domain so that $U = D \ast \bv$. 
\item Let $\SI^1 \subset \SI^{n}$ be a great circle meeting $\SI^{n-1}$ at $\bv$ transversally.
\item Extend $G$ to act on $\SI^1$ as a nondiagonalizable transformation fixing $\bv$. 
\item Let $\zeta$ be a projective automorphism 
acting on $U$ and $\SI^1$ so that $\zeta$ commutes with $G$ and restrict to a diagonalizable transformation on $\clo(D)$
and act as a nondiagonalizable transformation on $\SI^1$ fixing $\bv$ and with largest norm eigenvalue at $\bv$. 
\end{itemize}
Every element of $G$ and $ \zeta$ can be written as a matrix
\begin{equation}
\left( \begin{array}{c|c}
S(g) & 0 \\
\hline
0 &  \begin{array}{cc}
\lambda_{\bv}(g) & \lambda_{\bv}(g)v(g)\\
0 & \lambda_{\bv}(g) \end{array} 
\end{array}\right) \label{eqn-qj}
\end{equation} 
where $\bv =[0, \dots, 1]$. 
Note that $g \mapsto v(g) \in \bR$ is a well-defined map inducing a homomorphism 
\[  \langle G, \zeta \rangle \ra H_1( \langle G, \zeta \rangle) \ra \bR\] 
and since $v(g) = v(hgh^{-1})$ for any element $h$, we obtain
\begin{equation} \label{eqn:vgcwl}
|v(g)| \leq C \cwl(g) \hbox{ for a positive constant } C. 
\end{equation} 

We assume that $\zeta$ has the largest eigenvalue associated with $\SI^{1}$
and acts trivially on $D$. 
Again, we assume that $G$ has the largest norm eigenvalue and the smallest norm eigenvalue occur in $D$. 
Hence $\lambda_{v}(g)$ for $g \in G$ is not the eigenvalue with largest or smallest norms.





\begin{description}
\item[Positive translation condition] We choose an affine coordinate on a component $I$ of $\SI^1 -\{\bv, \bv_-\}$.
We assume that for each $g \in \langle G, \zeta \rangle$,
\begin{itemize}
\item if $\lambda_{\bv}(g) > \lambda_D(g)$ for the largest eigenvalue $\lambda_D$ associated with $\clo(D)$, 
then $v(g) > 0$ in equation \eqref{eqn-qj}, 
\item For $g$ satisfying $\lambda_{\bv}(g) > \lambda_D(g)$,  there exists a constant $c_{1}$ independent of $g$
\[\frac{v(g)}{\log \frac{\lambda_{\bv}(g)}{\lambda_D(g)}} > c_1 > 0. \]
\end{itemize} 
\end{description}
\end{definition}

Clearly, this type of construction can be done easily by choosing $G$ and $\zeta$ 
satisfying the above properties by essentially choosing $\zeta$ well. 
Also, $v$ induces a homomorphism
\[v: \bGamma_{\tilde E} \ra \bR \]
inducing $H^{1}(\bGamma_{\tilde E}) \ra \bR$.
Thus, $v$ is a cocycle. 

The converse to this construction is the following: 

%


\begin{proposition}\label{prop-quasilens1} 
Suppose that $\langle G, \zeta \rangle$ is admissible and satisfies the weak middle eigenvalue condition
and the positive translation condition. 
Then  the above $U$ is in the boundary of a properly convex p-end open neighborhood $V$ of $\bv$ 
and $\langle G, \zeta \rangle$ acts on $V$.
\end{proposition}
\begin{proof}
Let $I$ be the segment in $\SI^1$ bounded by $\bv$ and $\bv_-$. 
Take $D\ast I$ is a tube with vertices $\bv$ and $\bv_-$. 

Taking the interior of the convex hull of an orbit and $U$ will give us $V$. 

Let $x$ be an interior point of the tube. 
Given a sequence $g_i \in G$, then we will show that $g_i(x)$ 
accumulates to points uniformly bounded away from $\bv_{-}$
by the positive translation conditions as we can show by using estimates. 
Hence, the convex hull of the orbit is bounded away from $\bv_{-}$ and we have a properly convex 
convex hull. 

Suppose not.
Then there exists a sequence $g_{i}\in \langle G, \zeta \rangle$ with $\{g_{i}(x)\}$ accumulates to $\bv_{-}$. 
Given any sequence $g_i \in \langle G, \zeta \rangle$, we write as $g_i = \zeta^{j_i} g'_i$ for $g'_i \in G$. 
We write 
\begin{align} 
& x = [v], v = v_1 + v_2, [v_1] \in D, [v_2] \in I -\{\bv\} \subset \SI^1, \nonumber \\
& g_i(x) = [g_i(v_1) + g_i(v_2)].
\end{align} 
Since we can always extract a subsequence for any converging subsequence, we 
consider only three cases: 
\begin{itemize}
\item[(i)] $\frac{\lambda_{\bv}(g_i)}{\lambda_D(g_i)} \ra \infty$.
\item[(ii)] $\frac{1}{C} <  \frac{\lambda_{\bv}(g_i)}{\lambda_D(g_i)} < C$ for some $C > 1$.
\item[(iii)] $\frac{\lambda_{\bv}(g_i)}{\lambda_D(g_i)}    \ra 0. $
   \end{itemize} 
In case (i), If $\lambda_{\bv}(g_i)/\lambda_D(g_i) \ra \infty$, then $||g_i(v_1)||/||g_i(v_2)|| \ra 0$  
and $g_i(x)$ converges to the limit of $[g_i(v_2)]$, i.e., $\bv$, since 
$v(g_i) \ra \infty$. 

Suppose (ii).  
Then we multiply by $\zeta^{j_{i}}$ for uniformly bounded $|j_{i}|$ so that 
$\lambda_{\bv}(\zeta^{j_{i}} g_{i}) > \lambda_D(\zeta^{j_{i}}g_{i})$
but the ratio 
\[\left|\log\frac{\lambda_{\bv}(\zeta^{j_{i}} g_{i})}{\lambda_D(\zeta^{j_{i}}g_{i})}\right|\]
is uniformly bounded. 
Then $|\min\{0, v(\zeta^{{j_{i}}} g_i)\}| < C'$ for a constant by the positive translation condition. 
This also implies that $|\min\{0, v(g_{i})\}|$ is uniformly bounded as $|j_{i}|$ is uniformly bounded. 
This implies $g_i(x)$ lies in a $(\pi-\eps)$-$\bdd$-neighborhood of 
$\bv_{\tilde E}$ for a uniform constant $\eps$.

Suppose now (iii).
As above, for each $i$, we find a sufficiently large $J_{i} > 0$ so that
\[ \lambda_{\bv_{\tilde E}}(\zeta^{J_{i}} g_{i}) > \lambda_{D}(\zeta^{J_{i}} g_{i}). \]
and 
\[ \left|\log \frac{\lambda_{\bv_{\tilde E}}(\zeta^{J_{i}} g_{i})}{\lambda_{D}(\zeta^{J_{i}} g_{i})} \right|\]
is a uniformly bounded sequence.  Now, $J_{i} \ra +\infty$.

Let $h_{i} = \zeta^{J_{i}}g_{i}$. Then $v(h_{i}) > 0$. 
Since $v(g_{i}) = v(h_{i}) - J_{i}v(\zeta) $, 
\[|\min\{0, v(g_{i})\}| < C_{1} J_{i} + C_{2} \hbox{ for positive constants } C_{1}, C_{2}.\] 
Also, 
\[\left|\log \frac{\lambda_{D}(g_{i})}{\lambda_{\bv_{\tilde E}}(g_{i})} \right| \sim
J_{i} \left|\log \frac{\lambda_{D}(\zeta)}{\lambda_{\bv_{\tilde E}}(\zeta)} \right|\]
(Here, $\sim$ means that the ratio is uniformly bounded.)
Hence, 
\[ \frac{\lambda_{D}(g_{i})}{\lambda_{\bv_{\tilde E}}(g_{i})} \sim \exp C''J_{i} \hbox{ for } C''>0.\]
Therefore, 
\[\min\left\{0, \frac{\lambda_{\bv_{\tilde E}}(g_{i}) v(g_{i})}{\lambda_{D}(g_{i})} \right\}\sim \frac{C_{1} J_{i} + C_{2} }{\exp(C''J_{i})}. \]
This implies that \[||g_i(v_2)||/||g_i(v_1)|| \ra 0,\]  
and $g_i(x)$ converges to a point of $D$. 

We showed in all cases that the accumulation points of any orbit is outside a small ball at $\bv_{-}$. 
This contradicts our assumption that $\{g_{i}(x)\}$ accumulates to $\bv_{-}$.
Thus, these orbit points are inside the properly convex tube and outside a small ball at $\bv_{-}$. 
 The interior of the convex hull of the orbit of $x$ is a properly convex open domain as desired above. 
 (See the proof of Proposition \ref{III-prop-qjoin} of \cite{EDC3} uses a slightly different argument.)
\end{proof}

This generalizes the quasi-hyperbolic annulus discussed in \cite{cdcr2}. 
We give a more concise condition at the end of the subsection. 



Conversely, we obtain:

\begin{proposition} \label{prop-quasilens2} 
Let $\orb$ be a strongly tame properly convex real projective orbifold. 
Suppose that holonomy group of $\pi_{1}(\orb)$ is strongly irreducible. 
Let $\tilde E$ be a p-R-end with an admissible holonomy group satisfying the weak uniform middle eigenvalue conditions
but not the uniform middle eigenvalue condition.
Then $\tilde E$ has a quasi-lens type p-end-neighborhood. 
\end{proposition}
\begin{proof}
(A) If $\tilde E$ is not virtually factorable and hyperbolic,
then it satisfies the uniform middle eigenvalue condition by definition.  
We recall a part of the proof of Theorem \ref{thm-redtot}. 

Now assume that $\tilde E$ is virtually factorable. 
Let $U$ be a p-end-neighborhood of $\tilde E$ in $\tilde {\mathcal{O}}$.
By admissibility of $\tilde E$, we obtain 
$\clo(\tilde \Sigma_{\tilde E}) = K_{1} \ast \cdots \ast K_{l_{0}}$ 
where $\bGamma_{i} = \bGamma_{\tilde E}| K_{i}$ acts irreducibly on $K_{i}$. 
$\bGamma_{\tilde E}$ is virtually isomorphic to 
\[\bZ^{l_{0}-1} \times \bGamma_{1} \times \cdots \times \bGamma_{l_{0}}.\]
(Here $K_{i}$ can be a singleton and $\Gamma_{i}$ a trivial group. )
We obtain the projective subspaces $S_1,..., S_{l_0}$ in general position meeting only at the p-end vertex $\bv_{\tilde E}$
corresponding to the subspaces in $\SI^{n-1}_{\bv_{\tilde E}}$ containing $K_{1}, \dots, K_{l_{0}}$ respectively. 
Let $C_i$ denote the union of great segments from $\bv_{\tilde E}$ corresponding to $K_{i}$ for each $i$. 
The abelian virtual center isomorphic to $\bZ^{l_0-1}$ acts as the identity on $C_i$ in the projective space $\SI^n_{\bv_{\tilde E}}$. 
Let $g\in \bZ^{l_0-1}$. $g| C_i$ can have more than two eigenvalues or just single eigenvalue. 
In the second case $g|C_i$ could be represented by a matrix with eigenvalues all $1$  fixing $\bv_{\tilde E}$. 
\begin{itemize} 
\item[{\rm (a)}] $g|C_i$ fixes each point of a hyperspace $P_i \subset S_i$ not passing through $\bv_{\tilde E}$ 
and $g$ has a representation as a nontrivial scalar multiplication in the affine subspace $S_i - P_i$ of $S_i$. 
Since $g$ commutes with every element of $\bGamma_i$ acting on $C_i$, 
$\bGamma_i$ acts on $P_i$ as well.  We let $D'_i = C_i \cap P_i$.
\item[{\rm (b)}] $g|C_i$ is represented by a matrix with eigenvalues all $1$  fixing $\bv_{\tilde E}$
in the vector subspace corresponding to $C_{i}$. 
\end{itemize}
We denote $I_1:=\{ i| \exists g \in \bZ^{l_0-1}, g|C_i \ne \Idd\} $ and 
 \[I_2:= \{i| \forall g \in \bZ^{l_0-1}, g|C_i \hbox{ is a scalar times a unipotent element} \}.\]
 
 Let $D_i \subset \SI^{n-1}_{\bv}$ denote the convex compact domain 
 that is the space of great segments in $C_i$ from $\bv_{\tilde E}$ to 
 $\bv_{\tilde E -}$. Then \[\tilde \Sigma_{\tilde E}=D_1 \ast \cdots \ast D_{l_0}.\] 
 Also, $D'_i$ is projectively diffeomorphic to $D_i$ by projection for $i \in I_1$. 
 
 Suppose that hyperbolic $\bGamma_i$ acts on $C_i$. 
 Then it satisfies the uniform middle eigenvalue condition 
 by Definition \ref{defn-umec}.
 By Theorem \ref{thm-equ}, $\bGamma_i$ acts on a lens domain $D_i$. 
 For $g$ in the virtual center of $\bGamma_{\tilde E}$, $g$ acts on each great segment from $\bv_{\tilde E}$ through $D_{i}$. 
 If $i \in I_2$, then $g|C_i$ must be the identity; otherwise, we again obtain a violation of the proper convexity 
 considering $g^{j}(D_{i})$. 
 
 Suppose that $l_2$ is empty. Then $\bGamma_{\tilde E}$ acts on a totally geodesic 
 subspace that is the span of $D'_{1}\ast \cdots \ast D'_{l_{0}}$. 
 Proposition \ref{prop-decjoin}  and the weak middle eigenvalue condition imply that
 $\lambda_{1}(g) > \lambda_{\bv_{E}}(g)$ for each $g \in \bZ^{l_{0}-1} - \{\Idd\}$. 
 For any diverging sequence $g_{i} \in \bZ^{l_{0}-1}$, we can show 
 \[ \frac{\lambda_{1}(g_{i})}{\lambda_{\bv_{\tilde E}}(g_{i})} \ra \infty\] 
by  Proposition \ref{prop-decjoin}. 
Since each factor groups $\bGamma_{i}$ satisfies the uniform middle eigenvalue conditions, 
 for any diverging sequence $g_{i} \in \bGamma_{\tilde E}$, it follows that 
 \[ \frac{\lambda_{1}(g_{i})}{\lambda_{\bv_{\tilde E}}(g_{i})} \ra \infty.\] 
  Since this condition is all we need to 
 follow the results of Section \ref{subsub:umecorbit}, $\tilde E$ is lens-shaped totally geodesic R-end
By Theorem \ref{thm-secondmain}, $\bGamma_{\tilde E}$ satisfies the 
uniform middle eigenvalue condition, contradicting the assumption. 
Therefore, we conclude $I_{2} \ne \emp$. 


(B) For $i \in I_2$, $\bGamma_i$ is not hyperbolic as above and hence must be a trivial group 
and $C_i$ is a segment. 
Consider $C_{I_2}:= \ast_{i \in I_2} C_i$. Then $g|C_i$ for $g \in \bZ^{l_0-1}$ 
has only eigenvalue $\lambda_{\bv_{\tilde E}}$ associated with it so that we don't have 
two distinct eigenvalues for $C_{i}$. 
Since $\dim C_{i} = 1$, $g|C_i$ is a translation in an affine coordinate system. 
Therefore, $\bZ^{l_0-1}$ acts trivially on the space of great segments in $C_{I_2}$. 
Thus, $\dim C_{I_2} = 1$ since otherwise we cannot obtain the compact quotient 
$\tilde \Sigma_{\tilde E}/\bGamma_{\tilde E}$. 


Let $l_2 =\{l_{0}\}$. 
Therefore, we obtain $D= \ast_{i =1}^{l_{0}-1}D_i$ is a totally geodesic plane disjoint from $\bv_{\tilde E}$.  
Let $\bv_{\tilde E} = [0, \dots, 0, 1]\in \SI^n$. 
We write $g \in \bGamma_{\tilde E}$ in coordinates as: 
\[ g = \left( \begin{array}{c|c} 
S_g & 0 \\
\hline
0 & \begin{array}{cc}
\lambda_{\bv}(g) & \lambda_{\bv}(g)v(g)\\
0 & \lambda_{\bv}(g) \end{array} 
\end{array}\right)\]
where $S_g$ is a $n-1\times n-1$-matrix representing coordinates $\{1, \dots, n-1\}$. 
Then $V: g \in \bGamma_{\tilde E} \ra v(g) \in \bR$ is a linear function. 

The proper convexity of $\torb$ implies that $v(g) \geq 0$ if 
$\lambda_{\bv}(g_i)/\lambda_D(g_i) > 1$: otherwise, we obtain a great segment in 
$\SI^1$ by a limit of $g_i(s)$ for a segment $s \subset U$ from $\bv$.
This is a contradiction since a great segment is not in a properly convex set $\clo(U)$. 


Suppose that we have an element $g$ with $v(g) = 0$ and 
$\lambda_{\bv}(g)/\lambda_D(g) > 1$.
Given a segment $s \subset U$ with an endpoint $\bv$, $\{g^{i}(s)\}$ as $i \ra \infty$ converges to a segment $s_\infty$ in
$\SI^1 \cap \clo(\torb)$. 
If $v(g)> 0$ for any $g \in \bGamma_{\tilde E}$, 
we can apply $g^i(s)$ to obtain a great segment in the limit for $i \ra \pm \infty$, 
a contradiction as above.
Therefore, $v(g) = 0$ for all $g \in \bGamma_{\tilde E}$. 


Then we can find 
a sequence $\{\eta_i\}$ of elements in the virtual center so that $\lambda_{\bv}(\eta_i)/\lambda_D(\eta_i) \ra \infty$ and 
$\eta_i| D$ is uniformly bounded since $\bZ^{l_{0}-1}$ is cocompact in $\bR^{l_{0}-1}$.  
We have $v(\eta_i) =0$ for all $i$ by the above paragraph. 
Then we can apply Propositions \ref{prop-joinred} and \ref{prop-decjoin}  
to obtain a contradiction to the strong irreducibility of $\bGamma$. 
Therefore, we conclude that $v(g) > 0$ provided $\lambda_{\bv}(g)/\lambda_D(g) > 1$.


(C) Since $\Sigma_{\tilde E}$ is a join with a factor equal to a vertex corresponding to $\SI^{1}$, 
we can choose a generator $\zeta$ of the virtual center so that $\lambda_{\bv}(\zeta)> \lambda_D(\zeta)$. 
$\langle \zeta \rangle$ is a factor of the virtual center of $\bGamma_{\tilde E}$.
Let $G$ be the product of other virtual factors of $\bGamma_{\tilde E}$. 

The part (B) shows $v(\zeta) > 0$. 
Every element $g$  with $\lambda_{\bv}(g) > \lambda_D(g)$
is of form $\zeta^i g'$ for $\lambda_{\bv}(g')/\lambda_D(g')$ uniformly bounded 
above.  
For such a set $A$ of $g'$, we have $v(g')$ are uniformly bounded below
since otherwise the orbit of a point under $A$ has a subsequence converging to $v_{\tilde E-}$.  
We can verify the uniform positive translation condition. 
By Proposition \ref{prop-quasilens1}, we obtain a quasi-lens p-end-neighborhood.  
\end{proof}

\begin{remark} 
To explain the positive translation condition more, $\log \lambda_{\bv_{\tilde E}}(g)$ and $v(g)$ give
us homomorphisms  $\log \lambda_{\bv}, V : H_1(\bGamma_{\tilde E}) \ra \bR$. 
Restricted to $\bZ^{l_0-1} \subset H_1(\bGamma_{\tilde E}) $, 
we obtain $\log \lambda_i: \bZ^{l_0-1} \ra \bR$ given by taking the log of the eigenvalues restricted to 
$D_i$ above. 
The condition restricts to the uniform positivity condition of $V$ on 
the cone $C$ in $\bZ^{l_0-1}$ defined by \[\log \lambda_{\bv_{\tilde E}}([g])  > \log \lambda_i([g]), i=1, \dots, l_0-1.\] 
That is, $V$ is positive on a compact $\phi^{-1}(1) \cap C$ for a linear functional $\phi$. 
\end{remark}


\section{An extension of Koszul's openness} \label{app-Koszul}

Here, we state and prove a minor modification of Koszul's openness result. 
This is of course trivial and known to many people already; however, 
we give a proof. 

\begin{proposition}[Koszul] \label{prop-koszul} 
Let $M$ be a properly convex real projective compact $n$-orbifold with strictly convex boundary. 
Let $h:\pi_{1}(M) \ra \PGL(n+1, \bR)$ {\rm (}resp.  $\ra \SLnp${\rm )} denote the holonomy homomorphism
acting on a properly convex domain $\Omega_{h}$ in $\bR P^{n}$ {\rm (}resp.  in $\SI^{n}${\rm )}.  
Assume $M$ is projectively diffeomorphic to $\Omega_{h}/h(\pi_{1}(M))$. 
Then there exists a neighborhood $U$ of $h$ in $\Hom(\pi_1(M), \PGL(n+1, \bR))$ 
{\rm (}resp. $\Hom(\pi_1(M), \SLnp)${\rm )} 
so that every $h' \in U$ acts on a properly convex domain $\Omega_{h'}$
so that $\Omega_{h'}/h'(\pi_{1}(M))$ is 
a compact properly convex real projective $n$-orbifold 
$\Omega_{h'}/h'(\pi_{1}(M))$ with strictly convex boundary. 
Also,  $\Omega_{h'}/h'(\pi_{1}(M))$ is diffeomorphic to $M$. 
\end{proposition}
\begin{proof} We prove for $\SI^{n}$.  
Let $\Omega_{h}$ be a properly convex domain covering $M$. 
We may modify $M$ by pushing $\partial M$ inward. 

Let $\Omega'_{h}$ be the inverse image of $M'$ in $M$. 
Then $M'$ and  $\Omega'_{h}$ are properly convex by Lemma \ref{I-lem-pushing} of \cite{EDC1}. 

The linear cone $C(\Omega^{o}_{h})\subset \bR^{n+1} = \Pi^{-1}(\Omega^{o}_{h})$ over $\Omega^{o}_{h}$ 
has a smooth strictly convex hessian function $V$ 
by Vey \cite{Vey} or Vinberg \cite{vin63}. Let $C(\Omega'_{h})$ denote the linear cone over $\Omega'_{h}$.
We extend the group $\mu(\pi_1(M))$ by adding
 a transformation $\gamma: \vec{v} \mapsto 2\vec{v}$ to $C(\Omega^{o}_{h})$. 
For the fundamental domain $F'$ of $C(\Omega'_{h})$ under this group, 
the hessian matrix of $V$ restricted to $F \cap C(\Omega'_{h})$ has a lower bound.
Also, the boundary $\partial C(\Omega'_{h})$ is strictly convex in any affine coordinates 
in any transversal subspace to the radial directions at any point.

Let $N'$ be a compact orbifold $C(\Omega'_{h})/\langle \mu(\pi_1(\tilde E)), \gamma \rangle$ with a flat affine structure. 
Note that $S_t$, $t \in \bR_+$, becomes an action of a circle on $M$.
The change of representation $h$ to $n': \pi_1(M) \ra \SLnp$ 
is realized by a change of holonomy representations of $M$ and hence by
a change of affine connections on $C(\Omega'_{h})$. Since $S_t$ commutes with the images of $h$ and $h'$, 
$S_t$ still gives us a circle action on $N'$ with a different affine connection. 
We may assume without loss of generality 
that the circle action is fixed and $N'$ is invariant under this action.

Thus, $N'$ is a union of $B_{1}, \dots, B_{m_{0}}$ that are $n$-ball times circles foliated by circles that are flow arcs of $S_{t}$.
We can change the affine structure on $N'$ to a one with 
the holonomy group $\langle h'(\pi_1(\tilde E)), \gamma\rangle$ by 
by local regluing $B_{1}, \dots, B_{m_{0}}$ as in \cite{Choi2004}. 
We assume that $S_{t}$ still gives us a circle affine action since $\gamma$ is not changed. 
We may assume that $N'$ and $\partial N'$ are foliated by circles that are flow curves of the circle action. 
The change corresponds to a sufficiently small $C^{r}$-change in the affine connection for $r \geq 2$
as we can see from \cite{Choi2004}. 
Now, the strict positivity of 
the hessian of $V$ in the fundamental domain, and the boundary convexity are preserved. 
Let $C(\Omega''_{h})$ denote the universal cover of $N'$ with the new affine connection. 
Thus, $C(\Omega''_{h})$ is also a properly convex affine cone by Koszul's work \cite{Kos}. 
Also, it is a cone over a properly convex domain $\Omega''_{h}$ in $\SI^{n}$.

\end{proof} 

We denote by $\PGL(n+1, \bR)_{v}$ the subgroup of $\PGL(n+1, \bR)$ fixing a point $v$. 

\begin{proposition}\label{prop-lensP}
Let $B$ be a strictly convex hypersurface bounding a properly convex domain in 
a tube domain $T$. Let $v, v_{-}$ be the vertices of $T$.
$B$ meets each radial ray in $T$ from $v$ transversally. 
Let $T$ be a tube domain over a properly convex domain $\Omega \subset \bR P^{n}$ {\rm(} resp. $\SI^{n-1}${\rm)}.
Assume that a projective group $\Gamma$ acts on $\Omega$ properly discontinuously and cocompactly.
Then there exists a neighborhood of $\Idd$ in 
$\Hom(\Gamma, \PGL(n+1, \bR)_{v})$ 
{\rm (}resp. $\Hom(\pi_1(M), \SLnp_{v})${\rm )} 
where every element $h$ acts on 
a strictly convex hypersurface $B_{h}$ in a tube domain $T_{h}$ meeting 
each radial ray at a unique point and bounding a properly convex domain in $T_{h}$. 
\end{proposition}
\begin{proof}
For sufficiently small neighborhood $V$ of $h$ in $\Hom(\Gamma, \PGL(n+1, \bR)_{v})$,
$h(\Gamma)$, $h\in V$ acts on a properly convex domain $\Omega_{h}$ properly discontinuously and cocompactly
by Koszul \cite{Kos}. 
Let $T_{h}$ denote the tube over $\Omega_{h}$. 
Since $B/\Gamma$ is a compact orbifold, 
we choose $V' \subset V$ so that for the projective connections on a compact neighborhood of $B/\Gamma$ 
corresponding to elements of $V'$, $B/\Gamma$ is still strictly convex
and transversal to radial lines. 
For each $h\in V'$, we obtain an immersion to a strictly convex domain $\iota_{h}: B \ra T_{h}$
transversal to radial lines. 
Let $p: T_{h} \ra \Omega_{h}$ denote the projection with fibers equal to the radial lines. 
Since $p\circ \iota_{h}$ is proper immersion to $\Omega_{h}$, the result follows. 
\end{proof}






\begin{thebibliography}{00}

\bibitem{Ballas2014} { S. Ballas}, 
\newblock{`Finite volume properly convex deformations of the figure-eight knot'}, 
 arXiv:1403.3314. 
 
 \bibitem{Ballas2012} { S. Ballas}, 
\newblock{`Deformations of non-compact, projective manifolds'}, arXiv:1210.8419.


\bibitem{Ben1} {   Y. Benoist},
 \newblock{`Convexes divisibles. I',}
 \newblock{{\em Algebraic groups and arithmetic}},
 ( Tata Inst. Fund. Res., Mumbai, 2004) {339--374},
 
\bibitem{Ben2} {   Y. Benoist},
\newblock{`Convexes divisibles. II'},
 \newblock{\em Duke Math. J.}, {120} (2003), 97--120.

\bibitem{Ben3} {    Y. Benoist},
\newblock{`Convexes divisibles. III'},
\newblock{{\em Ann. Sci. Ecole Norm. Sup.} (4) 38 (2005), no. 5, 793--832. }

\bibitem{Ben4} {    Y. Benoist},
\newblock{`Convexes divisibles IV : Structure du bord en dimension 3'}, 
\newblock{{\em Invent. math.} 164 (2006), 249--278.}

\bibitem{Ben5} {    Y. Benoist},
 \newblock{`Automorphismes des c\^ones convexes'},
\newblock{\em Invent. Math.}, {141} (2000), 149--193.

\bibitem{Benasym} {    Y. Benoist}, 
\newblock{`Propri\'et\'es asymptotiques des groupes lin\'eaires',} 
\newblock{ {\em Geom. Funct. Anal.} 7 (1997), no. 1, 1--47.}






  \bibitem{Choi2004}
   {    S. Choi},
    \newblock{`Geometric structures on orbifolds and holonomy representations'},
    \newblock{{\em Geom. Dedicata} 104 (2004), 161--199.}

 \bibitem{psconv} {   S.~Choi}, 
\newblock{`The convex and concave decomposition of manifolds with real projective structures'}, 
\newblock{{\em M\'emoires SMF}, No. 78, 1999, 102 pp.}

 





\bibitem{cdcr1}
{    S.~Choi}, 
\newblock{`Convex decompositions of real projective surfaces {{\rm {I}:}}
  $\pi$-annuli and convexity'},
\newblock {\em J. Differential Geom.} 40 (1994), 165--208.

\bibitem{cdcr2}
{    S.~Choi},
\newblock{`Convex decompositions of real projective surfaces {{\rm {II}:}}
  {A}dmissible decompositions',}
\newblock {\em J. Differential Geom.}, 40 (1994), 239--283.


\bibitem{Cbook} 
{    S.~Choi}, 
\newblock{{\em Geometric structures on 2-orbifolds\,{\rm :} exploration of discrete symmetry'}}, 
\newblock{MSJ Memoirs, Vol. 27. 171pp + xii, 2012}

\bibitem{endclass} 
{   S.~Choi},
\newblock{ `The classification of radial ends of convex real projective orbifolds',}
\newblock{ arXiv:1304.1605} 



\bibitem{EDC1} 
{    S.~Choi}, 
\newblock{`The classification of ends of properly convex real projective orbifolds I: survey'},
\newblock{preprint}  

\bibitem{EDC3} 
{    S.~Choi}, 
\newblock{`The classification of ends of properly convex real projective orbifolds III: nonproperly convex, convex ends',} 
\newblock{in preparation} 




\bibitem{conv} {    S. Choi}, 
\newblock{`The convex real projective manifolds and orbifolds with radial or totally geodesic ends: the closedness and openness of deformations',} 
\newblock{arXiv:1011.1060}

\bibitem{conv1} 
{    S. ~Choi}, 
\newblock{{`The deformation spaces of convex real projective orbifolds with
radial or totally geodesic ends I{\rm :} general openness'}}, 
\newblock{ in preparation } 

\bibitem{conv2} 
{    S. ~Choi}, 
\newblock{{`The deformation spaces of convex real projective orbifolds with
radial or totally geodesic ends II{\rm :} relative hyperbolicity'}}, 
\newblock{ in preparation} 

\bibitem{conv3} 
{    S. ~Choi}, 
\newblock{{`The deformation spaces of convex real projective orbifolds with
radial or totally geodesic ends III{\rm :} openness and closedness}'}, 
\newblock{ in preparation } 



 \bibitem{CG}
   {    S. Choi \and W.M. Goldman}, 
    \newblock{`The deformation spaces of convex $\mathbb{RP}^2$-structures on 2-orbifolds',}
    \newblock{{\em Amer. J. Math.} 127 (2005), 1019--1102.}


\bibitem{afftame} {    S.~Choi \and W. M. Goldman}, 
\newblock{`Topological tameness of Margulis spacetimes'}, 
\newblock{arXiv 1204.5308} 



    
    \bibitem{CLT2} 
{       D. Cooper, D. Long, \and S. Tillmann,} 
    \newblock{`On convex projective manifolds and cusps',} 
    \newblock {\em Adv. Math.} 277 (2015), 181--251.











\bibitem{wmgnote} {    W. Goldman}, 
\newblock{`Projective geometry on manifolds',} 
\newblock{\em Lecture notes available from the author.}


\bibitem{GLM} {    W, Goldman, F. Labourie, \and G. Margulis},  
\newblock{`Proper affine actions and geodesic flows of
hyperbolic surfaces'},
\newblock{{\em Annals of Mathematics} 170 (2009), 1051--1083. }




\bibitem{Guichard} {    O. Guichard}, 
\newblock{`Sur la r\'egularit\'e H\"older des convexes divisibles'}, 
\newblock{{\em Erg. Th. \& Dynam. Sys.} 25 (2005), 1857--1880.}


\bibitem{GW} {    O. Guichard \and A. Wienhard}, 
\newblock{`Anosov representations: domains of discontinuity and applications'}, 
\newblock{{\em Invent. Math.}} 190 (2012), no. 2, 357--438. 





    
    

\bibitem{Kobpaper} {    S. Kobayashi},
\newblock `Projectively invariant distances for affine and projective structures', 
\newblock {\em Differential geometry} (Warsaw, 1979) 
(Banach Center Publ., 12, PWN, Warsaw, 1984) 127--152.





\bibitem{Kos} {    J. Koszul}, 
\newblock{`Deformations de connexions localement plates', }
\newblock{{\em Ann. Inst. fourier (Grenoble)} 18 fasc. 1 (1968), 103--114. }

\bibitem{Lab} {    F. Labourie}, 
\newblock `Flat projective structures on surfaces and cubic holomorphic differentials', 
\newblock{\em Pure and applied mathematics quaterly} 3 no. 4 (2007), 1057--1099.


\bibitem{Leitner1} {    A. Leitner}, 
\newblock{`Limits under conjugacy of the diagonal subgroup in $\SL_{3}(\bR)$',}
\newblock{ arXiv.math.GT/2014.4534.} 


\bibitem{Leitner2} {   A. Leitner}, 
\newblock{`Limits under conjugacy of the diagonal subgroup in $\SL_{n}(\bR)$',}
\newblock{ arXiv.math.GT/2014.5523.} 


    \bibitem{Marquis}
{      L. Marquis},
    \newblock `Espace des modules de certains poly\`edres projectifs miroirs',
    \newblock{\em Geom. Dedicata} 147 (2010), 47--86.
    
    \bibitem{Mess}
 {      G. Mess}, 
\newblock{`Lorentz spacetimes of  curvature',} 
\newblock{{\em Geom. Dedicata} 126 (2007), 3--45. }



\bibitem{Moore} 
{    C. Moore}, 
\newblock `Distal affine transformation groups', 
\newblock{\em Amer. J. Math.} 90 (1968) 733--751.





\bibitem{Thnote} {    W. Thurston}, 
\newblock{\em Geometry and topology of $3$-manifolds,} 
\newblock{available at \url{http://library.msri.org/books/gt3m/}.}

\bibitem{Thbook} {    W. Thurston}, 
\newblock{\em Three-dimensional geometry and topology,} 
\newblock{(Princeton University Press, Princeton NJ, 1997).}



 \bibitem{Vey68} {   J. Vey},  
  \newblock `Une notion d'hyperbolicit\'e sur les vari\'et\'es localement plates', 
  \newblock {\em C.R. Acad. Sc. Paris}, 266(1968), 622--624.

 \bibitem{Vey} {    J. Vey}, 
  \newblock `Sur les automorphismes affines des ouverts convexes saillants',
  \newblock {\em Ann. Scuola Norm. Sup. Pisa} (3) 24(1970), 641--665.




       
\bibitem{vin63}
{     \`{E}.B. Vinberg},
\newblock `Homogeneous convex cones',
{\em Trans. Moscow Math. Soc.} 12 (1963), 340--363.





\end{thebibliography}
\end{document}